\documentclass[12pt]{amsart}

\usepackage{tikz}
\usepackage{pgfplots}
\usepackage{stmaryrd}
\usepackage{amssymb}
\usepackage{color}
\usepackage{graphicx}
\usepackage[vmargin=2.5cm, hmargin=2.5cm]{geometry}
\usepackage{enumerate}
\usepackage[utf8]{inputenc}

\author{Abdallah Assi} 
\address{Universit\'e d'Angers, Math\'ematiques,
49045 Angers ceded 01, France}
\email{assi@univ-angers.fr} 

\thanks{The first author is partially supported by the project GDR CNRS 2945 and a GENIL-SSV 2014 grant}

\author{Pedro A. Garc\'{\i}a-S\'anchez}
\address{Departamento de \'Algebra and CITIC-UGR, Universidad de Granada, E-18071 Granada, Espa\~na}
\email{pedro@ugr.es} 

\thanks{The second author is supported by the projects MTM2010-15595, FQM-343,  FQM-5849, G\'eanpyl (FR n\textordmasculine 2963 du CNRS), and FEDER funds}

\title{Numerical semigroups and applications}

\newtheorem{teorema}{Theorem}

\newtheorem{proposicion}[teorema]{Proposition}
\newtheorem{lema}[teorema]{Lemma}

\newtheorem{corolario}[teorema]{Corollary}

\theoremstyle{definition}
\newtheorem{gap}[teorema]{\texttt{GAP} example}
\newtheorem{conjetura}[teorema]{Conjecture}

\theoremstyle{remark}

\newtheorem{nota}[teorema]{Remark}
\newtheorem{example}[teorema]{Example}
\newtheorem{exemples}[teorema]{Examples}

\newcommand{\KK}{{\mathbb K}}
\newcommand{\NN}{{\mathbb N}}
\newcommand{\PP}{{\mathbf P}}
\newcommand{\ZZ}{{\mathbb Z}}

\thanks{The authors would like to thank M. Delgado, D. Llena and V. Micale for their comments}
\begin{document}
\maketitle

\tableofcontents

\section*{Introduction}

Numerical semigroups arise in a natural way in the study of nonnegative integer solutions to Diophantine equations of the form $a_1x_1+\cdots+a_nx_n=b$, where $a_1,\ldots,a_n,b\in \mathbb N$ (here $\mathbb N$ denotes the set of nonnegative integers; we can reduce to the case $\gcd(a_1,\ldots,a_n)=1$). Frobenius in his lectures asked what is the largest integer $b$ such that this equation has no solution over the nonnegative integers, for the case $n=2$. Sylvester and others solved this problem, and since then this has been known as the Frobenius problem (see \cite{alfonsin} for an extensive exposure of this and related problems).

Other focus of interest comes from  commutative algebra and algebraic geometry. Let $\KK$ be a field and let ${\bf A}=\KK[t^{a_1},\ldots,t^{a_n}]$ be the $\KK$-algebra of polynomials in $t^{a_1},\ldots,t^{a_n}$. The ring ${\bf A}$ is the   coordinate ring of the curve parametrized by $t^{a_1},\ldots,t^{a_n}$, and information from ${\bf A}$  can be derived from the properties of the numerical semigroup generated by the exponents $a_1,\ldots, a_n$. Thus in many cases the names of invariants in numerical semigroups are inherited from Algebraic Geometry. Along this line Bertin and Carbonne (\cite{bertin}), Delorme (\cite{delorme}), Watanabe (\cite{watanabe}) and others found several families of numerical semigroups yielding complete intersections and thus Gorenstein semigroup rings. In the monograph \cite{barucci} one can find a good dictionary Algebraic Theory-Numerical semigroups.

 Numerical semigroups are also useful in the study of singularities of plane algebraic curves.  Let $\KK$ be an algebraically closed field of characteristic zero and let $f(x,y)$ be an element of $\KK[[x,y]]$. Given another element $g\in \KK[[x,y]]$, we define
 the local intersection multiplicity of $f$ with $g$ to be the rank of the $\KK$-vector space $\KK[[x,y]]/(f,g)$. When $g$ runs over the set of elements of $\KK[[x,y]]\setminus (f)$, these numbers define a semigroup. If furthermore $f$ is irreducible, then this semigroup is a numerical semigroup. This leads to a classification of irreducible formal power series in terms of their associated numerical semigroups. This classification can  be generalized to polynomials with one place at infinity. The arithmetic properties of numerical semigroups have been in this case the main tool  in the proof of Abhyankar-Moh lemma which says that a coordinate has a unique embedding in the plane.


Recently, due to use of algebraic codes and Weierstrass numerical semigroups, some applications to coding theory and cryptography have arise. The idea is finding properties of the codes in terms of the associated numerical semigroup. See for instance \cite{ieee} and the references therein. 

Another focus of recent interest has been the study of factorizations in monoids. If we consider again the equation $a_1 x_1+\cdots + a_nx_n=b$, then we can think of the set of nonnegative integer solutions as the set of factorizations of $b$ in terms of $a_1,\ldots,a_n$. It can be easily shown that no numerical semigroup other than $\mathbb N$ is half-factorial, or in other words,  there are elements with factorizations of different lengths. Several invariants measure how far are monoids from being half-factorial, and how wild are the sets of factorizations. For numerical semigroups several algorithms have been developed in the last decade, and this is why studying these invariants over numerical semigroups offer a good chance to test conjectures and obtain families of examples.

The aim of this manuscript is to give some basic notions related to numerical semigroups, and from these on the one hand describe a classical application to the study of singularities of plane algebraic curves, and on the other, show how numerical semigroups can be used to obtain handy examples of nonunique factorization invariants. 

\section{Notable elements}

Most of the results appearing in this section are taken from \cite[Chapter 1]{ns}.

Let $S$ be a subset of ${\mathbb N}$. The set $S$ is a submonoid of  ${\mathbb N}$ if the following holds:
\begin{enumerate}[(i)]
\item $0\in S$,

\item  If $a,b\in S$ then $a+b\in S$.
\end{enumerate}

Clearly, $\lbrace 0\rbrace$ and ${\mathbb N}$ are submonoids of ${\mathbb N}$. Also, if $S$ contains a nonzero element $a$, then $da\in S$ for all $d\in{\mathbb N}$, and in particular, $S$ is an infinite set. 

Let $S$ be a submonoid of ${\mathbb N}$ and let $G$ be the subgroup of ${\mathbb Z}$ generated by $S$ (that is, $G=\lbrace \sum_{i=1}^s\lambda_ia_i \mid  s\in{\mathbb N}, \lambda_i\in{\mathbb Z}, a_i\in S\rbrace$). If $1\in G$, then we say that $S$ is a \emph{numerical semigroup}.

We set $\mathrm G(S)=\NN\setminus S$ and we call it the set of \emph{gaps} of $S$. We denote by $\mathrm g(S)$ the cardinality of $\mathrm G(S)$, and we call $\mathrm g(S)$ the \emph{genus} of $S$. Next proposition in particular shows that the genus of any numerical semigroup is a nonnegative integer.

\begin{proposicion} 
Let $S$ be a submonoid of ${\mathbb N}$. Then $S$ is a numerical semigroup if and only if $\NN\setminus S$ is a finite set.
\end{proposicion}

\begin{proof} 
Let $S$ be a numerical semigroup and let $G$ be the group generated by $S$ in $\ZZ$. Since $1\in G$, we can find an expression  $1=\sum_{i=1}^k\lambda_ia_i$ for some $\lambda_i\in \mathbb Z$ and $a_i\in S$. Assume, without loss of generality, that $\lambda_1,\dots,\lambda_l<0$ (respectively  $\lambda_{l+1},\dots,\lambda_k>0$). If $s=\sum_{i=l}^l-\lambda_ia_i$, then $s\in S$ and  $\sum_{i=1+l}^k\lambda_ia_i=1+s\in S$. We claim that for all $n\geq (s-1)(s+1)$, $n\in S$.  Let  $n\geq (s-1)(s+1)$ and write $n=qs+r$, $0\leq r<s$. Since $n=qs+r\geq (s-1)s+(s-1)$, we have $q\geq s-1\geq r$, whence $n=qs+r=(rs+r)+(q-r)s=r(s+1)+(q-r)s\in S$.

Conversely, assume that $\mathbb N\setminus S$ has finitely many elements. Then there exist $s\in S$ such that $s+1\in S$. Hence $1=s+1-s\in G$.
\end{proof}

The idea of focusing on numerical semigroups instead of submonoids of $\mathbb N$ in general is the following. 

\begin{proposicion}
Let $S$ be a submonoid of ${\mathbb N}$. Then $S$ is isomorphic to a numerical semigroup.
\end{proposicion}
\begin{proof} Let $d$ be $\gcd(S)$, that is, $d$ is the generator of the group generated by $S$ in ${\mathbb Z}$. Let $S_1=\lbrace {s/ d}\mid s\in S\rbrace$ is a numerical semigroup. The map $\phi: S\to S_1$, $\phi(s)=s/ d$ is a homomorphism of monoids that is clearly bijective.
\end{proof}

Even though any numerical semigroup has infinitely many elements, it can be described by means of finitely many of them. The rest can be obtained as linear combinations with nonnegative integer coefficients from these finitely many. 

Let $S$ be a numerical semigroup and let $A\subseteq S$. We say that $S$ is generated by $A$ and we write $S=\langle A\rangle$ if for all $s\in S$, there exist $a_1,\dots,a_r\in A$ and $\lambda_1,\dots,\lambda_r\in\NN$ such that $a=\sum_{i=1}^r\lambda_ia_i$. Every numerical semigroup $S$ is finitely generated, that is, $S=\langle A\rangle$ with $A\subseteq S$ and $A$ is a finite set.

Let $S^*=S\setminus\{0\}$. The smallest nonzero element of $S$ is called the \emph{multiplicity} of $S$, $\mathrm m(S)=\min S^*$.

\begin{proposicion}\label{ns-fg}
Every numerical semigroup is finitely generated.
\end{proposicion}
\begin{proof}
Let $A$ be a system of generators of $S$ ($S$ itself is a system of generators). Let $m$ be the multiplicity of $S$. Clearly $m\in A$. Assume that $a<a'$ are two elements in $A$ such that $a\equiv a'\bmod m$. Then $a'=km+a$ for some positive integer $k$. So we can remove $a'$ from $A$ and we still have a generating system for $S$. Observe that at the end of this process we have at most one element in $A$ in each congruence class modulo $m$, and we conclude that we can choose $A$ to have finitely many elements.
\end{proof}

The underlying idea in the last proof motivates the following definition.

Let $n\in S^*$. We define the \emph{Ap\'ery set} of $S$ with respect to $n$, denoted $\mathrm{Ap}(S,n)$, to be the set 
\[ \mathrm{Ap}(S,n)=\lbrace s\in S\mid s-n\notin S\rbrace.\]
This is why some authors call $\{n\}\cup(\mathrm{Ap}(S,n)\setminus\{0\})$ a standard basis of $S$, when $n$ is chosen to be the least positive integer in $S$.

As we see next, $\mathrm{Ap}(S,n)$ has precisely $n$ elements.

\begin{lema}\label{char-ap}
Let $S$ be a numerical semigroup and let $n\in S^*$. For all $i\in \{1,\ldots, n\}$, let $w(i)$ be the smallest element of $S$ such that $w(i)\equiv i \bmod n$. Then
\[\mathrm{Ap}(S,n)=\lbrace 0,w(1),\dots,w(n-1)\rbrace.\]
\end{lema}
\begin{proof} 
%
Let $0\leq i\leq n-1$. By definition, $w(i)\in S$ and clearly $w(i)-n\equiv i$ mod$(n)$, hence $w(i)-n\notin S$, in particular $w(i)\in \mathrm{Ap}(S,n)$. This proves one inclusion. Observe that there are no elements $a,b\in \mathrm{Ap}(S,n)$ such that $a\equiv b\bmod n$. Hence we get an equality, because we are ranging all possible congruence classes modulo $n$.
\end{proof}

Next we give an example using the \texttt{numericalsgps} \texttt{GAP} package (\cite{numericalsgps} and \cite{GAP}, respectively). We will do this several times along the manuscript, since it is also our intention to show how calculations can easily be accomplished using this package.

\begin{gap}
Let us start defining a numerical semigroup.
\begin{verbatim}
gap> s:=NumericalSemigroup(5,9,21);;
gap> SmallElementsOfNumericalSemigroup(s);
[ 0, 5, 9, 10, 14, 15, 18, 19, 20, 21, 23 ]
\end{verbatim}
This means that our semigroup is $\{0, 5, 9, 10, 14, 15, 18, 19, 20, 21, 23,\to\}$, where the arrow means that every integer larger than 23 is in the set. If we take a nonzero element $n$ in the semigroup, its Ap\'ery set has exactly $n$ elements.
\begin{verbatim}
gap> AperyListOfNumericalSemigroupWRTElement(s,5);
[ 0, 21, 27, 18, 9 ]
\end{verbatim}
We can define the Ap\'ery set for other integers as well, but the above feature no longer holds.
\begin{verbatim}
gap> AperyListOfNumericalSemigroupWRTInteger(s,6);
[ 0, 5, 9, 10, 14, 18, 19, 23, 28 ]
\end{verbatim}
\end{gap}

Apéry sets are one of the most important tools when dealing with numerical semigroups. Next we see that they can be used to represent elements in a numerical semigroup in a unique way (actually the proof extends easily to any integer).

\begin{proposicion}\label{write-ap}
Let $S$ be a numerical semigroup and let $n\in S^*$. For all $s\in S$, there exists a unique $(k,w)\in \NN\times \mathrm{Ap}(S,n)$ such that $s=kn+w$.
\end{proposicion}
\begin{proof} 
Let $s\in S$. If $s\in \mathrm{Ap}(S,n)$, then we set $k=0, w=s$. If $s\notin\mathrm{Ap}(S,n)$, then $s_1=s-n\in S$. We restart with $s_1$. Clearly there exists $k$ such that $s_k=s-kn\in \mathrm{Ap}(S,n)$.

Let $s=k_1n+w_1$ with $k_1\in\NN, w_1\in\mathrm{Ap}(S,n)$. Suppose that $k_1 \not= k$. Hence $0\neq (k_1-k)n=w-w_1$. In particular $w\equiv w_1$ mod$(n)$. This is a contradiction.
\end{proof}

This gives an alternative proof that $S$ is finitely generated.

\begin{corolario}\label{cor:ns-fg}
Let $S$ be a numerical semigroup. Then $S$ is finitely generated.
\end{corolario}
\begin{proof} 
Let $n\in S^*$. By the proposition above, $S=\langle \{n\}\cup \mathrm{Ap}(S,n)\setminus\{0\}\rbrace\rangle$. But the cardinality of $\mathrm{Ap}(S,n)=n$. This proves the result.
\end{proof}

Let $S$ be a numerical semigroup and let $A\subseteq S$. We say that $A$ is a \emph{minimal set of generators} of $S$ if $S=\langle A\rangle$ and no proper subset of $A$ has this property.

\begin{corolario}Let $S$ be a numerical semigroup. Then $S$ has a minimal set of generators. This set is finite and unique: it is actually $S^*\setminus (S^*+S^*)$.
\end{corolario}
\begin{proof} 
Notice that by using the argument in the proof of Proposition \ref{ns-fg}, every generating set can be refined to a minimal generating set.

Let $A=S^*\setminus (S^*+S^*)$ and let $B$ be another minimal generating set. If $B$ is not included in $A$, there exists $a,b,c\in B$ such that $a=b+c$. But this contradicts the minimality of $B$, since in this setting $B\setminus\{a\}$ is a generating system for $S$. This proves $B\subseteq A$. 

Now take $a\in A\subseteq S=\langle B\rangle$. Then $a=\sum_{b\in B}\lambda_b b$. But $a\in S^*\setminus (S^*+S^*)$, and so $\sum_{b\in B}\lambda_b=1$. This means that there exists $b\in B$ with $\lambda_b=1$ and $\lambda_{b'}=0$ for the rest of $b'\in B$. We conclude that $a=b\in B$, which proves the other inclusion.
\end{proof}

Let $S$ be a numerical semigroup. The cardinality of a minimal set of generators of $S$ is called the \emph{embedding dimension} of $S$. We denote it by $\mathrm e(S)$. 

\begin{lema}
Let $S$ be a numerical semigroup. We have $\mathrm e(S)\leq \mathrm m(S)$.
\end{lema}

\begin{proof}
The proof easily follows from the proof of Corollary \ref{cor:ns-fg} or from that of Proposition \ref{ns-fg}.
\end{proof}

\begin{exemples}
\begin{enumerate}[i)]
\item  $S=\NN$ if and only if $\mathrm e(S)=1$.
\item Let $m\in \NN^*$ and let $S=\langle m,m+1,\dots,2m-1 \rangle$. We have $\mathrm{Ap}(S,m)=\lbrace 0,m+1,\dots,2m-1\rbrace$ and $\lbrace m,m+1,\dots,2m-1\rbrace$ is a minimal set of generators of $S$. In particular $\mathrm e(S)=\mathrm m(S)=m$.
\item Let $S=\lbrace 0,4,6,9,10,\dots\rbrace$. We have $\mathrm{Ap}(S,4)=\lbrace 0, 9,6,12\rbrace$. In particular $\mathrm m(S)=4$ and $S=\langle 4,6,9,12\rangle=\langle 4,6,9\rangle$. Hence $\mathrm e(S)=3$.
\end{enumerate}
\end{exemples}

\begin{gap}$ $
We can easily generate ``random'' numerical semigroups with the following command. The first argument is an upper bound for the number of minimal generators, while the second says the range where the generators must be taken from.
\begin{verbatim}
gap> s:=RandomNumericalSemigroup(5,100);
<Numerical semigroup with 5 generators>
gap> MinimalGeneratingSystemOfNumericalSemigroup(s);
[ 6, 7 ]
\end{verbatim}
\end{gap}

Let $S$ be a numerical semigroup. We set $\mathrm F(S)=\max(\NN\setminus S)$ and we call it the \emph{Frobenius number} of $S$. We set $\mathrm C(S)=\mathrm F(S)+1$ and we call it the \emph{conductor} of $S$. Recall that we denoted $\mathrm G(S)=\NN\setminus S$ and we called it the set of \emph{gaps} of $S$. Also we used $\mathrm g(S)$ to denote the cardinality of $\mathrm G(S)$ and we call $\mathrm g(S)$ the \emph{genus} of $S$. We denote by $\mathrm n(S)$ the cardinality of $\lbrace s\in S: s\leq \mathrm F(S)\rbrace$.

\begin{proposicion}[Selmer's formulas]\label{selmer-formulas}
Let $S$ be a numerical semigroup and let $n\in S^*$. We have the following:
\begin{enumerate}[(i)]
\item $\mathrm F(S)=\max(\mathrm{Ap}(S,n))-n$,

\item $\mathrm g(S)=\frac{1}{n}\left(\sum_{w\in\mathrm{Ap}(S,n)}w\right)-\frac{n-1}{2}$.
\end{enumerate}

\end{proposicion}

\begin{proof} 

\begin{enumerate}[(i)]
\item Clearly $\max(\mathrm{Ap}(S,n))-n\notin S$. If $x>\max(\mathrm{Ap}(S,n))-n$ then $x+n>\max(\mathrm{Ap}(S,n))$. Write $x+n=qn+i$, $q\in\NN$, $i\in \{0, \ldots, n-1\}$ and let $w(i)\in\mathrm{Ap}(S,n)$ be the smallest element of $S$ which is congruent to $i$ modulo $n$. Since $x+n >w(i)$, we have $x+n=kn+w(i)$ with $k >0$. Hence $x=(k-1)n+w(i)\in S$.

\item For all $w\in\mathrm{Ap}(S,n)$, write $w=k_in+i$, $k_i\in\NN$, $i\in\{ 0,\ldots, n-1\}$. We have:
\[ \mathrm{Ap}(S,n)=\lbrace 0,k_1n+1,\dots,k_{n-1}n+n-1\rbrace.\]
Let $x\in \NN$ and suppose that $x\equiv i$ mod$(n)$. We claim that $x\in S$ if and only if $w(i)\leq x$. In fact, if $x=q_in+i$, then $x-w(i)=(q_i-k_i)n$.  Hence $w(i)\leq x$ if and only if $k_i\leq q_i$ if and only if $x=(q_i-k_i)n+w(i)\in S$. It follows that $x\notin S$ if and only if $x=q_in+i, q_i<k_i$. Consequently
\[
\mathrm g(S)=\sum_{i=1}^{n-1}k_i=\frac{1}n \left(\sum_{i=1}^{n-1}(k_in+i)\right)-\frac{n-1}{2}=\frac{1}{n}\left(\sum_{w\in\mathrm{Ap}(S,n)}w\right)-\frac{n-1}{2}.\qedhere
\]
\end{enumerate}

\end{proof}

\begin{example}
Let $S=\langle a,b\rangle$ be a numerical semigroup. We have \[\mathrm{Ap}(S,a)=\lbrace 0,b,2b,\dots,(a-1)b\rbrace.\] Hence
\begin{enumerate}[(i)]
\item $\mathrm F(S)=(a-1)b-a=ab-a-b$.
\item $\mathrm g(S)=\frac{1}{a} (a+2a+\dots+(b-1)a)-\frac{a-1}{2}=\frac{(a-1)(b-1)}{2}=
\frac{\mathrm F(S)+1}{2}$.
\end{enumerate}
\end{example}

\begin{lema}
Let $S$ be a numerical semigroup. We have $\mathrm g(S)\geq 
\frac{\mathrm F(S)+1}{2}$.
\end{lema}
\begin{proof} 
Let $s\in \NN$. If $s\in S$, then $\mathrm F(S)-s\notin S$. Thus $\mathrm g(S)$ is greater than or equal to $\mathrm n(S)$. But $\mathrm n(S)+\mathrm g(S)=\mathrm F(S)+1$. This proves the result. 
\end{proof}

\begin{gap}Let $S=\langle 5,7,9\rangle$.

\begin{verbatim}
gap> s:=NumericalSemigroup(5,7,9);
<Numerical semigroup with 3 generators>
gap> FrobeniusNumber(s);
13
gap> ConductorOfNumericalSemigroup(s);
14
gap> ap:=AperyListOfNumericalSemigroupWRTElement(s,5);
[ 0, 16, 7, 18, 9 ]
gap> Maximum(ap)-5;
13
gap> Sum(ap)/5 -2;
8
gap> GenusOfNumericalSemigroup(s);
8
gap> GapsOfNumericalSemigroup(s);
[ 1, 2, 3, 4, 6, 8, 11, 13 ]
\end{verbatim}

\end{gap}

\begin{conjetura}
 Let $g$ be positive integer and let $n_g$ be the number of numerical semigroups $S$ with $\mathrm g(S)=g$. Is $n_g\leq n_{g+1}$? This conjecture is known to be true for $g\leq 67$ but it is still open in general (J. Fromentin, personal communication; see also Manuel Delgado's web page).
\end{conjetura}

\begin{gap}
\begin{verbatim}
gap> List([1..20],i->Length(NumericalSemigroupsWithGenus(i)));
[ 1, 2, 4, 7, 12, 23, 39, 67, 118, 204, 343, 592, 1001, 1693, 2857, 4806, 8045, 
13467, 22464, 37396 ]
\end{verbatim}
\end{gap}

The following result allows to prove Johnson's formulas (see Corollary \ref{johnson-form} below).

\begin{proposicion}
Let $S$ be a numerical semigroup minimally generated by $n_1,\dots,n_p$. Let $d=\gcd(n_1,\dots,n_{p-1})$ and let $T=\langle n_1/d,\dots,n_{p-1}/ d,n_p\rangle$. We have $\mathrm{Ap}(S,n_p)=d\mathrm{Ap}(T,n_p)$.
\end{proposicion}
\begin{proof} Let $w\in\mathrm{Ap}(S,n_p)$. Since $w-n_p\notin S$ then $w\in \langle n_1,\dots,n_{p-1}\rangle$, hence 
$\frac{w}{d}\in \left\langle\frac{n_1}{d},\dots,\frac{n_{p-1}}{d}\right\rangle$. If $\frac{w}{d}-n_p\in T$, then $w-dn_p\in S$, which is a contradiction. Hence  $\frac{w}{d}\in \mathrm{Ap}(T,n_p)$, in particular $w\in d\mathrm{Ap}(T,n_p)$.

Conversely, if $w\in \mathrm{Ap}(T,n_p)$, then $w \in \left\langle \frac{n_1}{d},\dots,\frac{n_{p-1}}{d}\right\rangle$, hence $dw\in \langle n_1,\dots,n_{p-1}\rangle \in S$. Suppose that $dw-n_p\subseteq S$. We have:
\[
dw-n_p=\sum_{i=1}^p\lambda_in_i \hbox{ implies } dw=\sum_{i=1}^{p-1}\lambda_in_i+(\lambda_p+1)n_p.
\]

In particular $d$ divides $\lambda_p+1$. Write $w=\sum_{i=1}^{p-1}\lambda_i\frac{n_i}d+\left(\frac{\lambda_p+1}{d}\right)n_p$, whence $w-n_p\in T$, which is a contradiction. Finally $dw\in \mathrm{Ap}(S,n_p)$. This implies our assertion.
\end{proof}

\begin{corolario}\label{johnson-form}
Let $S$ be a numerical semigroup minimally generated by $\{n_1,\dots,n_p\}$.  Let $d=\gcd(n_1,\dots,n_{p-1})$ and let $T=\langle \frac{n_1}{d},\dots,\frac{n_{p-1}}{d},n_p\rangle$. We have the following:
\begin{enumerate}[(i)]
\item $\mathrm F(S)=dF(T)+(d-1)n_p$,
\item $\mathrm g(S)=dg(T)+\frac{(d-1)(d-2)}{2}$.
\end{enumerate}
\end{corolario}
\begin{proof} 
\begin{enumerate}[(i)]
\item  $\mathrm  F(S)=\max\mathrm{Ap}(S,n_p)-n_p= d\max\mathrm{Ap}(T,n_p)-n_p= d(\max\mathrm{Ap}(T,n_p)-n_p)+(d-1)n_p=d\mathrm F(T)+(d-1)n_p$.

\item $\mathrm g(S)=\frac{1}{n_p}\left(\sum_{w\in\mathrm{Ap}(S,n_p)}w\right)-\frac{n_p-1}{2}= \frac{d}{n_p}\left(\sum_{w\in\mathrm{Ap}(T,n_p)}w\right)-\frac{n_p-1}{2}\\ =d\left(\frac{1}{n_p}\sum_{w\in\mathrm{Ap}(T,n_p)}w-\frac{n_p-1}{2}\right)+\frac{(d-1)(n_p-1)}{2}$.\qedhere
\end{enumerate}
\end{proof}

\begin{example}
Let $S=\langle 20,30,17\rangle$, $T=\langle 2,3,17 \rangle=\langle 2,3\rangle$; $\mathrm F(S)=10\mathrm F(T)+9\times 17=163$ and $\mathrm g(S)=10+9\mathrm 16/2=82$.
\end{example}

Let $S$ be a numerical semigroup. We say that $x\in \NN$ is a \emph{pseudo-Frobenius number} if $x\notin S$ and $x+s\in S$ for all $s\in S^*$. We denote by $\mathrm{PF}(S)$ the set of pseudo Frobenius numbers. The cardinality of $\mathrm{PF}(S)$ is denoted by $\mathrm t(S)$ and we call it the \emph{type} of $S$. Note that $\mathrm F(S)=\max(\mathrm{PF}(S))$.

Let $a,b\in\NN$. We define $\leq_S$ as follows: $a\leq_S b$ if $b-a\in S$. Clearly $\leq_S$ is a (partial) order relation. With this order relation $\mathbb Z$ becomes a poset. We see next that $\mathrm{PF}(S)$ are precisely the maximal gaps of $S$ with respect to $\le_S$.

\begin{proposicion}\label{pseudo-frob-max-gaps}
Let  $S$ be a numerical semigroup. We have \[\mathrm{PF}(S)=\max\nolimits_{\leq_S}(\NN\setminus S).\]
\end{proposicion}
\begin{proof} Let $x\in \mathrm{PF}(S)$: $x\notin S$ and $x+S^*\subseteq S$. Let $y\in \NN\setminus S$ and assume that $x\leq_S y$. If $x\not=y$, then $y-x=s\in S^*$, hence  $y=x+s\in x+S^*\subseteq S$. This is a contradiction.

\noindent Conversely, let $x\in\mathrm{Max}_{\leq_S}\NN\setminus S$. If $x+s\notin S$ for some $s\in S^*$, then $x\leq_S x+s$. This is a contradiction.
\end{proof}

We can also recover the pseudo-Frobenius elements by using, once more, the Apéry sets.

\begin{proposicion}\label{pseudo-frob-and-maximal}
Let  $S$ be a numerical semigroup and let $n\in S^*$. Then
\[\mathrm{PF}(S)=\lbrace w-n \mid w\in \max\nolimits_{\leq_S}\mathrm{Ap}(S,n)\rbrace.\]
\end{proposicion}
\begin{proof} 
Let $x\in \mathrm{PF}(S): x+n\in S$ and $x\notin S$. Hence $x+n\in\mathrm{Ap}(S,n)$. Let us prove that $x+n$ is maximal with respect to $\leq_{S}$. Let $w\in\mathrm{Ap}(S,n)$ such that $x+n\leq_{S}w$. Let $s\in S$ such that $w-x-n=s$. We have  $w-n=x+s$. If $s\in S^*$, then $x+s\in S$. But $w-n\notin S$, a contradiction.

\noindent Conversely let $w\in\mathrm{Max}_{\leq_S}\mathrm{Ap}(S,n)$ and let $s\in S^*$. If $w-n+s\notin S$, then $w+s\in \mathrm{Ap}(S,n)$. This contradicts the maximality of $w$.
\end{proof}

\begin{exemples}
\begin{enumerate}[(i)]
\item Let $S=\langle 5,6,8\rangle$; $\mathrm{Ap}(S,5)=\lbrace 0,6,12,8,14\rbrace$. Hence $\mathrm{PF}(S)=\lbrace 12-5,14-5\rbrace=\lbrace 7,9\rbrace$. In particular, $\mathrm t(S)=2$.
 
\item Let $S=\langle a,b\rangle$ where $a,b\in\NN\setminus\{0,1\}$ and $\gcd(a,b)=1$. We have $\mathrm{Ap}(S,a)=\lbrace 0,b,2b,\dots,(a-1)b\rbrace$. Thus $\mathrm{PF}(S)=\lbrace \mathrm F(S)=(a-1)b-a\rbrace$ and $\mathrm t(S)=1$.
\end{enumerate}
\end{exemples}

In particular, we get the following consequence, which gives an upper bound for the type of a numerical semigroup.

\begin{corolario}
Let $S$ be a numerical semigroup other than $\mathbb N$. We have $\mathrm t(S)\leq \mathrm m(S)-1$.
\end{corolario}
\begin{proof} 
The type $S$ is nothing but the cardinality of the set of maximal elements of $\mathrm{Ap}(S,\mathrm m(S))$ with respect to $\leq_S$. Since $0$ is not a maximal element,  the result follows. 
\end{proof}

\begin{nota}
Let $S$ be a numerical semigroup. In the above inequality, one cannot replace $\mathrm m(S)-1$ by $\mathrm e(S)$ as it is shown in the following example: let $S=\langle s,s+3,s+3n+1,5+3n+2\rangle$, $n\geq 2$, $r\geq 3n+2$, $s=r(3n+2)+3$; then $\mathrm t(S)=2n+3$.
\end{nota}

Type, the number of sporadic elements (elements below the Frobenius number) and the genus of a numerical semigroup are related in the following way.

\begin{proposicion}
Let $S$ be a numerical semigroup and recall that $\mathrm n(S)$ is the cardinality of $\mathrm N(S)=\lbrace s\in S \mid s<\mathrm F(S)\rbrace$. With these notations we have $\mathrm g(S)\leq \mathrm t(S)\mathrm n(S)$.
\end{proposicion}
\begin{proof} Let $x\in\NN$ and let $f_x=\min\lbrace f\in \mathrm{PF}(S)\mid f-x\in S\rbrace$. Let
\[
\phi: \mathrm G(S)\to \mathrm{PF}(S)\times N(S),\ \phi(x)=(f_x,f_x-x).
\]
The map $\phi$ is clearly injective. In particular $\mathrm g(s)$ is less than or equal than the cardinality of $\mathrm{PF}(S)\times N(S)$, which is $\mathrm t(S)\mathrm n(S)$.
\end{proof}

In particular, if we use the fact that $\mathrm g(s)+\mathrm n(s)=\mathrm F(S)+1$, then we obtain the following easy consequence.

\begin{corolario}
Let $S$ be a numerical semigroup. We have $\mathrm F(S)+1\leq (\mathrm t(S)+1)\mathrm n(S)$.
\end{corolario}

\begin{gap}
We go back to $S=\langle 5,7,9\rangle$. 
\begin{verbatim}
gap> PseudoFrobeniusOfNumericalSemigroup(s);
[ 11, 13 ]
gap> TypeOfNumericalSemigroup(s);
2
gap> MultiplicityOfNumericalSemigroup(s);
5
gap> SmallElementsOfNumericalSemigroup(s);
[ 0, 5, 7, 9, 10, 12, 14 ]
gap> Length(last-1);
7
\end{verbatim}
\end{gap}

\begin{conjetura}[Wilf]
$\mathrm F(S)+1\leq \mathrm e(S)\mathrm n(S)$.
\end{conjetura}

\subsection{Numerical semigroups with maximal embedding dimension.}
Let $S$ be a numerical semigroup and recall that $\mathrm e(S)\leq \mathrm m(S)$. In the following we shall consider the case where $\mathrm e(S)=\mathrm m(S)$. We are going to see that if this is the case, then the type is also maximal.

Let $S$ be a numerical semigroup. We say that $S$ has \emph{maximal embedding dimension} if $\mathrm e(S)=\mathrm m(S)$.

Trivially, any minimal generator is in the Ap\'ery set of any nonzero element different from it. We write the short proof for this.

\begin{lema}\label{gen-in-ap}
Let $x$ be a minimal generator of $S$ and let $n\in S^*$, $n\not=x$. We have $x-n\notin S$. In particular $x\in\mathrm{Ap}(S,n)$.
\end{lema}
\begin{proof} If $x-n\in S$, since  $x=n+(x-n)$, this contradicts the the fact that $x$ is a minimal generator.
\end{proof}

As a consequence, we get that the Apéry set of the multiplicity consists of 0 plus the rest of minimal generators.

\begin{proposicion}
Let $n_1<n_2<\cdots <n_e$ be a minimal set of generators of $S$. Then $S$ has maximal embedding dimension if and only if $\mathrm{Ap}(S,n_1)=\lbrace 0,n_2,\dots,n_e\rbrace$.
\end{proposicion}
\begin{proof} 
Assume that $S$ has maximal embedding dimension. By Lemma \ref{gen-in-ap}, $n_2,\dots,n_e\in \mathrm{Ap}(S,n_1)$. But $n_1=\mathrm m(S)=e$, whence $\mathrm{Ap}(S,n_1)=\lbrace 0,n_2,\dots,n_e\rbrace$.

Conversely, $S=\langle (\mathrm{Ap}(S,n_1)\setminus\{0\}) \cup\{n_1\}\rangle=\langle n_1,n_2,\dots,n_e\rangle$. Hence $e=\mathrm e(S)=\mathrm m(S)$.
\end{proof}

As we already mentioned above, the type is also maximal in this kind of numerical semigroup. Actually this also characterizes maximal embedding dimension. 

\begin{proposicion}
Let $n_1<n_2<\dots <n_e$ be a minimal set of generators of $S$. The following are equivalent.
\begin{enumerate}[(i)]
\item   $S$ has maximal embedding dimension. 

\item $\mathrm g(S)=\frac{1}{n_1}\sum_{i=2}^en_i-\frac{n_1-1}{2}$.

\item $\mathrm t(S)=n_1-1=\mathrm m(S)-1$.
\end{enumerate}
\end{proposicion}
\begin{proof} 
If $S$ has maximal embedding dimension, then $\mathrm{Ap}(S,n_1)=\lbrace 0,n_2,\dots,n_e\rbrace$. By Selmer's formulas (Proposition \ref{selmer-formulas}), $\mathrm g(S)= \frac{1}{n_1}\sum_{w\in\mathrm{Ap}(S,n_1)}w-\frac{n_1-1}{2}= \frac{1}{n_1}\sum_{i=2}^en_i-\frac{n_1-1}{2}$. Conversely, we have $\lbrace n_2,\dots,n_e\rbrace\subseteq \mathrm{Ap}(S,n_1)$ and $\frac{1}{n_1}\sum_{w\in\mathrm{Ap}(S,n_1)}w=\frac{1}{n_1}\sum_{i=2}^en_i$. Hence $\mathrm{Ap}(S,n_1)=\lbrace 0,n_2,\dots,n_e\rbrace$. In particular, $S$ has maximal embedding dimension. This proves that (i) and (ii) are equivalent.

Finally we prove that (i) is equivalent to (iii). If $S$ has maximal embedding dimension, then $\mathrm{Ap}(S,n_1)=\lbrace 0,n_2,\dots,n_e\rbrace$. It easily follows that $\mathrm{Max}_{\leq_S}\mathrm{Ap}(S,n_1)=\lbrace n_2,\dots,n_e\rbrace$, whence $\mathrm t(S)=n_1-1=\mathrm m(S)-1$. Now assume that $\mathrm t(S)=n_1-1$.  Then the cardinality of $\mathrm{PF}(S)$ is $n_1-1=\mathrm m(S)-1$. According to Proposition \ref{pseudo-frob-and-maximal}, this means that all the elements in $\mathrm{Ap}(S,n_1)$ with the exception of $0$ are maximal with respect to $\le_S$. We also know that $\{n_2,\dots,n_e\}\subseteq \mathrm{Ap}(S,n_1)$ (Lemma \ref{gen-in-ap}). Hence all minimal generators (other than $n_1$) are maximal in $\mathrm{Ap}(S,n_1)$ with respect to $\le_S$. Assume that there exists $x\in \mathrm{Ap}(S,n_1)\setminus\lbrace 0, n_2,\dots,n_e\rbrace$. Then $x=\sum_{i=1}^e\lambda_in_i, \lambda_i\geq 0$, and since $x-n_1\notin S$, we deduce that $\lambda_1=0$. Since $x\neq0$,  $\lambda_k\not=0$ for some $k$. Thus  $x-n_k\in S$, and consequently $n_k$ is not maximal with respect to $\leq_S$. This is a contradiction. Hence $\mathrm{Ap}(S,n_1)=\lbrace 0, n_2,\dots,n_e\rbrace$, and this yields $n_1=\mathrm m(S)=\mathrm e(S)$.
\end{proof}

As another consequence of Selmer's formulas, we get an easy expression of the Frobenius number of a maximal embedding dimension numerical semigroup.

\begin{proposicion}
Let $n_1<n_2<\dots <n_e$ be a minimal set of generators of $S$ with $e=n_1$.  Then $\mathrm F(S)=n_e-n_1$.
\end{proposicion}
\begin{proof}
This follows from the fact that  $\mathrm F(S)=\max\mathrm{Ap}(S,n_1)-n_1$ (Proposition \ref{selmer-formulas}).
\end{proof}

The converse to this proposition is not true. Just take $S=\langle 4,5,11\rangle$. 

%

\begin{gap} One can always construct maximal embedding dimension numerical semigroups from any numercal semigroup in the following way (see \cite[Chapter 2]{ns}).
\begin{verbatim}
gap> s:=NumericalSemigroup(4,7,9);
<Numerical semigroup with 3 generators>
gap> AperyListOfNumericalSemigroup(s);
[ 0, 9, 14, 7 ]
gap> t:=NumericalSemigroup(4+last);
<Numerical semigroup with 4 generators>
gap> MinimalGeneratingSystemOfNumericalSemigroup(t);
[ 4, 11, 13, 18 ]
\end{verbatim}
\end{gap}

\subsection{Special gaps and unitary extensions of a numerical semigroup}

We introduce in this section another set of notable elements of numerical semigroups, that is, in some sense dual to the concept of minimal generating system. Let $S$ be a numerical semigroup. Notice that an element $s\in S$ is a minimal generator if and only if $S\setminus\{s\}$ is a numerical semigroup. We define the set of \emph{special gaps} of $S$ as
$$
\mathrm{SG}(S)=\lbrace x\in \mathrm{PF}(S) \mid 2x\in S\rbrace.
$$
The duality we mentioned above comes in terms of the following result.

\begin{lema}\label{sg-dual-mg}
Let $x\in \mathbb Z$. Then $x\in \mathrm{SG}(S)$ if and only if $S\cup\lbrace x\rbrace$ is a numerical semigroup.
\end{lema}
\begin{proof} 
Easy exercise.
\end{proof}

If $S$ and $T$ are numerical semigroups, with $S\subset T$, we can construct a chain of numerical semigroups $S=S_1\subset S_2\subset \cdots \subset S_k=T$ such that for every $i$, $S_{i+1}$ is obtained from $S_i$ by adjoining a special gap. This can be done thanks to the following result.

\begin{lema}\label{max-comp-sg}
Let $T$ be a numerical semigroup and assume that $S\subset T$. Then $\max(T\setminus S)\in \mathrm{SG}(S)$. In particular, $S\cup \lbrace \max(T\setminus S)\rbrace$ is a numerical semigroup.
\end{lema}
\begin{proof} Let $x=\max(T\setminus S)$. Clearly $2x\in S$. Take $s\in S^*$. Then $x+s\in T$ and $x<x+s$. Hence $x+s\in S$.
\end{proof}

\begin{nota}
Let $\mathcal O(S)$ be the set of \emph{oversemigroups} of $S$, that is, the set of numerical semigroups $T$ such that $S\subseteq T$. Since $\NN\setminus S$ is a finite set, we deduce that $\mathcal O(S)$ is a finite set.
\end{nota}

\begin{gap}
\begin{verbatim}
gap> s:=NumericalSemigroup(7,9,11,17);;
gap> GenusOfNumericalSemigroup(s);
12
gap> o:=OverSemigroupsNumericalSemigroup(s);;
gap> Length(o)
51
gap> s:=NumericalSemigroup(3,5,7);;
gap> o:=OverSemigroupsNumericalSemigroup(s);;
gap> List(last,MinimalGeneratingSystemOfNumericalSemigroup);
[ [ 1 ], [ 2, 3 ], [ 3 .. 5 ], [ 3, 5, 7 ] ]
\end{verbatim}
\end{gap}

\section{Irreducible numerical semigroups}

A numerical semigroup $S$ is \emph{irreducible} if it cannot be expressed as the intersection of two proper oversemigroups. In the following we will show that irreducible semigroups decompose into two classes: symmetric and pseudo-symmetric. We will also give characterizations of these two classes. Usually in the literature the concepts of symmetric and pseudo-symmetric have been studied separately; the second a sort of generalization of first. Irreducible numerical semigroups gathered these two families together, and since then many papers devoted to them have been published.

The following lemma is just a particular case of Lemma \ref{max-comp-sg}, taking $T=\mathbb N$.

\begin{lema}
Let $S$ be a numerical semigroup other than $\NN$. Then $S\cup\lbrace \mathrm F(S)\rbrace$ is a numerical semigroup.
\end{lema}

The following result is just one of the many characterizations that one can find for irreducible numerical semigroups.

\begin{teorema}\label{carac1-irr}
Let $S$ be a numerical semigroup. The following are equivalent.
\begin{enumerate}[(i)]
\item $S$ is irreducible.

\item $S$ is maximal (with respect to set inclusion) in the set of numerical semigroups $T$ such that $\mathrm F(S)=\mathrm F(T)$.

\item  $S$ is maximal (with respect to set inclusion) in the set of numerical semigroups $T$ such that $\mathrm F(S)\notin T$.
\end{enumerate}
\end{teorema}
\begin{proof} 
\emph{(i) implies (ii)} Let $T$ be a numerical semigroup such that $\mathrm F(S)=\mathrm F(T)$. If $S\subseteq T$, then $S=T\cap (S\cup \lbrace \mathrm F(S)\rbrace)$. Since $S\not= S\cup \lbrace \mathrm F(S)\rbrace$, we deduce $S=T$.

\emph{(ii) implies (iii)} Let $T$ be a numerical semigroup such that $\mathrm F(S)\notin T$ and assume that $S\subseteq T$. The  set  $T_1=T\cup\lbrace \mathrm F(S)+1,\mathrm F(S)+2,\dots\rbrace$ is a numerical semigroup with $\mathrm F(T_1)=\mathrm F(S)$. But $S\subseteq T_1$, whence $S=T_1$. Since $\mathrm F(S)+k\in S$ for all $k\geq 1$, it follows that $S=T$.

\emph{(iii) implies (i)} Let $S_1,S_2$ be two numerical semigroups such that $S\subseteq S_1$, $S\subseteq S_2$ and $S=S_1\cap S_2$. Since $\mathrm F(S)\not\in S$, $\mathrm F(S)\not \in S_i$ for some $i\in\{1,2\}$. By (iii), $S_i=S$.
\end{proof}

Let $S$ be a numerical semigroup. We say that $S$ is \emph{symmetric} if
\begin{enumerate}[(i)]
\item $S$ is irreducible,
\item $\mathrm F(S)$ is odd.
\end{enumerate}

We say that $S$ is \emph{pseudo-symmetric} if
\begin{enumerate}[(i)]
\item $S$ is irreducible,

\item $\mathrm F(S)$ is even.
\end{enumerate}


Next we show some characterizations of symmetric and pseudo-symmetric numerical semigroups. We first prove the following.

\begin{proposicion}\label{max-t-2}
Let $S$ be a numerical semigroup and suppose that 
\[
H=\left\{ x\in \mathbb Z\setminus S\mid \mathrm F(S)-x\notin S, x\not=\frac{\mathrm F(S)}{2}\right\}\]
is not empty. If $h=\max H$, then $S\cup \lbrace h\rbrace$ is a numerical semigroup.
\end{proposicion}
\begin{proof} 
Since $S\subseteq S\cup \lbrace h\rbrace$, the set  $\NN\setminus (S\cup \lbrace h\rbrace)$ has finitely many elements. Let $a,b\in S\cup \lbrace h\rbrace$.
\begin{itemize}
\item If $a,b\in S$, then $a+b\in S$.

\item Let  $a\in S^*$. Assume that $a+h\notin S$, by the maximality of $h$,  we deduce  $\mathrm F(S)-a-h=\mathrm F(S)-(a+h)\in S$. Hence $\mathrm F(s)-h=a+\mathrm F(S)-a-h\in S$. This contradicts the definition of $h$.

\item If $2h\notin S$, then $\mathrm F(S)-2h=s\in S^*$. This implies that $\mathrm F(S)-h=h+s \in S$ (by the preceding paragraph). This is a contradiction.\qedhere
\end{itemize}
\end{proof}

\begin{gap}
In light of Proposition \ref{max-t-2} and Lemma \ref{sg-dual-mg}, if for a numerical semigroup, there exists a maximum of $\{x\in \mathbb Z\setminus (S\cup\{\mathrm F(S)/2\})\mid \mathrm F(S)-x\not\in S\} $, then it is a special gap.
\begin{verbatim}
gap> s:=NumericalSemigroup(7,9,11,17);
<Numerical semigroup with 4 generators>
gap> g:=GapsOfNumericalSemigroup(s);
[ 1, 2, 3, 4, 5, 6, 8, 10, 12, 13, 15, 19 ]
gap> Filtered(g, x-> (x<>19/2) and not(19-x in s));
[ 4, 6, 13, 15 ]
gap> SpecialGapsOfNumericalSemigroup(s);
[ 13, 15, 19 ]
\end{verbatim}
\end{gap}

We have introduced the concepts of symmetric and pseudo-symmetric as subclasses of the set of irreducible numerical semigroups. However, these two concepts existed before that of irreducible numerical semigroup, and thus the definitions were different than the ones we have given above. Next we recover the classical definitions of these two classical concepts. Needless to say that as in the case of irreducible numerical semigroups, there are many different characterizations of these properties. We will show some below.

\begin{proposicion}\label{carac2-irre}
Let $S$ be a numerical semigroup.
\begin{enumerate}[(i)]
\item $S$ is symmetric if and only if for all $x\in \mathbb Z\setminus S$, we have $\mathrm F(S)-x\in S$.

\item $S$ is pseudo-symmetric if and only if for all $x\in \mathbb Z\setminus S$, $\mathrm F(S)-x\in S$ or $x=\frac{\mathrm F(S)}{2}$.
\end{enumerate}
\end{proposicion}
\begin{proof} 
\begin{enumerate}[(i)]
\item Assume that $S$ is symmetric. Then $\mathrm F(S)$ is odd, and thus $H=\lbrace x\in \mathbb Z\setminus S\mid \mathrm F(S)-x\notin S\}=\lbrace x\in \mathbb Z\setminus S\mid \mathrm F(S)-x\notin S, x\not={\mathrm F(S)}/2\}$. If $H$ is not the emptyset, then $T=S\cup\lbrace h=\max H\rbrace$ is a numerical semigroup with Frobenius number $\mathrm F(S)$ containing properly $S$, which is impossible in light of Theorem \ref{carac1-irr}.

For the converse note that $\mathrm F(S)$ cannot be even, since otherwise as $\mathrm F(S)/2\not\in S$, we would have $\mathrm F(S)-\mathrm F(S)/2= \mathrm F(S)/2\in S$; a contradiction. So, we only need to prove that $S$ is irreducible. Let to this end $T$ be a numerical semigroup such that $\mathrm F(S)\notin T$ and suppose that $S\subset T$. Let $x\in T\setminus S$. By hypothesis $\mathrm F(S)-x\in S$. This implies that $\mathrm F(S)=(\mathrm F(S)-x)+x\in T$. This is a contradiction (we are using here Theorem \ref{carac1-irr} once more).

\item  The proof is the same as the proof of (i). \qedhere
\end{enumerate}
\end{proof}

The maximality of irreducible numerical semigroups, in the set of numerical semigroups with the same Frobenius number, translates to minimality in terms of gaps. This is highlighted in the next result.

\begin{corolario}\label{formula-genus-irred}
Let $S$ be a numerical semigroup.
\begin{enumerate}[(i)]
\item $S$ is symmetric if and only if $\mathrm g(S)=\frac{\mathrm F(S)+1}{2}$.

\item $S$ is pseudo-symmetric if and only if $\mathrm g(S)=\frac{\mathrm F(S)+2}{2}$.
\end{enumerate}
Hence irreducible numerical semigroups are those with the least possible genus.
\end{corolario}

Recall that the Frobenius number and genus for every embedding dimension two numerical semigroup are known; as a consequence, we get the following.

\begin{corolario}
Let $S$ be a numerical semigroup. If $\mathrm e(S)=2$, then $S$ is symmetric.
\end{corolario}

The rest of the section is devoted to characterizations in terms of the Ap\'ery sets (showing in this way their ubiquity). First we show that Ap\'ery sets are closed under summands.

\begin{lema}\label{ap-closed}
Let $S$ be a numerical semigroup and let $n\in S^*$. If $x,y\in S$ and $x+y\in\mathrm{Ap}(S,n)$, then $x,y\in\mathrm{Ap}(S,n)$.
\end{lema}
\begin{proof}
Assume to the contrary, and without loss of generality, that $y-n\in S$. Then $x+y-n\in S$, and consequently $x+y\not\in \mathrm{Ap}(S,n)$.
\end{proof}

\begin{proposicion}\label{carac-sym-ap}
Let $S$  be a numerical semigroup and let $n\in S^*$. Let $\mathrm{Ap}(S,n)=\lbrace 0=a_0<a_1<\dots<a_{n-1}\rbrace$. Then $S$ is symmetric if and only if $a_i+a_{n-1-i}=a_{n-1}$ for all $i\in\{0,\dots, n-1\}$.
\end{proposicion}
\begin{proof}
Suppose that $S$ is symmetric. From Proposition \ref{selmer-formulas}, we know that $\mathrm F(S)=a_{n-1}-n$. Let $0\leq i\leq n-1$. Since $a_i-n\notin S$, we get $\mathrm F(S)-a_i+n=a_{n-1}-a_i\in S$. Let $s\in S$ be such that $a_{n-1}-a_i=s$. Since $a_{n-1}=a_i+s\in\mathrm{Ap}(S,n)$, by Lemma \ref{ap-closed}, $s\in\mathrm{Ap}(S,n)$. Hence $s=a_j$ for some $0\leq j\leq n-1$. As this is true for any $i$, we deduce that $j=n-1-i$.

Conversely, the hypothesis implies that $\mathrm{Max}_{\leq_S}\mathrm{Ap}(S,n)= a_{n-1}$. Hence $\mathrm{PF}(S)=\lbrace \mathrm F(S)\rbrace$ (Proposition \ref{pseudo-frob-and-maximal}). Also, by Proposition \ref{pseudo-frob-max-gaps}, $\{\mathrm F(S)\}= \mathrm{Max}_{\leq_S}(\NN\setminus S)$. If $x\notin S$, then $x\leq_S \mathrm F(S)$, whence $\mathrm F(S)-x\in S$. To prove that $\mathrm F(S)$ is odd, just use the same argument of the proof of Proposition \ref{carac2-irre}.
\end{proof}

As a consequence of the many invariants that can be computed using Ap\'ery sets, we get the following characterizations of the symmetric property.

\begin{corolario}
Let $S$  be a numerical semigroup. The following conditions are equivalent.
\begin{enumerate}[(i)]
\item $S$ is symmetric.

\item $\mathrm{PF}(S)=\lbrace \mathrm F(S)\rbrace$.

\item If $n\in S$, then $\mathrm{Max}_{\leq S}(\mathrm{Ap}(S,n)=\{\mathrm F(S)+n\}$.

\item $\mathrm t(S)=1$.
\end{enumerate}
\end{corolario}

Now, we are going to obtain the analogue for pseudo-symmetric numerical semigroups. The first step is to deal with one half of the Frobenius number.

\begin{lema} 
Let $S$  be a numerical semigroup and let $n\in S^*$. If $S$ is pseudo-symmetric, then $\frac{\mathrm F(S)}{2}+n\in\mathrm{Ap}(S,n)$.
\end{lema}
\begin{proof} 
Clearly $\frac{\mathrm F(S)}{2}\notin S$. If $\frac{\mathrm F(S)}{2}+n\notin S$, then $\mathrm F(S)-\frac{\mathrm F(S)}{2}-n\in S$. This implies that $\frac{\mathrm F(S)}{2}\in S$, which is a contradiction.
\end{proof}

\begin{proposicion} 
Let $S$  be a numerical semigroup and let $n\in S^*$. Let $\mathrm{Ap}(S,n)=\lbrace 0=a_0<a_1,\dots<a_{n-2}\rbrace\cup\left\lbrace \frac{\mathrm F(S)}{2}+n\right\rbrace$. Then $S$ is pseudo-symmetric if and only if $a_i+a_{n-2-i}=a_{n-2}$ for all $0\leq i\leq n-2$.
\end{proposicion}
\begin{proof} 
Suppose that $S$ is pseudo-symmetric and let $w\in\mathrm{Ap}(S,n)$. If $w\not= \frac{\mathrm F(S)}{2}+n$, then $w-n\notin S$ and $w-n\not= \frac{\mathrm F(S)}{2}$. Hence $\mathrm F(S)-(w-n)=\mathrm F(S)+n-w=\max\mathrm{Ap}(S,n)-w\in S$. Since $\mathrm F(S)-w\notin S$, then $\mathrm F(S)+n-w=\max\mathrm{Ap}(S,n)-w\in \mathrm{Ap}(S,n)$. But $\max(S,n)-w\not=\frac{\mathrm F(S)}{2}+n$ (otherwise $w=\frac{\mathrm F(S)}{2}$, a contradiction). Now we use the same argument as in the symmetric case (Proposition \ref{carac-sym-ap}).

Conversely, let $x\not= \frac{\mathrm F(S)}{2}$, $x\notin S$. Take $w\in\mathrm{Ap}(S,n)$ such that $w\equiv x\pmod n$. There exists $k\in\NN^*$ such that $x=w-kn$ (compare with Proposition \ref{write-ap}).
\begin{enumerate}
\item If $w=\frac{\mathrm F(S)}{2}+n$, then $\mathrm F(S)-x=\frac{\mathrm F(S)}{2}+(k-1)n$. But $x\not=\frac{\mathrm F(S)}{2}$.  Hence $k\geq 2$, and consequently $\mathrm F(S)-x=w+(k-2)n\in S$.

\item If $w\not= \frac{\mathrm F(S)}{2}+n$, then $\mathrm F(S)-x=\mathrm F(S)+n-w+(k-1)n=a_{n-2}-w+(k-1)n\in S$, because $a_{n-2}-w\in S$.\qedhere
\end{enumerate}
\end{proof}

Again, by using the properties of the Ap\'ery sets, we get several characterizations for pseudo-symmetric numerical semigroups.

\begin{corolario}
Let $S$  be a numerical semigroup. The following conditions are equivalent.
\begin{enumerate}[(i)]
\item $S$ is pseudo-symmetric.
\item $\mathrm{PF}(S)=\left\lbrace \mathrm F(S),\frac{\mathrm F(S)}{2} \right\rbrace$.
\item If $n\in S$, then $\mathrm{Max}_{\leq S}(\mathrm{Ap}(S,n))=\left\lbrace \frac{\mathrm F(S)}{2}+n, \mathrm F(S)+n\right\rbrace$.
\end{enumerate}
\end{corolario} 

\begin{example}
Let  $S$  be a numerical semigroup. If $S$ is pseudo-symmetric, then $\mathrm t(S)=2$. The converse is not true in general. Let $S=\langle 5,6,8\rangle$. We have $\mathrm{Ap}(S,5)=\lbrace 0,6,12,8,14\rbrace$. Thus, $\mathrm{PF}(S)=\lbrace 7,9\rbrace$, and $\mathrm t(S)=2$. However $S$ is not pseudo-symmetric.
\end{example}

\begin{gap}
Let us see how many numerical semigroups with Frobenius number 15 and type 2 are not pseudo-symmetric.
\begin{verbatim}
gap> l:=NumericalSemigroupsWithFrobeniusNumber(16);;
gap> Length(l);
205
gap> Filtered(l, s->TypeOfNumericalSemigroup(s)=2);
[ <Numerical semigroup>, <Numerical semigroup>, <Numerical semigroup>,
  <Numerical semigroup>, <Numerical semigroup>, <Numerical semigroup>,
  <Numerical semigroup>, <Numerical semigroup>, <Numerical semigroup>,
  <Numerical semigroup>, <Numerical semigroup>, <Numerical semigroup>,
  <Numerical semigroup>, <Numerical semigroup> ]
gap> Filtered(last,IsPseudoSymmetricNumericalSemigroup);
[ <Numerical semigroup>, <Numerical semigroup>, <Numerical semigroup>,
  <Numerical semigroup>, <Numerical semigroup>, <Numerical semigroup>,
  <Numerical semigroup> ]
gap> Difference(last2,last);
[ <Numerical semigroup with 3 generators>, <Numerical semigroup with 3 generators>,
  <Numerical semigroup with 3 generators>, <Numerical semigroup with 3 generators>,
  <Numerical semigroup with 3 generators>, <Numerical semigroup with 4 generators>,
  <Numerical semigroup with 5 generators> ]
gap> List(last, MinimalGeneratingSystemOfNumericalSemigroup);
[ [ 3, 14, 19 ], [ 3, 17, 19 ], [ 5, 7, 18 ], [ 5, 9, 12 ], [ 6, 7, 11 ],
  [ 6, 9, 11, 13 ], [ 7, 10, 11, 12, 13 ] ]
\end{verbatim}
\end{gap}

%
%
%
%
%

\subsection{Decomposition of a numerical semigroup into irreducible semigroups}

Recall that a numerical semigroup $S$ is irreducible if it cannot be expressed as the intersection of two numerical semigroups properly containing it. We show in this section that every numerical semigroup can be expressed as a finite intersection of irreducible numerical semigroups.

\begin{teorema}\label{existence-dec-irred}
Let $S$ be a numerical semigroup. There exists a finite set of irreducible numerical semigroups $\{S_1,\dots,S_r\}$ such that $S=S_1\cap\dots\cap S_r$.
\end{teorema}
\begin{proof} 
If $S$ is not irreducible, then there exist two numerical semigroups $S^1$ and $S^2$ such $S=S^1\cap S^2$ and  $S\subset S^1$ and  $S\subset S^2$. If $S^1$ is not irreducible, then we restart with $S^1$, and so on. We construct this way a sequence of oversemigroups of $S$. This process will stop, because $\mathcal O(S)$ has finitely many elements.
\end{proof}

The next step is to find a way to compute an ``irredundant" decomposition into irreducible numerical semigroups. The key result to accomplish this task is the following proposition.

\begin{proposicion}
Let $S$ be a numerical semigroup and let $S_1,\dots,S_r\in \mathcal O(S)$. The following conditions are equivalent.
\begin{enumerate}[(i)]
\item $S=S_1\cap\dots\cap S_r$.
\item For all $h\in \mathrm{SG}(S)$, there is $i\in \{1,\dots,r\}$ such that $h\notin S_i$.
\end{enumerate}
\end{proposicion}
\begin{proof} 
\emph{(i) implies (ii)} Let $h\in \mathrm{SG}(S)$. Then $h\notin S$, which implies that $h\notin S_i$ for some $i\in \{1,\ldots,r\}$.

\emph{(ii) implies (i)} Suppose that  $S\subset S_1\cap\dots\cap S_r$, and let $h={\max}(S_1\cap\dots\cap S_r\setminus S)$. In light of Lemma \ref{max-comp-sg},  $h\in \mathrm{SG}(S)$, and for all $i \in\{1,\dots,r\}$, $h\in S_i$, contradicting the hypothesis.
\end{proof}

\begin{nota}
Let $\mathcal I(S)$ be the set of irreducible numerical semigroups of $\mathcal O(S)$, and let $\mathrm{Min}_{\subseteq}(I(S))$ be the set of minimal elements of $\mathcal I(S)$ with respect to set inclusion. Assume that  $\mathrm{Min}_{\le_\subseteq}(I(S))=\lbrace S_1,\dots,S_r\rbrace$. Define  
$\mathrm C(S_i)=\lbrace h\in \mathrm{SG}(S):h\notin S_i\rbrace$. We have $S= S_1\cap\dots\cap S_r$ if and only if $\mathrm{SG}(S)=\mathrm C(S_1)\cup\dots\cup \mathrm C(S_r)$. This gives a procedure to compute a (nonredundant) decomposition of $S$ into irreducibles. This decomposition might not be unique, and not all might have the same number of irreducibles involved.
\end{nota}

\begin{gap}
\begin{verbatim}
gap> s:=NumericalSemigroup(7,9,11,17);;
gap> DecomposeIntoIrreducibles(s);
[ <Numerical semigroup>, <Numerical semigroup>, <Numerical semigroup> ]
gap> List(last,MinimalGeneratingSystemOfNumericalSemigroup);
[ [ 7, 8, 9, 10, 11, 12 ], [ 7, 9, 10, 11, 12, 13 ], [ 7, 9, 11, 13, 15, 17 ] ]
\end{verbatim}
\end{gap}

There exists some (inefficient) bound on the number of irreducible numerical semigroups appearing in a minimal decomposition of a numerical semigroup into irreducibles. Actually, there might be different minimal decompositions (in the sense that they cannot be refined to other decompositions) with different cardinalities. So it is an open problem to know the minimal cardinality among all possible minimal decompositions.

%
%
%
%
%

\subsection{Free numerical semigroups}

We present in this section a way to construct easily symmetric numerical semigroups. This idea was originally exploited by Bertin, Carbonne and Watanabe among others (see \cite{bertin, delorme, watanabe}) and goes back to the 70's.

Let $S$ be a numerical semigroup and let $\{a_0,\dots,a_h\}$ be its minimal set of generators. Let $d_1=a_0$ and for all $k\geq 2$, set $d_k=\gcd(d_{k-1},a_k)$. Define $e_k=\frac{d_k}{d_{k+1}}$, $1\leq k\leq h$. 

We say that $S$ is \emph{free} for the arrangement $(a_0,\dots,a_h)$ if for all $k\in \{1,\ldots, h\}$:
\begin{enumerate}[(i)]
\item   $e_k>1$,

\item $e_ka_k$ belongs to the semigroup generated by $\{a_0,\dots,a_{k-1}\}$.
\end{enumerate}

We say that $S$ is \emph{telescopic} if $a_0<a_1<\dots<a_h$ and $S$ is free for the arrangement $(a_0,\dots,a_h)$.

There is an alternative way of introducing free semigroups with the use of gluings (more modern notation), see for instance \cite[Chapter 8]{ns}. 

One of the advantages of dealing with free numerical semigroups is that every integer admits a unique representation in terms of its minimal generators if we impose some bounds on the coefficients.

\begin{lema}\label{std-repr}
Assume that $S$ is free for the arrangement $(a_0,\dots,a_h)$, and let $x\in\ZZ$. There exist unique $\lambda_0,\dots,\lambda_{h}\in \ZZ$ such that the following holds:
\begin{enumerate}[(i)]
\item $x=\sum_{k=1}^h\lambda_ka_k$,
\item for all $h\in\{1,\dots,h\}$, $0\leq \lambda_k < e_k$.
\end{enumerate}
\end{lema}
\begin{proof} 
\emph{Existence.} The group generated by $S$ is $\ZZ$, and so there exist $\alpha_0,\dots,\alpha_h\in\ZZ$ such that $x=\sum_{k=1}^h\alpha_ka_k$. Write $\alpha_h=q_he_h+\lambda_h$, with $0\le \lambda_h<e_h$. But $e_ha_h=\sum_{i=0}^{h-1}\beta_ia_i$, with $\beta_i\in\NN$ for all $i\in\{1,\ldots,h-1\}$. Hence 
$$
x=\sum_{k=0}^{h-1}(\lambda_k+q_h\beta_k)+\lambda_ha_h,
$$
and $0\leq \lambda_h<e_h$. Now the result follows by an easy induction on $h$.

\emph{Uniqueness}. Let $x=\sum_{k=0}^h\alpha_ka_k=\sum_{k=0}^h\beta_ka_k$ be two distinct such representations, and let $j\geq 1$ be the greatest integer such that $\alpha_j\not= \beta_j$. We have
$$
(\alpha_j-\beta_j)a_j=\sum_{k=1}^{j-1}(\beta_k-\alpha_k)a_k.
$$
In particular, $d_j$ divides $(\alpha_j-\beta_j)a_j$. But $\gcd(d_j,a_j)=d_{j+1}$, , whence $\frac{d_j}{d_{j+1}}$ divides $(\alpha_j-\beta_j)\frac{a_j}{d_{j+1}}$. As $\gcd(d_j/d_{j+1},a_j/d_{j+1})=1$, this implies that $\frac{d_j}{d_{j+1}}$ divides $\alpha_j-\beta_j$. However $|\alpha_j-\beta_j|< e_j=\frac{d_j}{d_{j+1}}$, yielding a contradiction. 
\end{proof}
An expression of $x$ like in the preceding lemma is called a \emph{standard representation}. As a consequence of this representation we obtain the following characterization for membership to a free numerical semigroup.


\begin{lema}\label{member-std-repr}
Suppose that $S$ is free for the arrangement $(a_0,\dots,a_h)$ and let $x\in\NN$. Let $x=\sum_{k=0}^h\lambda_ka_k$ be the standard representation of $x$. We have $x\in S$ if and only if $\lambda_0\geq 0$.
\end{lema}
\begin{proof} If $\lambda_0\geq 0$ then clearly $x\in S$. Suppose that $x\in S$ and write $x=\sum_{k=0}^h\alpha_ka_k$ with $\alpha_0,\dots,\alpha_h\in \NN$. Imitating the construction of a standard representation made in the above Lemma, we easily obtain the result.
\end{proof}

With this it is easy to describe the Ap\'ery set of the first generator in the arrangement that makes the semigroup free.

\begin{corolario} \label{ap-free}
Suppose that $S$ is free for the arrangement $(a_0,\dots,a_h)$. Then
$$
\mathrm{Ap}(S,a_0)=\left\lbrace \sum_{k=1}^h\lambda_ka_k \mid 0\leq \lambda_k<e_k \hbox{ for all } k\in \{1,\ldots ,  h\}\right\rbrace.
$$
\end{corolario}
\begin{proof} Let $x\in S$ and let $x=\sum_{k=0}^h\lambda_ka_k$ be the standard representation of $x$. Clearly $x-a_0=(\lambda_0-1)a_0+\sum_{k=1}^h\lambda_ka_k$ is the standard representation of $x-\lambda_0$. Hence $x-a_0\notin S$ if and only if $\lambda_0=0$. This proves our assertion.
\end{proof}

As usual, once we know an Ap\'ery set, we can derive many properties of the semigroup.

\begin{proposicion}
Let $S$ be free for the arrangement $(a_0,\dots,a_h)$.
\begin{enumerate}[(i)]
\item $\mathrm F(S)=\sum_{k=1}^h(e_k-1)-r_0$.

\item  $S$ is symmetric.
 
\item  $\mathrm g(S)=\frac{\mathrm F(S)+1}{2}$.
\end{enumerate} 
\end{proposicion}
\begin{proof}  
We have $\mathrm F(S)=\max\mathrm{Ap}(S,a_0)-a_0$, by Proposition \ref{selmer-formulas}. As $\max\mathrm{Ap}(S,a_0)=\sum_{k=1}^h(e_k-1)a_k$, (i) follows easily.

Assertion (ii) is a consequence of Corollary \ref{ap-free} and Proposition \ref{carac-sym-ap}.

%
%
%

Finally (iii) is a consequence of (ii) and Corollary \ref{formula-genus-irred}.
\end{proof}

\begin{gap}
The proportion of free numerical semigroup compared with symmetric numerical semigroups with fixed Frobenius number (or genus) is small.

\begin{verbatim}
gap> List([1,3..51], i ->
> [Length(FreeNumericalSemigroupsWithFrobeniusNumber(i)),
> Length(IrreducibleNumericalSemigroupsWithFrobeniusNumber(i))]);
[ [ 1, 1 ], [ 1, 1 ], [ 2, 2 ], [ 3, 3 ], [ 2, 3 ], [ 4, 6 ], [ 5, 8 ], [ 3, 7 ],
  [ 7, 15 ], [ 8, 20 ], [ 5, 18 ], [ 11, 36 ], [ 11, 44 ], [ 9, 45 ], [ 14, 83 ],
  [ 17, 109 ], [ 12, 101 ], [ 18, 174 ], [ 24, 246 ], [ 16, 227 ], [ 27, 420 ],
  [ 31, 546 ], [ 21, 498 ], [ 35, 926 ], [ 38, 1182 ], [ 27, 1121 ] ]
\end{verbatim}

\end{gap}

\section{Semigroup of an irreducible meromorphic curve}
Let $\KK$ be an algebraically closed field of characteristic zero and let $f(x,y)=y^n+a_1(x)y^{n-1}+\dots+a_n(x)$ be a nonzero polynomial of $\KK((x))[y]$ where $\KK((x))$ denotes the field of meromorphic series in $x$. The aim of this section is to associate with $f$, when $f$ is irreducible, a subsemigroup of $\mathbb{Z}$. The construction of this subsemigroup is based on the notion of Newton-Puiseux exponents. These exponents
appear when we solve $f$ as a polynomial in $y$, and it turns out that the roots are elements of ${\mathbb K}((x^\frac{1}{n})$. Two cases are of intereset: the local case, that is, the case when $f\in {\mathbb K}[[x]][y]$, and the case when $f\in {\mathbb K}[x^{-1}][y]$ with the condition that $F(x,y)=f(x^{-1},y)$ has one place at infinity. In the first case, the subsemigroup associated with $f$ is a numerical semigroup. In the second case, this subsemigroup is a subset of $-{\mathbb N}$, and some
of its numerical properties have some interesting applications in the study of the embedding of special curves with one place at infinity in the affine plane.

\subsection{Newton-Puiseux theorem}

\medskip

\begin{teorema}[Hensel's Lemma]  Let $f$ be as above and assume that $f\in\KK\llbracket x\rrbracket [y]$.  Assume that there exist  two nonconstant polynomials $\tilde{g},\tilde{h}\in \KK[y]$ such that:
\begin{enumerate}[(i)]
\item $\tilde{g},\tilde{h}$ are monic in $y$,

\item $\gcd(\tilde{g},\tilde{h})=1$,

\item $f(0,y)=\tilde{g} \tilde{h}$.
\end{enumerate}
Then there exist $g,h\in\KK\llbracket x\rrbracket [y]$ such that:
\begin{enumerate}[(i)]
\item $g,h$ are monic in $y$, 

\item $\mathrm g(0,y)=\tilde{g}, h(0,y)=\tilde{h}$,  

\item $\deg_y g = \deg \tilde{g}$, $\deg_y h=\deg\tilde{h}$,

\item $f=g h$.
\end{enumerate}
\end{teorema}

\begin{proof} 

Let $r$ (respectively $s$) be the degree of $\tilde{g}$ (respectively $\tilde{h}$) and write $f(x,y)=\sum_{q\geq 0}f_q(y)x^q$. Clearly $f(0,y)=f_0$ is monic of degree $n$ in $y$. Furthermore, we can assume that  $\deg_yf_q<n$ for all $q\geq 1$. For all $i\geq 0$, we construct $f_i,g_i\in\KK[y]$ such that:
\begin{enumerate}
\item $g_0=\tilde{g}, h_0=\tilde{h}$,
\item For all $i\geq 1$, $\deg_yg_i< r$ and $\deg_yh_i<s$,
\item for all $q\geq 1$, $f_q=\sum_{i=1}^qg_ih_{q-i}$.
\end{enumerate}

If $i=0$, then we set $g_0=\tilde{g}, h_0=\tilde{h}$. Suppose that we have $g_0,\dots,g_{q-1},h_0,\dots,h_{q-1}$. Let $e_q=f_q-\sum_{i=0}^{q-1}g_ih_{q-i}$. Note that $\deg_ye_q< n$. We need to prove the existence of two monic polynomials $g_q,h_q$ such that $e_q=h_0g_q+g_0h_q$, $\deg_yg_q<r$ and $\deg_yh_q<s$. To this end, we use Euclid's extended algorithm for polynomials with coefficients in a field.  By  hypothesis, $\gcd(g_0,h_0)=1$. Let $\alpha,\beta\in\KK[y]$ be such that $\alpha g_0+\beta h_0=1$. We have $e_q=(e_q\alpha) g_0+(e_q\beta) h_0$. Let $G_q=e_q\beta$, $H_q=e_q\alpha$ and write $G_q=Qg_0+R$ with $\deg_yR<r$. We have
$$
e_q=(e_q\alpha)g_0+(Qg_0+R) h_0=(e_q\alpha+Qh_0)g_0+Rh_0.
$$
Let $g_q=R, h_q=e_q\alpha+Qh_0$. Since $\deg_yg_q<r$, it follows that  $\deg_yh_q<s$. Hence $g_q,h_q$ fulfill the above conditions.
\end{proof}

\begin{proposicion}
Let $f(x,y)\in \KK((x))[y]$ be as above. There exist $m\in\NN$ and $y(t)\in\KK((t)) $ such that $f(t^m,y(t))=0$.
\end{proposicion}

\begin{proof} 
We shall prove the result by induction on the degree in $y$ of $f$. If $n=1$, then $f=y-a_1(x)$. Hence $f(t,a_1(t))=0$. Suppose that $n\geq 2$. We shall assume the following.

\begin{enumerate}

\item $a_1(x)=0$.

\item $a_k(x)\in\KK\llbracket x\rrbracket $ for all $k\in \{2,\ldots, n\}$ and $a_k(0)\not=0$ for some $k\in\{2,\ldots,n\}$.
\end{enumerate}
It follows that $f(0,y)=y^n+a_2(0)y^{n-2}+\dots+a_n(0)$ is not a power in $\KK[y]$. Hence there exist nonconstant  monic polynomials $\tilde{g}(y),\tilde{h}(y)\in\KK[y]$ such that $\gcd(\tilde{g}(y),\tilde{h}(y))=1$ and $f(0,y)=\tilde{g}(y)\tilde{h}(y)$. By Hensel's lemma, there exist monic polynomials $g,h\in\KK\llbracket x\rrbracket [y]$ such that 
$\deg_y g, \deg_yh <n$ and $f=g h$. By induction hypothesis there exist $n\in\NN$ and $y(t)\in\KK((t))$ such that $\mathrm g(t^n,y(t))=0$. In particular, $f(t^n,y(t))=0$.

\begin{enumerate}
\item  Assume that $a_1(x)\not=0$. Let $z=y+\frac{a_1}{n}$ and let $\mathrm F(x,z)=f(x,z-\frac{a_1}{n})$. Let $m\in\NN$ and $z(t)\in\KK((t))$ such that $\mathrm F(t^n,z(t))=0$. We have $f\left(t^n,z(t)-\frac{a_1}{n}\right)=0$.

\item  Let $f=y^n+\sum_{k=2}^na_k(x)y^{n-k}$ with $a_k(x) \not=0$ for some $k\in \{2,\dots,n\}$ (if $f(x,y)=y^n$, then $f(t,0)=0$, and so it suffices to take $m=1$ and $y(t)=0$). For all $k\in \{2,\dots,n\}$ such that $a_k\not= 0$, let $u_k=\mathrm{ord}_x(a_k)$. Set $u=\min\left\{ \frac{u_k}{k} \mid a_k\not=0 \right\}$. There exists an index  $r$  such that $u=\frac{u_r}{r}$. Let $x=w^r$ and $z=w^{-u_r}y$, and let $g(w,z)=w^{-nu_r}f(w^r,y)$. We have
\begin{align*}
g(w,z)&=w^{-nu_r}\left(w^{nu_r}z^n+\sum_{k=2}^na_k(w^r)w^{u_r(n-k)}z^{n-k}\right)\\
      &=z^n+\sum_{k=2}^na_k(w^r)w^{u_r(n-k)}z^{n-k}.
\end{align*}
\noindent Let $b_k(w)=a_k(w^r)w^{u_r(n-k)}$. We have $\mathrm{ord}_wb_k=ru_k-ku_r\geq 0$, hence $b_k\in\KK[[w]]$. Furthermore, $\mathrm{ord}_wb_r(w)=0$, that is, $b_r(0)\not=0$. Finally, if $m\in\NN$ and $w(t)\in\KK((t))$ are such that $\mathrm g(t^m,w(t))=0$, then $f(t^{mr},t^{mu_r}z(t))=0$. \qedhere
\end{enumerate}
\end{proof}

\begin{lema}
Let $m\in\NN^*$. The extension $\KK((t^m))\to \KK((t)) $ is an algebraic extension of degree $m$.
\end{lema}
\begin{proof} 
The field $\KK((t))$ is a $\KK((t^m))$-vector space with basis $\{1,t,\dots,t^{m-1}\}$.
\end{proof}

Let $y(t)\in\KK((t))$ and let $F(t^m,y)\in\KK((t^m))[y]$ be the minimal polynomial of $y$ over $\KK(( t^m))$. By abuse of notation we write $F(x,y)\in\KK((x))[y]$ for $F(t^m,y)$. Then

\begin{enumerate}

\item $F(x,y)$ is a monic irreducible polynomial of $\KK((x))[y]$,

\item $F(t^m,y(t))=0$,

\item for all $g(x,y)\in\KK((x))[y]$, if $g(t^m,y(t))=0$, then $F(x,y)$ divides $g(x,y)$,

\item $\deg_yF(x,y)=[\KK((t^m)) (y(t)):\KK(( t^m))]$,

\item $\deg_yF(x,y)$ divides $m$.
\end{enumerate}

\noindent Write $y(t)=\sum_{p}c_pt^p$. Define the \emph{support} of $y(t)$ to be $\mathrm{Supp}(y(t))=\lbrace p \mid c_p\not=0\rbrace$. 

\begin{proposicion}\label{min-pol}
Let the notations be as above. If $\gcd(m,\mathrm{Supp}(y(t)))=1$, then the following holds.
\begin{enumerate}[(i)]
\item $F(t^m,y)=\prod_{w,w^m=1}(y-y(wt))$, and if $w_1\not=w_2$, $w_1^n=w_2^n=1$, then $y(w_1t)\not=y(w_2t)$.

\item $\deg_yF(x,y)=m$.
\end{enumerate}
\end{proposicion}

\begin{proof} 
Clearly (i) implies (ii). 

To prove (i), note that if $w^m=1$, then $F((wt)^m,y(wt))=0$. Let $w_1\not= w_2$ be such that $w_1^m=w_2^m=1$. We have $y(w_1t)-y(w_2t)=\sum_{p}(w_1^p-w_2^p)c_pt^p$. If $y(w_1t)=y(w_2t)$, then $w_1^p=w_2^p$ for all $p\in\mathrm{Supp}(y(t))$. But $w_1^m=w_2^m$, and $\gcd(m,\mathrm{Supp}(y(t)))=1$, which yields $w_1=w_2$; a contradiction.
\end{proof}

\begin{proposicion}\label{desc-f-min-pol}
Suppose that $f(x,y)$ is irreducible. There exists $y(t)\in\KK((t))$ such that $f(t^n,y(t))=0$. Furthermore,
\begin{enumerate}
\item $f(t^n,y)=\prod_{w^n=1}(y-y(wt))$,

\item if $w_1\not= w_2$, $w_1^n=w_2^n=1$,  then $y(w_1t)\not= y(w_2t)$,

\item $\gcd(n,\mathrm{Supp}(y(t)))=1$.
\end{enumerate}
\end{proposicion}
\begin{proof} 
We know that there exist $m\in \NN$ and $y(t)\in \KK((t)) $ such that $f(t^m,y(t))=0$. Let $m$ be the smallest integer with this property and let $d=\gcd(m,\mathrm{Supp}(y(t)))$. If $d>1$, then $y(t)=z(t^d)$ for some $z(t)\in\KK((t)) $, hence  $f\left((t^{m/d})^d,z(t^d)\right)=0=f\left(t^{m/d},z(t)\right)$, which is a contradiction. The polynomial $f$ is monic and irreducible. Thus $f$ is consequently the minimal polynomial of $y(t)$ over $\KK((t^m))$. In particular $m=n$. This with Proposition \ref{min-pol} completes the proof of the assertion.
\end{proof}

Suppose that $f$ is irreducible and let $y(t)=\sum_{p}c_pt^p$ as above. Let $d_1=n=\deg_y f$ and let $m_1=\inf\lbrace p\in\mathrm{Supp}(y(t))\mid d_1\nmid p\rbrace$, $d_2=$ $\gcd(d_1,m_1)$. Then for all $i\geq 2$, let $m_i=\inf\lbrace p\in\mathrm{Supp}(y(t))\mid d_i\nmid p\rbrace$ and $d_{i+1}=\gcd(d_i,m_i)$. By hypothesis, there exists $h \geq 1$ such that $d_{h+1}=1$. We set $\underline{m}=(m_1,\dots,m_h)$ and $\underline{d}=(d_1,\dots,d_{h+1})$. We also set $e_i=\frac{d_i}{d_{i+1}}$ for all $i\in\{1,\ldots, h\}$. We finally define the sequence $\underline{r}=(r_0,\dots,r_h)$ as follows: $r_0=n, r_1=m_1$ and for all $i\in\{2,\ldots,h\}$,
$$
r_i=r_{i-1}e_{i-1}+m_{i}-m_{i-1}.
$$
The sequence  ${\underline{m}}$ is called the set of \emph{Newton-Puiseux exponents} of $f$. The sequences $\underline{m}, \underline{d}, \underline{r}$ are called the \emph{characteristic sequences associated} with $f$. Note that $d_k=\gcd(r_0,\dots,r_{k-1})$ for all $1\leq k\leq {h+1}$.

\begin{lema}\label{not-in-group}
Let $k\in \{1,\ldots, h\}$ and let $i \in \{1,\ldots, e_k-1\}$. Then $ir_k$ is not in the group generated by $\{r_0,\ldots,r_{k-1}\}$. 
\end{lema}
\begin{proof} 
Assume we can write $ir_k=\sum_{j=1}^{k-1}\theta_jr_j$, for some $\theta_0,\cdots,\theta_{k-1}\in\ZZ$. Since $\gcd(r_0,\cdots,r_{k-1})=d_k$, we get that $d_k$ divides $ir_k$.  Hence $e_k=\frac{d_k}{d_{k+1}}$ divides  $i\frac{r_k}{d_{k+1}}$. But $\gcd\left(e_k,\frac{r_k}{d_{k+1}}\right)=1$ and $i<e_k$. This is a contradiction.
\end{proof}

\begin{lema}\label{ordenes}
Let the notations be as above, in particular $f$ is irreducible and $y(t)\in\KK((t))$ is a root of $f(t^n,y)=0$. 
\begin{enumerate}[(i)]
\item $\mathrm{ord}_t(y(t)-y(wt))\in\lbrace m_1,\dots,m_h\rbrace$.

\item The cardinality of $\lbrace y(wt)\mid \mathrm{ord}_t(y(t)-y(wt)) >m_k\rbrace$ is $d_{k+1}$.

\item The cardinality of $\lbrace y(wt)\mid \mathrm{ord}_t(y(t)-y(wt)) =m_k\rbrace$ is $d_k-d_{k+1}$.

\end{enumerate}
\end{lema}

\begin{proof} 
\begin{enumerate}[(i)]
\item From the expression of $y(t)$, we get $y(t)-y(wt)=\sum_{p}(1-w^p)c_pt^p$. Let $M=\mathrm{ord}_t(y(t)-y(wt))$. It follows that for all $p< M$, $w^p=1$. Hence, if $d=\gcd(n,\lbrace p\in \mathrm{Supp}(y(t))\mid p<M\rbrace)$, then $w^d=1$. But $d=d_k$ for some $1\leq k\leq h$, whence $M=m_{k-1}$.

\item In fact, $\mathrm{ord}_{t}(y(t)-y(wt))>m_k$ if and only if $w^{d_{k+1}}=1$.

\item Observe that $\mathrm{ord}_{t}(y(t)-y(wt))=m_k$ if and only if $\mathrm{ord}_{t}(y(t)-y(wt))>m_{k-1}$ and $\mathrm{ord}_{t}(y(t)-y(wt))\leq m_{k}$. Hence the result follows from (ii). \qedhere
\end{enumerate}
\end{proof}

Let the notations be as above and let $1\leq k\leq h$. Let $\bar{y}(t)=\sum_{p<m_k}c_pt^p$ and let $G_k(x,y)$ be the minimal polynomial of $\bar{y}(t)$ over $\KK(( t^n))$. Since $\gcd(n,\mathrm{Supp}(\bar{y}(t)))=d_k$, the polynomial $G_k$ is a monic irreducible polynomial of degree $\frac{n}{d_k}$ in $y$. Furthermore, if $Y(t)=\bar{y}(t^\frac{1}{d_k})$, then
$$
G_k(t^\frac{n}{d_k},y)=\prod_{v,v^\frac{n}{d_k}=1}(y-Y(vt)).
$$

We call $G_k$ a \emph{$d_k$th pseudo root} of $f$.

\begin{proposicion}
Under the standing hypothesis.
\begin{enumerate}[(i)]
\item The sequence of Newton-Puiseux exponents of $G_k$ is  given by $\left(\frac{m_1}{d_{k}},\cdots,\frac{m_{k-1}}{d_{k}}\right)$.

\item The $\underline{r}$-sequence and $\underline{d}$-sequence of $G_k$ are given by $\left(\frac{r_0}{d_{k}},\cdots,\frac{r_{k-1}}{d_{k}}\right)$ and $\left(\frac{d_0}{d_{k}},\cdots,\frac{d_{k-1}}{d_{k}},1\right)$, respectively.

\end{enumerate}
\end{proposicion}
\begin{proof} 
\begin{enumerate}[(i)]
\item This follows from the expression of $Y(t)$, using the fact that $\gcd(\frac{n}{d_{k}},\cdots,\frac{m_{k-1}}{d_{k}})=1$.

\item Let $\underline{R}$,  $\underline{D}$, $\underline{E}$  be the characteristic sequences associated with $G_k$. We have $D_1=R_0=$$\deg_yG_k=\frac{n}{d_k}=\frac{r_0}{d_k}=\frac{d_1}{d_k}$,  $R_1=\frac{m_1}{d_k}=\frac{r_1}{d_k}$ and $D_2=\gcd(\frac{r_0}{d_k}, \frac{r_1}{d_k})=\frac{d_2}{d_k}$.  Hence $E_1=e_1$. Now 
$R_2=R_1E_1+\frac{m_{2}}{d_{k}}-\frac{m_{1}}{d_{k}}$, hence $R_2=\frac{r_2}{d_k}$. The assertion now follows by an easy induction on $i\leq k-1$. \qedhere
\end{enumerate}
\end{proof}

 Let $g$ be a nonzero polynomial of $\KK((x))[y]$. We define the \emph{intersection multiplicity} of $f$ with $g$, denoted $\mathrm{int}(f,g)$, to be $\mathrm{int}(f,g)=\mathrm{ord}_tg(t^n,y(t))$.  Note that this definition does not depend on the root $y(t)$, that is, $\mathrm{int}(f,g)=\mathrm{ord}_tg(t^n,y(wt))$ for all $w\in\KK$ such that $w^n=1$.

\begin{proposicion}\label{int-pr}
Let the notations be as above. We have $\mathrm{int}(f,G_k)=r_k$.
\end{proposicion}
\begin{proof} 
From Proposition \ref{desc-f-min-pol}, we can write
$$
f(t^n,\bar{y}(t^{d_k}))=\prod_{w^n=1}(\bar{y}(t^{d_k})-y(wt)).
$$
As in the proof of Lemma \ref{ordenes}, we deduce
\[\mathrm{ord}_t(\bar{y}(t^{d_k})-y(wt))= \begin{cases} m_i &\mbox{if } \mathrm{ord}_t(y(t)-y(wt))=m_i<m_k, \\ 
m_k & \mbox{if }  \mathrm{ord}_t(y(t)-y(wt))\ge m_k. \end{cases}\]
Hence $\mathrm{ord}_tf(t^n,\bar{y}(t^{d_k}))=\sum_{i=1}^{k-1}(d_i-d_{i+1})m_i+m_kd_k$.  Now

\begin{align*}
r_kd_k&=r_{k-1}d_{k-1}+(m_k-m_{k-1})d_k\\
      &= r_{k-2}d_{k-2}+(m_{k-1}-m_{k-2})d_{k-1}+(m_k-m_{k-1})d_k\\
      &\ldots\\
      &= r_1d_1-m_1d_2+\sum_{i=2}^{k-1}(d_i-d_{i+1})m_i+m_kd_k\\
	  &=\sum_{i=1}^{k-1}(d_i-d_{i+1})m_i+m_kd_k.
\end{align*}
Hence $\mathrm{ord}_tf(t^n,\bar{y}(t^{d_k}))=r_kd_k$. In particular $\mathrm{int}(f,G_k)=\mathrm{ord}_tf\left(t^{n/d_k}, \bar{y}(t^{1/d_k})\right)=r_k$.
\end{proof}

Let $G_1,\dots,G_h$ be the set of pseudo-roots of $f$ constructed above and recall that for all $k\in \{1,\ldots,h\}$, $G_k$ is a monic irreducible polynomial of degree $\frac{n}{d_k}$ in $y$. Recall also that the set of characteristic sequences associated with $G_k$ are given by $\left(\frac{m_1}{d_{k}},\dots,\frac{m_{k-1}}{d_k}\right)$, $\left(\frac{d_1}{ d_{k}},\dots,\frac{d_{k-1}}{d_k},1\right)$ and  $\left(\frac{r_0}{ d_{k}}, \dots,\frac{r_{k-1}}{d_k}\right)$.

\begin{proposicion}\label{expansion}
Let $g\in\KK((x)) [y]$. Then
$$
g=\sum_{\theta}c_{\theta}(x)G_1^{\theta_1}\dots G_h^{\theta_h}f^{\theta_{h+1}}
$$
for some $\theta=(\theta_1,\ldots,\theta_{h+1})\in \mathbb N^{h+1}$, with $0\leq \theta_k <e_k$ for all $k\in\{1,\ldots,h\}$, and some $c_\theta(x)\in \KK((x))$. We call this expression the \emph{expansion of $g$ with respect to} $(G_1,\dots,G_h,f)$.
\end{proposicion}
\begin{proof} Write $g=Qf+R$ where $\deg_yR<n$. If $\deg_yQ\geq n$, then write $Q=Q^1f+R^1$ with $\deg_yR^1<n$. We have $g=Q^1f^2+R^1f+R$, then we restart with $Q^1$. This process will stop giving the following expression of $g$ in terms of the powers of $f$:
$$
g=\sum_{k=0}^l \alpha_k(x,y)f^k,
$$
where $\deg_y\alpha_k(x,y) <n$ for all $k\in\{0,\ldots,l\}$. Fix $k\in\{0,\ldots,l\}$ and write the expression of $\alpha_k$ in terms of the powers of $G_h$:
$$
\alpha_k=\sum_{i=0}^{l_k}\alpha_i^kG_h^i,
$$
with $\deg_y\alpha_i^k<\frac{n}{d_h}$ for all $i\in\{0,\ldots,l_k\}$. Note that, since $\deg_y\alpha_k < n$, we have $i < e_h=d_h$. Finally we get
$$
g=\sum_{\theta}c_{\theta}(x,y)G_h^{\theta_h}f^{\theta_{h+1}},
$$
with $\theta_h< e_h$ for all $\theta=(\theta_h,\theta_{h+1})\in\mathbb N^2$ such that $c_{\theta}\not= 0$. Now we restart with the set of polynomials $c_{\theta}(x,y)$ and $G_{h-1}$. We get the result by induction on $k\le h$.
\end{proof}

\begin{proposicion}\label{exp-int-f-g}
Let $g\in\KK((x))[y]$. If $f\nmid g$, then there exist unique $\lambda_0\in\ZZ, \lambda_1,\dots, \lambda_h\in\NN$ such that $\mathrm{int}(f,g)=\sum_{k=0}^h\lambda_kr_k$ and for all $k\in\{1,\ldots, h\}$, $\lambda_k<e_k$. 
\end{proposicion}
\begin{proof} 
Let $g=\sum_{\theta}c_{\theta}(x)G_1^{\theta_1}\cdots G_h^{\theta_h}f^{\theta_{h+1}}$ be the expansion of $g$ with respect to $(G_1,\dots,G_h,f)$. Notice that $\theta_k<e_k$ for all $k\in\{1,\ldots,h\}$ (Proposition \ref{expansion}). Given a monomial $c_{\theta}(x)G_1^{\theta_1}\cdots G_h^{\theta_h}f^{\theta_{h+1}}$ of $g$, if $\theta_{h+1}=0$ and $\theta_0=\mathrm{ord}_xc_{\theta}(x)$, then \[\mathrm{int}(f,c_{\theta}(x)G_1^{\theta_1}\cdots G_h^{\theta_h})=\sum_{k=0}^h\theta_kr_k.\] Let  $c_{\alpha}(x)G_1^{\alpha_1}\cdots G_h^{\alpha_h}$ and $c_{\beta}(x)G_1^{\beta_1}\cdots G_h^{\beta_h}$ be two monomials of $g$, and let $\alpha_0$ and $\beta_0$ be the orders in $x$ of $c_{\alpha}(x)$ and $c_{\beta}(x)$, respectively. Assume that  $\sum_{k=0}^h\alpha_kr_k=\sum_{k=0}^h\beta_kr_k$ and let $j$ be the greatest integer such that $\alpha_j\not=\beta_j$. Suppose that $j\geq 0$ and that $\alpha_j>\beta_j$. We have
 $$
 (\alpha_j-\beta_j)r_j=\sum_{k=0}^{j-1}(\beta_k-\alpha_k)r_k
 $$
with $0<\alpha_j-\beta_j<e_j$. This contradicts Lemma \ref{not-in-group}. Finally either $f\mid g$ or there is a unique monomial $c_{\theta^0}(x)G_1^{\theta_1^0}\cdots G_h^{\theta_h^0}$ such that 
\[
\mathrm{int}(f,g)=\sum_{k=0}^h\theta_k^0r_k=\inf\lbrace \mathrm{int}(f,c_{\theta}(x)G_1^{\theta_1}\cdots G_h^{\theta_h}), c_{\theta}\not=0\rbrace.\qedhere
\]
\end{proof}

\begin{corolario}\label{trunc-comb}
Let $g\in\KK((x)) [y]$. If $\deg_yg<\frac{n}{d_{k+1}}$ for some $k\in\{1,\ldots,h\}$, then there exist $\lambda_0\in\ZZ,\lambda_1,\dots, \lambda_k\in\NN$ such that $\mathrm{int}(f,g)=\sum_{i=0}^k\lambda_ir_i$.
\end{corolario}
\begin{proof} 
Let $g=\sum_{\theta}c_{\theta}(x)G_1^{\theta_1}\dots G_h^{\theta_h}f^{\theta_{h+1}}$ be the expansion of $g$ with respect to $(G_1,\dots,G_h,f)$. Since $\deg_yg<\frac{n}{d_{k+1}}$, we deduce that for all  nonzero monomial $c_{\theta}(x)G_1^{\theta_1}\dots G_h^{\theta_h}f^{\theta_{h+1}}$, $\theta_{k+1}=\dots=\theta_{h+1}=0$. This implies the result.
\end{proof}

More generally, let $k\in\{1,\ldots, h\}$ and define a \emph{$d_k$th pseudo root} of $f$ to be any monic polynomial $G_k$ of degree $\frac{n}{d_k}$ in $y$ such that $\mathrm{int}(f,G_k)=r_k$.  The following proposition shows that such a polynomial is necessarily irreducible.

\begin{proposicion}\label{irreducible-root}
Let $k\in\{1,\ldots,h\}$ and  let $H$ be a monic polynomial of $\KK((x))[y]$, of degree $\frac{n}{d_k}$ in $y$. If $\mathrm{int}(f,H)=r_k$, then $H$ is irreducible. 
\end{proposicion}

\begin{proof} Let $H=H_1\cdots H_s$ be the decomposition of $G$ into irreducible components in $\KK((x))[y]$. Suppose that $s>1$ and let $i\in\{1,\ldots,s\}$. Since  $\deg_yH_i<\frac{n}{d_k}$, by Corollary \ref{trunc-comb}, there exist $\lambda_0^i\in\ZZ, \lambda_1^i,\dots,\lambda_{k-1}^i$ such that 
$\mathrm{int}(f,H_i)=\lambda_0^ir_0+\cdots+\lambda_{k-1}^ir_{k-1}$. Hence $r_k=(\sum_{i=1}^{s}\lambda_0^i)r_0+\cdots+(\sum_{i=1}^{s}\lambda_{k-1}^i)r_{k-1}$. This contradicts Lemma \ref{not-in-group}. 
\end{proof}

\begin{lema}\label{ekrk}
Let the notations be as above.  For all $1\leq k\leq h$, there exist  $\lambda_0^k,\dots,\lambda_{k-1}^k\in\NN$ such that $e_kr_k=\sum_{i=1}^{k-1}\lambda_i^kr_i$.
\end{lema}
\begin{proof} 
Let $G_h$ be a $d_h$th pseudo root of $f$. Write $f=G_h^{d_h}+\alpha_1(x,y)G_h^{d_h-1}+\dots+\alpha_{d_h}(x,y)$. For all $k\in\{0,\ldots, d_h\}$, $\mathrm{int}(f,\alpha_k(x,y)G_h^{d_h-k})=\mathrm{int}(f,\alpha_k(x,y))+(d_h-k)r_h$ (where $\alpha_0(x,y)=1$). But $f(t^n,y(t))=0$, and by Corollary \ref{trunc-comb}, for all $k\in\{1,\ldots,d_h\}$, there exist $\alpha_0^k,\dots,\alpha_{h-1}^k$ such that $\mathrm{int}(f,\alpha_k(x,y))=\sum_{i=1}^{h-1}\alpha_i^kr_i$. Now a similar argument as in  Proposition \ref{exp-int-f-g} shows that if $0\leq k_1\not= k_2\leq d_h-1$, then $\mathrm{int}(f, \alpha_{k_1}G_h^{d_h-k_1})\not= \mathrm{int}(f,\alpha_{k_2}G_h^{d_h-k_2})$. The same argument shows also that for all $i\in\{1,\ldots, d_h-1\}$, if $\alpha_i(x,y)\not=0$, then $\mathrm{int}(f,\alpha_i(x,y)G_h^{d_h-i})\not=\mathrm{int}(f,\alpha_{d_h}(x,y))$. Let $E=\lbrace \mathrm{int}(f,\alpha_k(x,y))+(d_h-k)r_h\mid k\in\{0,\ldots,  d_h\}\rbrace$ and let $k_0\in\{0,\ldots,d_h\}$ be such that 
$\mathrm{int}(f,\alpha_{k_0}(x,y))+(d_h-k_0)r_h=\mathrm{inf}(E)$. If $k_0$ is unique with this property, then   $\mathrm{ord}_tf(t^n,y(t))=\mathrm{int}(f,\alpha_{k_0}(x,y))+(d_h-k_0)r_h$, which is a contradiction because $f(t^n,y(t))=0$. Hence there is at least one $k_1\not= k_0$ such that $\mathrm{int}(f,\alpha_{k_0}(x,y))+(d_h-k_0)r_h=\mathrm{int}(f,\alpha_{k_1}(x,y))+(d_h-k_1)r_h$. This is possible only for $\lbrace k_0,k_1\rbrace=\lbrace 0,d_h\rbrace$, in particular $\mathrm{int}(f,G_h^{d_h})=\mathrm{int}(f,\alpha_{d_h}(x,y))$. This proves the result for $k=h$. Now we use an induction on $1\leq k\leq h$.
\end{proof}

\begin{proposicion}\label{exp-int-from-roots}
Let the notations be as above. Let $g$ be a nonzero  polynomial of $\KK((x))[y]$ and let $G_1,\ldots,G_{h}$ be a set of $d_1,\ldots,d_h$th pseudo roots of $f$. If $\deg_yg<\frac{n}{d_{k+1}}$ for some $k\in \{0,\ldots,  h-1\}$, then $\mathrm{int}(f,g)=d_{k+1}\mathrm{int}(G_{k+1},g)$.
\end{proposicion}
\begin{proof} 
Let $g=\sum_{\underline{\theta}}c_{\underline{\theta}}(x)G_1^{\theta_1}\cdots G_k^{\theta_k}$ be the expansion of $g$ with respect to $(G_1,\dots,G_{h},f)$. By Proposition \ref{exp-int-f-g}, there is a unique monomial $c_{\underline{\theta^0}}(x)G_1^{\theta^0_1}\cdots G_k^{\theta^0_k}$ such that
$$
\mathrm{int}(f,g)=\mathrm{int}(f,c_{\underline{\theta^0}}(x)G_1^{\theta^0_1}\cdots G_k^{\theta^0_k})=\inf\lbrace \mathrm{int}(f,c_{\underline{\theta}}(x)G_1^{\theta_1}\cdots G_k^{\theta_k}, c_{\underline{\theta}}\not=0\rbrace.
$$
Now clearly, the expansion of $g$ above is also that of $g$ with respect to $(G_1,\dots,G_{k+1})$. Furthermore, if $c_{\underline{\theta}}(x)\not= 0$ and if $\theta_0=\mathrm{ord}_xc_{\underline{\theta}}(x)$, then $\mathrm{int}(G_{k+1}, c_{\underline{\theta}}(x)G_1^{\theta_1}\dots G_k^{\theta_k})=\sum_{i=0}^k\theta_i\frac{r_i}{d_{k+1}}=\frac{1}{d_{k+1}}\mathrm{int}(f,c_{\underline{\theta}}(x)G_1^{\theta_1}\cdots G_k^{\theta_k})$. This implies the result.
\end{proof}

\begin{proposicion} \label{pseudo-roots-of-pr}
Let $(G_1,\dots,G_h)$ be a set of pseudo-roots of $f$. For all $k\in\{1,\ldots,h-1\}$, $(G_1,\dots,G_k)$ is a set of pseudo roots of $G_{k+1}$.
\end{proposicion}
\begin{proof} 
Fix $k\in\{1,\ldots, h-1\}$ and let $i\in\{1,\ldots, k\}$. By Proposition \ref{exp-int-from-roots}, \[\mathrm{int}(G_{k+1}, G_i)=\frac{1}{d_{k+1}}\mathrm{int}(f,G_i)=\frac{r_i}{d_{k+1}}.\] 
Furthermore, $G_i$ is irreducible by Proposition \ref{irreducible-root}. This proves the result.
\end{proof}

Let $(G_1,\cdots,G_h)$ be a set of pseudo-roots of $f$. Let $k\in\{1,\ldots, h\}$ and write
\[
f=G_k^{d_k}+\alpha_1(x,y)G_k^{d_k-1}+\dots+\alpha_{d_k}(x,y),
\]
where $\deg_y\alpha_i(x,y) <\frac{n}{d_k}$ for all $i\in\{1,\ldots,d_k\}$. Let $G'_k=G_k+\frac{\alpha_1}{d_k}$. We call $G'_k$ the \emph{Tchirnhausen transform} of $G_k$ and denote it by $\mathrm T(G_k)$. With these notations we have the following.

\begin{proposicion}\label{int-t}
$\mathrm{int}(f,\mathrm T(G_k))=r_k$.
\end{proposicion}
\begin{proof} Let $k=h$ and let
$$
f=G_h^{d_h}+\alpha_1(x,y)G_h^{d_h-1}+\cdots+\alpha_{d_h}(x,y)
$$
be the $G_h$-adic expansion of $f$. Hence $\mathrm{int}(f,G_h^{d_h})=r_hd_h$ and $\mathrm{int}(f,\alpha_i(x,y)G_h^{d_h-i}=\mathrm{int}(f,\alpha_i(x,y)+(d_h-i)r_h$ for all $i$ such that $\alpha_i(x,y)\not=0$. Let $i,j\in\{0,\ldots,d_h-1\}$, $i\not=j$, and assume that $\alpha_i(x,y)\not=0\not=\alpha_j(x,y)$. By a similar argument as in Lemma \ref{ekrk}, we have   $\mathrm{int}(f,\alpha_i(x,y))+(d_h-i)r_h\not=\mathrm{int}(f,\alpha_j(x,y))+(d_h-j)r_h$.  Also if $\alpha_i(x,y)\not=0$ for some $i\in\{1,\ldots, d_h-1\}$, then $\mathrm{int}(f,\alpha_i(x,y)G_h^{d_h-i})\not=\mathrm{int}(f,\alpha_{d_h}(x,y)$. Now $f(t^n,y(t))=0$. This implies 
\begin{enumerate}

\item $\mathrm{int}(f,\alpha_{d_h})=r_hd_h=\mathrm{int}(f,G_h^{d_h})$,

\item $\mathrm{int}(f,\alpha_i(x,y))>ir_h$ for all $1\leq i\leq d_h-1$ such that $\alpha_i(x,y)\not=0$.
\end{enumerate}
It follows that $\mathrm{int}(f,\alpha_1(x,y)>r_h$, hence $\mathrm{int}(f,T(G_h))=\mathrm{int}(f,G_h+\frac{\alpha_1(x,y)}{d_h})=\mathrm{int}(f,G_h)=r_h$.

Let $k<h$ and let 
$$
f=G_{k+1}^{d_{k+1}}+\alpha_1(x,y)G_{k+1}^{d_{k+1}-1}+\cdots+\alpha_{d_{k+1}}(x,y)
$$
be the $G_{k+1}$-adic expansion of $f$. Let also
$$
G_{k+1}=G_k^{e_k}+\beta_1(x,y)G_k^{e_k-1}+\cdots+\beta_{e_k}(x,y)
$$
be the $G_k$-adic expansion of $G_{k+1}$. Easy calculations show that $\alpha_1(x,y)=d_{h+1}\beta_1(x,y)$.  Repeating the argument above for $(G_{k+1},G_k)$ instead of $(f,G_h)$, we prove that $\mathrm{int}(G_{k+1},\beta_1(x,y))> \mathrm{int}(G_{k+1},G_k)=\frac{r_k}{d_{k+1}}$, hence, by Proposition \ref{exp-int-from-roots}, 
$$
\mathrm{int}(f,\alpha_1(x,y))=\mathrm{int} (f,\beta_1(x,y))=d_{k+1}\mathrm{int}(G_{k+1},\beta_1(x,y))>r_k
$$
In particular $\mathrm{int}(f,T(G_k))=\mathrm{int}(f,G_k+\frac{\alpha_1(x,y)}{d_k})=\mathrm{int}(f,G_k)=r_k$.
\end{proof}

\begin{corolario} \label{t-is-pr}
Let $k\in\{1,\ldots,h\}$ and let $G_k$ be a $d_k$th pseudo root of $f$. Then $T(G_k)$ is a $d_k$th pseudo root of $f$.
\end{corolario}
\begin{proof} 
Clearly $T(G_k)$ is a monic polynomial of degree $\frac{n}{d_k}$ in $y$. By Proposition \ref{int-t}, $\mathrm{int}(f,T(G_k))=r_k$, and by Proposition \ref{irreducible-root}, $T(G_k)$ is irreducible. This proves the result.
\end{proof}

Let $d$ be a divisor of $n$ and let $g$ be a monic polynomial of $f$ of degree $\frac{n}{d}$ in $y$. Let
$$
f=g^d+\alpha_1(x,y)g^{d-1}+\dots+\alpha_d(x,y)
$$
be the $g$-adic expansion of $f$. We say that $g$ is a \emph{$d$th approximate root} of $f$ if $\alpha_1(x,y)=0$.

\begin{lema}\label{app-root-unique}
Let the notations be as above. A $d$th approximate root of $f$ exists and it is unique.
\end{lema} 
\begin{proof} 
Let $G=y^\frac{n}{d}$ and let $f=G^d+\alpha_1(x,y)G^{d-1}+\dots+\alpha_d(x,y)$ be the $G$-adic expansion of $f$. If $\alpha_1(x,y)=0$, then $G$ is a $d$th approximate root of $f$. If $\alpha_1(x,y)\not=0$, then we set $G_1=T(G)=G+\frac{\alpha_1(x,y)}{d}$. Let  $f=G_1^d+\alpha^1_1(x,y)G_1^{d-1}+\cdots+\alpha^1_d(x,y)$ be the $G_1$-adic expansion of $f$. Easy calculations show that if $\alpha^1_1(x,y)\not= 0$, then $\deg_y\alpha^1_1(x,y)<\mathrm{deg}_y\alpha_1(x,y)$. In this case we restart with $f$  and $G_2=T(G_1)$. Clearly there exists $k$ such that if $f=G_k^d+\alpha^k_1(x,y)G_k^{d-1}+\cdots+\alpha^k_d(x,y)$ is the $G_k$-adic expansion of $f$, then $\alpha^k_1(x,y)=0$. Hence $G_k$ is a $d$-th approximate root of $f$.

Let $G,H$ be two $d$th approximate roots, and let $f=G^d+\alpha_2(x,y)G^{d-2}+\dots+\alpha_d(x,y)$ and $f=H^d+\beta_2(x,y)H^{d-2}+\dots+\beta_d(x,y)$) be the $G$-adic and $H$-adic expansion of $f$, respectively. We have $G^d-H^d= (G-H)(G^{d-1}+HG^{d-2}+\dots+H^{d-1})=\beta_2(x,y)H^{d-2}+\dots+\beta_d(x,y)-(\alpha_2(x,y)G^{d-2}+\dots+\alpha_d(x,y))$. If $G\not=H$, then $\mathrm{deg}_y(G-H)\geq 0$, but $\mathrm{deg}_y(G^{d-1}+HG^{d-2}+\dots+H^{d-1})=(d-1)\frac{n}{d}> \mathrm{deg}_y(\beta_2(x,y)H^{d-2}+\dots+\beta_d(x,y)-(\alpha_2(x,y)G^{d-2}+\dots+\alpha_d(x,y)))$. This is a contradiction.
\end{proof}

It results from Lemma \ref{app-root-unique} that, given a divisor $d$ of $n$, a $d$th approximate root exists and it is unique. We denote it by $\mathrm{App}(f;d)$.

\begin{proposicion} \label{int-with-approot}
Let the notations be as above. For all $k\in\{1,\ldots, h\}$, $\mathrm{int}(f,\mathrm{App}(f;d_k))=r_k$.
\end{proposicion}
\begin{proof} Let $1\leq k\leq h$ and let $G_k$ be a $d_k$th pseudo root of $f$. By Proposition \ref{exp-int-from-roots},  $\mathrm{int}(f,G_k)= \mathrm{int}(f,T(G_k))$. But $\mathrm{App}(f,d_k)$ is obtained by applying the operation $T$ finitely many times to $G_k$. Hence the result is a consequence of Proposition \ref{pseudo-roots-of-pr} and Corollary \ref{t-is-pr}.
\end{proof}

\begin{corolario} For all $k\in\{1,\ldots, h\}$, $\mathrm{App}(f,d_k)$ is irreducible. In particular $\mathrm{App}(f,d_k)$ is a $d_k$th pseudo root of $f$.
\end{corolario}
\begin{proof} This results from Propositions \ref{irreducible-root} and \ref{int-with-approot}.
\end{proof}

Next we shall introduce the notion of contact between two irreducible polynomials of $\KK((x))[y]$. The notion tells us how far the parametrizations of these two polynomials are close.

Let $g$ be a monic irreducible polynomial of $\KK((x))[y]$, of degree $p$ in $y$ and let $z_1(t),\cdots,z_p(t)$ be the set of roots of $g(t^p,y)=0$.  We define the \emph{contact} of $f$ with $g$, denoted $c(f,g)$, to be:

$$
\mathrm c(f,g)=\frac{1}{np}\max_{i,j}\mathrm{ord}_t(y_i(t^p)-z_j(t^n))
$$

Note that $\mathrm c(f,g)=\frac{1}{np}\max_{i}\mathrm{ord}_t(y_i(t^p)-z(t^n))=\frac{1}{np}\max_{j}\mathrm{ord}_t(y(t^p)-z_j(t^n))$ where $y(t)$ and $z(t)$ are roots of $f(t^n,y)=0$ and $g(t^p,y)=0$, respectively.


\begin{proposicion} \label{char-contact}
Let $g$ be an irreducible monic polynomial of $\KK((x))[y]$ and let $p=\mathrm{deg}_yg$. We have the following.
\begin{enumerate}
\item $\mathrm c(f,g)<\frac{m_1}{n}$ if and only if $\mathrm{int}(f,g)=np\mathrm{c}(f,g)$.

\item  $\frac{m_k}{n}<c=\mathrm c(f,g)\leq \frac{m_{k+1}}{n}$ for some $k\in\{1,\ldots,h\}$ (with the assumption that $m_{h+1}=+\infty$) if and only if $\mathrm{int}(f,g)=(r_kd_k+nc-m_k)\frac{p}{n}$ 
\end{enumerate}
\end{proposicion}
\begin{proof}  Let $z(t)$ be a root of $g(t^p,y)=0$. We have $\mathrm{int}(f,g)=\mathrm{ord}_tf(t^p,z(t))$. Note that $f(t^p,z(t))=f((t^{\frac{p}{n}})^n,z(t))$. Also, $f(t^n,y)=\prod_{i=1}^n(y-y_i(t))$. Hence $f((t^{\frac{p}{n}})^n,y)=\prod_{i=1}^n(y-y_i(t^{\frac{p}{n}}))$, which implies that 
$f(t^p,z(t))=\prod_{i=1}^n(z(t)-y_i(t^{\frac{p}{n}}))$, and it follows that 
$$
\mathrm{int}(f,g)=\mathrm{ord}_tf(t^p,z(t))=\frac{1}{n}\mathrm{ord}_t\left(\prod_{i=1}^n(z(t^n)-y_i(t^p))\right)=\frac{1}{n}\sum_{i=1}^n\mathrm{ord}_t(z(t^n)-y_i(t^p)).
$$
Suppose, without loss of generality, that $\mathrm c(f,g)= \frac{1}{np}\mathrm{ord}_t(y_1(t^p)-z(t^n))$.  It follows that $\mathrm{int}(f,g)\leq \mathrm{ord}_t(y_1(t^p)-z(t^n))\leq np\mathrm c(f,g)$. 

If $c(f,g)<\frac{m_1}{n}$, then $\mathrm{ord}_t(y_1(t^p)-z(t^n))<m_1p$. Let $2\leq i\leq n$. We have $z(t^n)-y_i(t^p)=z(t^n)-y_1(t^p)+y_1(t^p)-y_i(t^p)$ and by Lemma 69, $\mathrm{ord}_t(y_1(t^p)-y_i(t^p))\geq pm_1$, hence $\mathrm{ord}_t(z(t^n)-y_i(t^p))=\mathrm{ord}_t(z(t^n)-y_1(t^p))$. Finally $\mathrm{ord}_tf(t^p,z(t))=\mathrm{ord}_t(z(t^n)-y_1(t^p))=npc(f,g)$. Conversely, if $\mathrm{int}(f,g)=np\mathrm c(f,g)$, then $\mathrm{ord}_t(y_i(t)-z(t))=\mathrm{ord}_t(y_1(t)-z(t))$ for all $i\in\{2,\ldots,n\}$. This is true only if $\mathrm c(f,g)<\frac{m_1}{n}$. This proves (1).

Suppose that $c(f,g)\geq\frac{m_1}{n}$ and let $k$ be the greatest element such that $\frac{m_k}{n}\leq c(f,g)<\frac{m_{k+1}}{n}$. Let $2\leq i\leq n$. We have $z(t^n)-y_i(t^p)=z(t^n)-y_1(t^p)+y_1(t^p)-y_i(t^p)$. Hence

$$
\mathrm{ord}_t(z(t^n)-y_i(t^p)) = \begin{cases} \mathrm{ord}_t(z(t^n)-y_1(t^p)) &\mbox{if }\mathrm{ord}_t(y_i(t)-y_1(t))>{m_k},\\ 
\mathrm{ord}_t(y_1(t^p)-y_i(t^p)) & \mbox{if } \mathrm{ord}_t(y_i(t)-y_1(t))\leq{m_k}. \end{cases} 
$$
By Lemma \ref{ordenes},  $\mathrm{ord}_tf(t^p,z(t))= \frac{1}{n} \sum_{i=1}^n\mathrm{ord}_t(z(t^n)-y_i(t^p))=
\frac{1}{n}(d_{k+1}\mathrm{ord}_t(z(t^n)-y_1(t^p))+p\sum_{i=1}^k(d_i-d_{i+1})m_i)=\frac{1}{n}(d_{k+1}npc(f,g)+p(r_kd_k-m_kd_{k+1}))=\frac{p}{n}(d_{k+1}nc(f,g))+\frac{p}{n}(r_kd_k-m_kd_{k+1})=
\frac{p}{n}(r_kd_k+(nc(f,g)-m_k)d_{k+1})$.

Conversely, suppose that $\mathrm{int}(f,g)=(r_kd_k+nc-m_k)\frac{p}{n}$ for some $k\geq 1$. If $c<\frac{m_1}{n}$, then $\mathrm{int}(f,g)=npc<np\frac{m_1}{n}=pm_1=pr_1=(r_1d_1)\frac{p}{n}\leq (r_kd_k+nc-m_k)\frac{p}{n}$, which is a contradiction. Hence $c(f,g)\geq \frac{m_1}{n}$, and a similar argument shows that $c=\mathrm c(f,g)$. This proves (2).
\end{proof}

\begin{corolario} Let $k\in\{1,\ldots, h\}$ and let $G_k$ be a pseudo-root of $f$. We have $\mathrm c(f,G_k)=\frac{m_k}{n}$. In particular $\mathrm c(f,\mathrm{App}_{d_k}(f))=\frac{m_k}{n}$.
\end{corolario}
\begin{proof} By Proposition \ref{int-pr}, $\mathrm{int}(f,G_k)=r_k=(r_kd_k+(n\frac{m_k}{n}-m_k)d_{k+1})\frac{n/ d_k}{n}$. Hence $\mathrm c(f,G_k)=m_k$ by Proposition \ref{char-contact}.
\end{proof}

\begin{proposicion} \label{divides}
Let $g$ be a monic irreducible polynomial of $\KK((x))[y]$ of degree $p$ in $y$. If $\mathrm c(f,g)> \frac{m_k}{n}$, for some $k\in\{1,\ldots, h\}$, then $\frac{n}{d_{k+1}}$ divides $p$. In particular, if $\mathrm c(f,g)>\frac{m_h}{n}$, then $n$ divides $p$.
\end{proposicion}
\begin{proof} 
Let $z(t)$ be a root of $g(t^p,y)=0$ and assume, without loss of generality, that $\mathrm c(f,g)=\frac{1}{np}\mathrm{ord}_t(z(t^n)-y_1(t^p))$. Write $y_1(t)=\sum_i a_it^i$ and $z(t)=\sum_jb_jt^j$. The hypothesis implies that for all $i$ in $\mathrm{Supp}(y_1(t))$ with $i\leq m_k$, there exist $j\in \mathrm{Supp}(z(t))$ such that $jn=ip$, that is, $j=i\frac{p}{n}\in\NN$. But $\gcd(i\in\mathrm{Supp}(y_1(t)),i\leq m_k)=d_{k+1}$, whence $d_{k+1}\frac{p}{n}\in\NN$, which implies that $\frac{n}{d_{k+1}}$ divides $n$.
\end{proof}

\begin{proposicion}\label{cond-irred}
Let $g$ be a monic polynomial of $\KK((x))[y]$ and assume that $\deg_yg=n$. If $\mathrm{int}(f,g)>r_hd_h$, then $g$ is irreducible. 
\end{proposicion}
\begin{proof} Let $g=g_1\cdots g_r$ be the decomposition of $g$ into irreducible components in $\KK((x))[y]$. If $r>1$ then, by Proposition \ref{divides}, $\deg_yg_i<n$ for all $1\leq i\leq r$. Thus $\mathrm c(f,g_i)<m_h$ for all $1\leq i\leq r$. By Proposition \ref{char-contact}, $\mathrm{int}(f,g_i)<r_hd_h\frac{\mathrm{deg}_yg_i}{n}$, hence $\mathrm{int}(f,g_i)=\sum_{i=1}^r\mathrm{int}(f,g_i)<r_hd_h$, which is a contradiction. This proves our assertion. 
\end{proof}


\subsection{The local case}

In the following we shall assume that $f(x,y)$ is an irreducible polynomial of $\KK\llbracket x\rrbracket [y]$. Note that in this case, for all $k\in\{1\dots,h\}$, $m_k> 0$ and $G_k\in\KK\llbracket x\rrbracket [y]$ for every $d_k$th pseudo root $G_k$ of $f$.

Let $\mathrm g(x,y)$ be a nonzero element of $\KK\llbracket x\rrbracket [y]$ and recall that the intersection multiplicity of $f$ with $g$, denoted $\mathrm{int}(f,g)$, is defined to be  the order in $t$ of $\mathrm g(t^n,y(t))$. Note that this definition does not depend on the choice of the root $y(t)$ of $f(t^n,y)=0$ and also that $\mathrm{int}(f,g)\geq 0$.

The set $\lbrace \mathrm{int}(f,g):g\in\KK\llbracket x\rrbracket [y]\rbrace$ is a semigroup of $\NN$. We call it the \emph{semigroup of $f$} and denote it by $\Gamma(f)$.

Let $g_k=\mathrm{App}(f,d_k)$ for all $1\leq k\leq h$. Recall that $\mathrm{int}(f,g_k)=r_k$.

\begin{proposicion} Under the standing hypothesis.
\begin{enumerate}[(i)]
\item  The semigroup $\Gamma(f)$ is generated by $r_0,\dots,r_h$. 

\item $\Gamma(f)$ is a numerical semigroup.

\item $\Gamma(f)$ is free for the arrangement $(r_0,\dots,r_h)$.

\item For all $k\in\{1,\ldots, h\}$, $r_kd_k< r_{k+1}d_{k+1}$.
\end{enumerate}
\end{proposicion}

\begin{proof}
\begin{enumerate}[(i)]
\item Follows from Proposition 73.

\item Follows from the fact that $d_{h+1}=\gcd(r_0,\cdots,r_h)=1$.

\item Is a consequence of Lemma 68 and Lemma 76.

\item For all $k\in\{1,\ldots, h\}$, we have $r_{k+1}d_{k+1}=r_kd_k+(m_{k+1}-m_k)d_{k+1}$ and $m_k<m_{k+1}$, whence $r_kd_k< r_{k+1}d_{k+1}$.\qedhere
\end{enumerate} 
\end{proof}

Conversely we have the following.

\begin{proposicion}
Let $r_0<r_1<\dots <r_h$ be a sequence of nonzero elements of $\NN$ and let $d_1=r_0$ and $d_{k+1}=$ $\gcd(r_{k},d_{k})$ for all $1\leq k\leq h$. Assume that the following conditions hold:

\begin{enumerate}

\item $d_{h+1}=1$,

\item for all $k\in\{1,\ldots, h\}$, $r_kd_k< r_{k+1}d_{k+1}$,

\item the semigroup $\Gamma =\langle r_0,\dots,r_h\rangle$ is free for the arrangement $(r_0,\dots,r_h)$.
\end{enumerate}
Then there exists a monic irreducible polynomial $f(x,y)\in\KK\llbracket x\rrbracket [y]$ of degree $r_0$ in $y$ such that $\Gamma(f)=\Gamma$.
\end{proposicion}
\begin{proof} Let $r_0=n$ and $m_1=r_1$, and for all $1\leq k\leq h$ let $m_{k+1}=r_{k+1}-r_k\frac{d_k}{d_{k+1}}+m_k$. Finally let $y(t)=t^{m_1}+t^{m_2}+\dots+t^{m_h}\in\KK\llbracket t\rrbracket$. Let $f(x,y)$ is the minimal polynomial of $y(t)$ over $\KK((t^{n}))$. We have

$$
f(x,y)=\prod_{w^n=1}(y-y(wt)). 
$$

Now $\mathrm{Supp}(y(t))=\lbrace m_1,\ldots,m_h\rbrace$, hence $\underline{m}=(m_1,\ldots,m_h)$ is nothing but the sequence of Newton-Puiseux exponents of $f$, and consequently $\Gamma(f)=\langle r_0,\ldots,r_h\rangle$.
\end{proof}

Let $f_x,f_y$ denote the partial derivatives of $f$. Let $H$ be an irreducible component of 
$f_y$. Let $\deg_yH=n_H$ and write $H(t^{n_H},y)=\prod_{i=1}^{n_H}(y-z_i(t))$. By the chain rule of derivatives we have:
$$
\frac{d}{dt}f(t^{n_H},z_1(t))=\frac{df}{dx}(t^{n_H},z_1(t))(n_Ht^{n_H-1})+\frac{df}{dy}(t^{n_H},z_1(t))(z_1'(t))=\frac{df}{dx}(t^{n_H},z_1(t))(n_Ht^{n_H-1}).
$$
It follows that $\mathrm{int}(f,H)-1=\mathrm{int}(f_x,H)+n_H-1$. Adding this equality over the set of irreducible components of $f_y$, we get that
$$
\mathrm{int}(f,f_y)=\mathrm{int}(f_x,f_y)+n-1.
$$
Write $f(t^n,y)=\prod_{i=1}^n(y-y_i(t))$. We have  $f_y(t^n,y)=\sum_{i=1}^n\prod_{k\not=i}(y-y_k(t))$. Hence $f_y(t^n,y_1(t))=\prod_{k=2}^n(y_1(t)-y_k(t))$, and $\mathrm{int}(f,f_y)=\sum_{k=2}^n\mathrm{ord}_t(y_1(t)-y_k(t))=\sum_{k=1}^h(d_k-d_{k+1})m_ k$ (see Lemma \ref{ordenes}). But $\sum_{k=1}^h(d_k-d_{k+1})m_ k=r_hd_h-m_h=\sum_{k=1}^h(e_k-1)r_k$. Finally 
$$
\mathrm{int}(f,f_y)=\sum_{k=1}^h(e_k-1)r_k=\mathrm{int}(f_x,f_y)+n-1.
$$
Note that the conductor $\mathrm C(\Gamma(f))=\sum_{k=1}^h(e_k-1)r_k-n+1$. Hence \[\mathrm C(\Gamma(f))=\mathrm{int}(f_x,f_y).\] 








\subsection{The case of curves with one place at infinity}

\medskip

\noindent Let the notations be as above and assume that $f(x,y)\in\KK[x^{-1}][y]$ and also that $f(x,y)$ is irreducible in $\KK((x))[y]$. 

Let $\mathrm g(x,y)$ be a nonzero element of $\KK[x^{-1}][y]$. As above, we define the intersection multiplicity of $f$ with $g$, denoted $\mathrm{int}(f,g)$, to be the order in $t$ of $\mathrm g(t^n,y(t))$. Note that this definition does not depend on the choice of the root $y(t)$ of $f(t^n,y)=0$.

The set $\lbrace \mathrm{int}(f,g):g\in\KK[x^{-1}][y]\rbrace$ is a semigroup of $-\NN$. We call it the \emph{semigroup of $f$} and we denote it by $\Gamma(f)$. For all $k\in \{1,\ldots,h\}$, let $g_k=\mathrm{App}_{d_k}(f)$, and let $-r_k=\mathrm{int}(f,g_k)$. 

\begin{proposicion}Under the standing hypothesis.
\begin{enumerate}[(i)]
\item  The semigroup $\Gamma(f)$ is generated by $-r_0,\dots,-r_h$. 

\item $-\Gamma(f)=\lbrace -r\mid r\in \Gamma(f)\rbrace$ is a numerical semigroup.

\item $-\Gamma(f)$ is free for the arrangement $(r_0,\dots,r_h)$.

\item For all $1\leq k\leq h, -r_kd_k< -r_{k+1}d_{k+1}$.
\end{enumerate} 
\end{proposicion}
\begin{proof} 
(i) follows from Proposition \ref{exp-int-f-g}, while (ii) holds because $d_{h+1}=1$. 

(iii) is a consequence of Lemma \ref{ekrk}.

Finally, for all $k\in\{1,\ldots, h\}$, $-r_{k+1}d_{k+1}=-r_kd_k+(m_{k+1}-m_k)d_{k+1}$ and $m_k<m_{k+1}$. Hence $-r_kd_k< -r_{k+1}d_{k+1}$ and (iv) follows.
\end{proof}

\noindent Let $f_x,f_y$ denote the partial derivatives of $f$. Let $H$ be an irreducible component of 
$f_y$. Let $\deg_yH=n_H$ and write $H(t^{n_H},y)=\prod_{i=1}^{n_H}(y-z_i(t))$. Arguing as above, by the chain rule of derivatives we have:
$$
\frac{d}{dt}f(t^{n_H},z_1(t))=\frac{df}{dx}(t^{n_H},z_1(t))(n_Ht^{n_H-1})+\frac{df}{dy}(t^{n_H},z_1(t))(z_1'(t))=\frac{df}{dx}(t^{n_H},z_1(t))(n_Ht^{n_H-1}).
$$
It follows that $\mathrm{int}(f,H)-1=\mathrm{int}(f_x,H)+n_H-1$. Adding this equality over the set of irreducible components of $f_y$, we get that
$$
\mathrm{int}(f,f_y)=\mathrm{int}(f_x,f_y)+n-1.
$$
Now write $f(t^n,y)=\prod_{i=1}^n(y-y_i(t))$. We have  $f_y(t^n,y)=\sum_{i=1}^n\prod_{k\not=i}(y-y_k(t))$. Hence $f_y(t^n,y_1(t))=\prod_{k=2}^n(y_1(t)-y_k(t))$, and $\mathrm{int}(f,f_y)=\sum_{k=2}^n\mathrm{ord}_t(y_1(t)-y_k(t))=\sum_{k=1}^h(d_k-d_{k+1})m_ k$ (see Lemma \ref{ordenes}). But
$\sum_{k=1}^h(d_k-d_{k+1})m_ k=(-r_h)d_h-m_h=\sum_{k=1}^h(e_k-1)(-r_k)$. Finally 
$$
\mathrm{int}(f,f_y)=\sum_{k=1}^h(e_k-1)(-r_k)=\mathrm{int}(f_x,f_y)+n-1.
$$
Let $F=y^n+a_1(x)y^{n-1}+\cdots+a_n(x)$ be a nonzero polynomial of $\KK[x][y]$ and assume, possibly after a change of variables, that $\deg_xa_i(x)<i$ for all $i\in\{1,\ldots,n\}$ such that $a_i(x)\not=0$. Let $C=\mathrm V(F)$ be the algebraic curve $F=0$ and let $h_F(u,x,y)=u^nF(\frac{x}{u},\frac{y}{u})$. The projective curve $\mathrm V(h_F)$ is the projective closure of $C$ in $\PP_{\KK}^2$. By hypothesis, $(0,1,0)$ is the unique point of $\mathrm V(h_F)$ at the line at infinity $u=0$. We say that $F$ \emph{has one place at infinity} if $h_F$ is analitycally irreducible at $(0,1,0)$. Set  $F_{\infty}(u,y)=h_F(u,1,y)$. Then $F$ has one place at infinity if and only if the formal power series $F_{\infty}(u,y)$ is irreducible in $\KK\llbracket u\rrbracket[y]$.

\begin{lema} \label{f-F}
Let the notations be as above. Let $f(x,y)=F(x^{-1},y)\in\KK[x^{-1},y]$.
\begin{enumerate}
\item $f(x,x^{-1}y)=x^{-n}F_{\infty}(x,y)$.

\item $F$ has one place at infinity if and only if $f(x,y)$ is irreducible in $\KK((x))[y]$.
\end{enumerate}
\end{lema}
\begin{proof} 
Write $F(x,y)=y^n+\sum_{i+j<n}a_{ij}x^iy^j$. We have $f(x,y)=y^n+\sum_{i+j<n}a_{ij}x^{-i}y^j$. Hence $f(x,x^{-1}y)=x^{-n}y^n+\sum_{i+j<n}a_{ij}x^{-i-j}y^j=
x^{-n}(y^n+\sum_{i+j<n}a_{ij}x^{n-i-j}y^j)=x^{-n}F_{\infty}(x,y)$. This proves (1).

Now suppose that $f$ is irreducible in $\KK((x))[y]$. If  $F_{\infty}(x,y)$ is not irreducible in $\KK\llbracket x\rrbracket[y]$, then $F=F_1F_2, F_1,F_2\in\KK\llbracket x\rrbracket[y]$ and $\deg_yF_i=n_i>0$. But $\tilde{f}(x,y)=f(x,x^{-1}y)=x^{-n}F_{\infty}(x,y)=x^{-n_1}F_1x^{-n_2}F_2$ and $f(x,y)=\tilde{f}(x,xy)= x^{-n_1}F_1(x,xy)x^{-n_2}F_2(x,y)=f_1(x,y)f_2(x,y)$ with $f_1,f_2\in\KK((x))[y]$ and $\deg_yf_1=n_1$, $\deg_yf_2=n_2$. This is a contradiction. A similar argument proves the converse.
\end{proof}

Assume that $F(x,y)$ has one place at infinity. Let $G(x,y)$ be a nonzero
element of $\KK[x,y]$ and denote by $\mathrm{Int}(F,G)$ the rank of the
$\KK$-vector space ${\KK[x,y]}/{(F,G)}$. After possibly a change of
variables ($y=Y^q-x$, $x=Y$, $q\gg 0$, for example), we may assume that
$G(x,y)=y^p+\sum_{i+j<p}b_{ij}x^iy^j$.

\begin{proposicion}\label{Int-int}
Let the notations be as above, in particular
$G(x,y)=y^p+\sum_{i+j<p}b_{ij}x^iy^j\in\KK[x,y]$. Let 
$f(x,y)=F(x^{-1},y)$ and $g(x,y)=G(x^{-1},y)$. We have
$\mathrm{Int}(F,G)=-\mathrm{int}(f,g)$.
\end{proposicion}
\begin{proof}
Let $y(t)$ be a root of $f(t^n,y)=0$. We have
$\mathrm{int}(f,g)=\mathrm{ord}_tg(t^n,y(t))$. Also,
$F_{\infty}(x,y)=x^nf(x,x^{-1}y)$ and $G_{\infty}(x,y)=x^pg(x,x^{-1}y)$.
It follows that
$F_{\infty}(t^n,t^ny(t))=t^{2n}f(t^n,t^{-n}t^ny(t))=t^{2n}f(t^n,y(t))=0$.
Hence $t^ny(t)$ is a root of $F_{\infty}(t^n,y)=0$. Now
\begin{multline*}
\mathrm{int}(F_{\infty},G_{\infty})=\mathrm{ord}_tG_{\infty}(t^n,t^ny(t))=\mathrm{ord}_t(x^pg(x,x^{-1}y))(t^n,t^ny(t))\\
=\mathrm{ord}_t(t^{np})+\mathrm{ord}_tg(t^n,y(t))=np+\mathrm{int}(f,g).
\end{multline*}
Finally  $\mathrm{int}(F_{\infty},G_{\infty})-\mathrm{int}(f,g)=np$. By
B\'ezout's Theorem,
$\mathrm{int}(F_{\infty},G_{\infty})+\mathrm{Int}(F,G)=np$. This implies
that $\mathrm{Int}(F,G)=-\mathrm{int}(f,g)$.
\end{proof}

More generally we have the following:

\begin{proposicion}\label{In-in-general}
Assume that $F(x,y)$ has one place at infinity and let $G(x,y)$ be a
nonzero element of $\KK[x,y]$. Let $f(x,y)=F(x^{-1},y)$ and
$g(x,y)=G(x^{-1},y)$. We have $\mathrm{Int}(F,G)=-\mathrm{int}(f,g)$.
\end{proposicion}

\begin{proof} Let $G(x,y)=G_p+G_{p-1}+\cdots+G_0$ be the decomposition of
$G$ into homogeneous components. Write
$G_p=\prod_{k=1}^s(a_ky+b_kx)^{p_k}$.

i) If for all $1\leq k\leq s, b_k\not=0$, then $F$ and $G$ do not have
common points at infinity. By Bézout theorem, $\mathrm{Int}(F,G)=np$. For
all $0\leq i\leq p$, write $g_{p-i}(x,y)=G_{p-i}(x^{-1},y)$. We have 
$g(x,y)=\sum_{i=0}^pg_{p-i}(x,y)$, and if $y(t)$ is a root of
$f(t^n,y)=0$, then $g_p(t^n,y(t))=\prod_{k=1}^s(a_ky(t)+b_kt^{-n})^{p_k}$,
hence $\mathrm{ord}_tg_p(t^n,y(t))=-np=\mathrm{ord}_t(x^{-p})(t^n,y(t))$.
Furthermore, if
$g_{p-i}(x,y)=\sum_{k+l=p+i}b_{kl}x^{-k}y^{l}$, then
$\mathrm{ord}_tg_{p-i}(t^n,y(t))\geq\mathrm{min}_{k,l}(-kn-lm)\geq
\mathrm{min}_{k,l}(-kn-ln)=-(p-i)n>-pn$ for all $1\leq i\leq p$. It follows
that $\mathrm{Int}(F,G)=-\mathrm{int}(f,g)$.

ii) Suppose that $b_1=0$, and that, without loss of generality,
$a_1=1$. We have  $G_p=y^{p_1}\prod_{k=2}^s(a_ky+b_kx)^{p_k}$ and
$b_k\not=0$ for all $2\leq k\leq s$. Let $P(0:1:0)$ be the unique common
point of $F$ and $G$ at infinity. We have
$\mathrm{int}(F_{\infty},G_{\infty})=\mathrm{int}_P(F,G)$, and by Bézout
theorem, $\mathrm{Int}(F,G)+\mathrm{int}_P(F,G)=np$. Clearly
$F_{\infty}(x,y)$ is the local equation of $F$ at $P$. Let
$g_{p}(x,y)=G_{p}(x^{-1},y)$ and write
$g(x,y)=g_p(x,y)+\sum_{k+l<p}b_{kl}x^{-k}y^l$. We have

$$
g(x,x^{-1}y)=x^{-p_1}y^{p_1}\prod_{k=2}^s(a_kx^{-1}y+b_kx^{-1})^{p_k}+\sum_{k+l<p}b_{kl}x^{-k}x^{-l}y^l
$$
$$
=x^{-p}(y^{p_1}\prod_{k=2}^s(a_ky+b_k)^{p_k}+\sum_{k+l<p}x^{p-k-l}y^l).
$$

Hence the local equation of $G$ at $P$, denoted $G_P$,  is given by
$G_P(x,y)=x^pg(x,x^{-1}y)$. Now the same calculations as in Proposition
\ref{Int-int} show that $\mathrm{int}(F_{\infty},
G_P)=np+\mathrm{int}(f,g)$, hence $\mathrm{Int}(F,G)=-\mathrm{int}(f,g)$.

\end{proof}

\begin{proposicion}\label{translation}
Let the notations be as above. If $F$ has one place at infinity, then so is for    $F(x,y)-\lambda$, for all  $\lambda\in\KK^*$.
\end{proposicion}
\begin{proof} Clearly $\mathrm{Int}(F,F-\lambda)=0=\mathrm{int}(f,f-\lambda)>-r_hd_h$. By Proposition \ref{cond-irred}, $f-\lambda$ is irreducible in $\KK((x))[y]$, hence $F-\lambda$ has one place at infinity by Lemma \ref{f-F}. This proves our assertion. 
\end{proof}

\begin{nota} 
The result of Proposition \ref{translation} is proper to curves with one place at infinity. More precisely, for all  $N>1$, there exist a polynomial $F$ with $N$ places at infinity and $\lambda\in\KK^*$ such that the number of places of $F-\lambda$ at infinity is not equal to $N$.
\end{nota}

Let \[\Gamma_{\infty}(F)=\lbrace \mathrm{Int}(F,G) \mid G\in\KK[x][y]\rbrace.\] Lemma \ref{f-F} and the calculations above imply the following.

\begin{proposicion} \label{Gamma-infty}
Under the standing hypothesis. 
\begin{enumerate}[(i)]
\item  The semigroup $\Gamma_{\infty}(F)$ is generated by $r_0,\dots,r_h$. 

\item  $\Gamma_{\infty}(F)$ is a numerical semigroup.

\item $\Gamma_{\infty}(F)$ is free for the arrangement $(r_0,\dots,r_h)$.

\item  For all $k\in\{1,\ldots, h\}$, $r_kd_k> r_{k+1}d_{k+1}$.

\item  $\mathrm{Int}(F,F_y)=\mathrm{Int}(F_x,F_y)+n-1=\sum_{k=1}^h(e_k-1)r_k$.

\item  The conductor $\mathrm C(\Gamma_{\infty}(F))= \mathrm{Int}(F_x,F_y)=(\sum_{k=1}^h(e_k-1)r_k)-n+1$.
\end{enumerate}
\end{proposicion}

\begin{example}
Sequences fulfilling the Condition (iv) in Proposition \ref{Gamma-infty} are known as $\delta$-sequences.
\begin{verbatim}
gap> DeltaSequencesWithFrobeniusNumber(11);
[ [ 5, 4 ], [ 6, 4, 9 ], [ 7, 3 ], [ 9, 6, 4 ], [ 10, 4, 5 ], [ 13, 2 ] ]
gap> List(last, CurveAssociatedToDeltaSequence);       
[ y^5-x^4, y^6-2*x^2*y^3+x^4-x^3, y^7-x^3, y^9-3*x^2*y^6+3*x^4*y^3-x^6-y^2, 
  y^10-2*x^2*y^5+x^4-x, y^13-x^2 ]
gap> List(last,SemigroupOfValuesOfPlaneCurveWithSinglePlaceAtInfinity);
[ <Modular numerical semigroup satisfying 5x mod 20 <= x >, 
  <Numerical semigroup with 3 generators>, 
  <Modular numerical semigroup satisfying 7x mod 21 <= x >, 
  <Numerical semigroup with 3 generators>, 
  <Numerical semigroup with 3 generators>, 
  <Modular numerical semigroup satisfying 13x mod 26 <= x > ]
gap> List(last,MinimalGeneratingSystemOfNumericalSemigroup);
[ [ 4, 5 ], [ 4, 6, 9 ], [ 3, 7 ], [ 4, 6, 9 ], [ 4, 5 ], [ 2, 13 ] ]
\end{verbatim}
\end{example}

\begin{corolario} \label{rk-equal-dk1}
Let the notations be as above. If $\mathrm{Int}(F_x,F_y)=0$, then $r_k=d_{k+1}$ for all $k\in \{1,\ldots, h\}$. In particular, $\Gamma_{\infty}(F)=\NN$ and $r_1$ divides $n$.
\end{corolario}
\begin{proof} 
For all $1\leq k\leq h, r_k\geq d_{k+1}$. If $r_i>d_{i+1}$ for some $1\leq i\leq h$, we get  $\sum_{k=1}^h(e_k-1)r_k)-n+1> (\sum_{k=1}^h(d_k-d_{k+1}))-n+1=0$, which contradicts the hypothesis. This proves the first assertion. Since $r_h=d_{h+1}=1$, then $\Gamma_{\infty}(F)=\NN$. On the other hand, $r_1=d_2=\gcd(n,r_1)$, whence $r_1$ divides $n$.
\end{proof}

\begin{proposicion} \label{rk-2dk1}
Let the notations be as above. If $d_2<r_1$ (that is, $r_1$ does not divide $n$), then $\mathrm{Int}(F_x,F_y)\geq n-1$ with equality if and only if $r_k=2d_{k+1}$ for all $1\leq k\leq h$. 
\end{proposicion}
\begin{proof} 
We have $\mathrm{Int}(F_x,F_y)+n-1=(\sum_{k=1}^h(e_k-1)r_k)$. But  $r_k\geq 2d_{k+1}$ for all $1\leq k\leq n$. Hence $\sum_{k=1}^h(e_k-1)r_k\geq 2\sum_{k=1}^h(e_k-1)d_{k+1}\geq  2\sum_{k=1}^h(d_k-d_{k+1})=2(n-1)$. In particular. $\mathrm{Int}(F_x,F_y)\geq n-1$. Clearly, if $r_k> 2d_{k+1}$ for some $k\in \{1,\ldots, n\}$, then $\mathrm{Int}(F_x,F_y)> n-1$. This proves our assertion.
\end{proof}

\begin{corolario} 
\begin{enumerate}[(i)]
\item Let $h=1$, then $\mathrm{Int}(F_x,F_y)=n-1$ if and only if $\Gamma_{\infty}(F)=\langle n,2\rangle$.

\item Let $h\geq 2$ and suppose that $d_2<r_1$. If $2$ does not divide $n$, then  $\mathrm{Int}(F_x,F_y)>n-1$.
\end{enumerate}
\end{corolario}
\begin{proof} (i) follows from Proposition 96. If $2$ does not divide $n$, then $r_2>2d_3$. Hence $\mathrm{Int}(F_x,F_y)>n-1$. This proves (ii).
\end{proof}

Let the notations be as above and assume that $F$ has one place at infinity. It follows that $F_{\infty}(u,y)$ is a monic irreducible polynomial of $\KK\llbracket u \rrbracket[y]$ of degree $n$ in $y$. Let
$f(t^n,y)=\prod_{i=1}^n(y-y_i(t))$. We have, as in the proof of Proposition \ref{Int-int}, $f(t^n,t^{-n}y)=\prod_{i=1}^n(t^{-n}y-y_i(t))=t^{-n^2}\prod_{i=1}^n(y-t^ny_i(t))$. Also $F_{\infty}(x,y)=x^nf(x,x^{-1}y)$, and thus 
$$
F_{\infty}(t^n,y)=\prod_{i=1}^n(y-t^ny_i(t)).
$$
In particular, the roots of $F(t^n,y)=0$ are given by $Y_i(t)=t^ny_i(t)$.

\begin{proposicion} 
\begin{enumerate}[(i)]
\item The set of characteristic exponents of $F_{\infty}$ is given by $\bar{m}_k=n+m_k$.

\item The $\underline{d}$-sequence of $F_{\infty}$  is equal to the  $\underline{d}$-sequence of $f$.

\item The $\underline{r}$-sequence of $F_{\infty}$ is given by $\bar{r}_k=n\frac{n}{d_k}-r_k$

\item For all $k\in \{1,\ldots, h\}$, $\mathrm{App}(F_{\infty}, d_k)= h_{G_k}(u,1,y)$, where we recall that $G_k=\mathrm{App}(F,d_k)$. 
\end{enumerate}
\end{proposicion}

\begin{proof} 
\begin{enumerate}[(i)]
\item The formal power series  $Y_1(t)=t^ny_1(t)$ is  a root of $F_{\infty}(t^n,y)=0$. Hence $\mathrm{Supp}(Y_1(t))=\lbrace n+i, i\in\mathrm{Supp}(y_1(t))\rbrace=n+\mathrm{Supp}(y_1(t))$. Now the proof of (i) follows immediately. 

\item In fact, for all $k$ with $1\leq k\leq h$, we have $\gcd(n,m_1,\ldots,m_k)= \gcd(n,n+m_1,\ldots,n+m_k)$.

\item We shall prove the result by induction on $k$, with $1\leq k\leq h$. We have $\bar{r}_0=n$ and $\bar{r}_1=n+m_1=n-r_1=n\frac{n}{d_1}-r_1$. Suppose that $\bar{r}_k=n\frac{n}{d_k}-r_k$ for some $1<k\leq h$. We have $\bar{r}_{k+1}d_{k+1}=\bar{r}_{k}d_k+(n+m_{k+1}-(n+m_k))d_{k+1}=(n\frac{n}{d_k}-r_k)d_k+
(m_{k+1}-m_k)d_{k+1}=(-r_kd_k+(m_{k+1}-m_k)d_{k+1})+n^2=-r_{k+1}d_{k+1}+n^2$. Hence $\bar{r}_{k+1}=
n\frac{n}{d_{k+1}}-r_{k+1}$.

\item Easy exercise.
\end{enumerate}

\end{proof}

\begin{proposicion}\label{F-Finfty}
Let the notations be as above and let $\Gamma(F_{\infty})$ be the numerical semigroup associated with $F_{\infty}$. We have the following:
\begin{enumerate}[(i)]
\item The conductor of $\Gamma(F_{\infty})$ is given by $\mathrm C(\Gamma(F_{\infty}))=(\sum_{k=1}^h(e_k-1)\bar{r}_k)-n+1=(\sum_{k=1}^h(e_k-1)(n\frac{n}{d_k}-r_k))-n+1$.
\item $\mathrm C(\Gamma(F_{\infty}))+\mathrm C(\Gamma_{\infty}(F))=(n-1)(n-2)$.
\end{enumerate}
\end{proposicion}
\begin{proof} \begin{enumerate}[(i)]
\item Follows from Proposition \ref{Gamma-infty}.

\item $C(\Gamma(F_{\infty}))+C(\Gamma_{\infty}(F))= (\sum_{k=1}^h(e_k-1)(n\frac{n}{d_k}-r_k))-n+1+\sum_{k=1}^h(e_k-1)r_k-n+1=(\sum_{k=1}^h(e_k-1)(n\frac{n}{d_k}))-2(n-1)=n(n-1)-2(n-1)=(n-1)(n-2)$.\qedhere
\end{enumerate}
\end{proof}

\begin{corolario} 
Let the notations be as above. The following are equivalent.
\begin{enumerate}[(i)]
\item $\mathrm C(\Gamma_{\infty}(F))=0$.
\item $\mathrm C(\Gamma(F_{\infty}))=(n-1)(n-2)$.
\item For all $k\in \{1,\ldots, h\}$, $r_k=d_{k+1}$.
\item For all $k\in \{1,\ldots, h\}$, $\bar{r}_k=n\frac{n}{d_k}-d_{k+1}$.
\end{enumerate} 
\end{corolario}

\begin{proof} This follows from Corollary \ref{rk-equal-dk1} and Proposition \ref{F-Finfty}.
\end{proof}

\begin{corolario} Let the notations be as above. Then $\mathrm C(\Gamma(F_{\infty}))\leq (n-1)(n-3)$ and the following are equivalent.
\begin{enumerate}[(i)]
\item $\mathrm C(\Gamma(F_{\infty}))= (n-1)(n-3)$.
\item $\mathrm C(\Gamma_{\infty}(F))=n-1$.
\item For all $k\in \{1,\ldots, h\}$, $r_k=2d_{k+1}$.
\item For all $k\in \{1,\ldots, h\}$, $\bar{r}_k=n\frac{n}{d_k}-2d_{k+1}$.
\end{enumerate}
\end{corolario}
\begin{proof} 
This follows from Corollary \ref{rk-2dk1} and Proposition \ref{F-Finfty}.
\end{proof}







\section{Minimal presentations}

It is usual in Mathematics to represent objects by means of a free object in some generators under certain relations fulfilled by these generators. The reader familiar to Group Theory surely has used many times definitions of groups by means of generators and relations. Relations are usually represented by means of equalities, or simply words in the free group on the generators (this means that they are equal to the identity element; this is due to the fact that we have inverse in groups). Here we represent relations by pairs. 

Let $S$ be a numerical semigroup minimally generated by $\{n_1,\ldots,n_p\}$. Then the monoid morphism 
\[\varphi: \mathbb N^p \to S,\ \varphi(a_1,\ldots,a_p)=\sum_{i=1}^p a_i n_i,\]
known as the \emph{factorization homomorphism} of $S$, is an epimorphism, and consequently $S$ is isomorphic to $\mathbb N^p/\ker \varphi$, where $\ker\varphi$ is the kernel congruence of $\varphi$:
\[\ker\varphi=\{ (a,b)\in \mathbb N^p\times \mathbb N^p\mid \varphi(a)=\varphi(b)\}.\]
Notice that for groups, vector spaces, rings \dots the kernel is defined by the elements mapping to the identity element. This is because there we have inverses and from $f(a)=f(b)$ we get $f(a-b)=0$. This is not the case in numerical semigroups, and this is why the kernel is a congruence, and not a ``subject'' of the domain.

Given $\tau \subset \NN^p\times \NN^p$, the \emph{congruence generated} by $\tau$ is the smallest congruence on $\NN^p$ containing $\tau$, that is, the intersection of all congruences containing $\tau$. We denote by $\mathrm{cong}(\tau)$ the congruence generated by $\tau$. Accordingly, we say that $\tau$ is a generating system of a congruence $\sigma$ on $\NN^p$ if $\mathrm{cong}(\tau)=\sigma$.

The congruence generated by a set is just the reflexive, symmetric, transitive closure (this would just make the closure an equivalence relation), to which we adjoin all pairs $(a+c,b+c)$ whenever $(a,b)$ is in the closure; so that we make the resulting relation a congruence. This can be formally written as follows.

\begin{proposicion}\label{congruence-generated}
Let $\rho\subseteq \NN^p\times\NN^p$. Define
\[\rho^0=\rho\cup\{(b,a)\mid (a,b)\in \rho\}\cup \{(a,a)\mid a\in \NN^p\},\]
\[\rho^1={(v+u,w+u)},{(v,w)\in \rho^0, u\in \NN^p}.\]
Then $\mathrm{cong}(\rho)$ is the set of pairs $(v,w)\in \NN^p\times\NN^p$ such that there exist $k\in \NN$ and $v_0,\ldots,v_k\in\NN^p$ with $v_0=v$, $v_k=w$ and $(v_i,v_{i+1})\in \rho^1$ for all $i\in\{0,\ldots,k-1\}$.
\end{proposicion}
\begin{proof}
  We first show that the set constructed in this way is a
  congruence. Let us call this set $\sigma$.
  \begin{enumerate}
  \item Since $(a,a)\in\rho^0\subseteq \sigma$ for all $a\in \NN^p$, the binary relation $\sigma$ is reflexive.
  \item If $(v,w)\in \sigma$, there exist $k\in \NN$ and     $v_0,\ldots,v_k\in\NN^p$ such that $v_0=v$, $v_k=w$ and $(v_i,v_{i+1})\in \rho^1$ for all $i\in\{0,\ldots,k-1\}$. Since $(v_i,v_{i+1})\in \rho^1$ implies that $(v_{i+1},v_i)\in \rho^1$, by defining $w_i=v_{k-i}$ for every $i\in\{0,\ldots k\}$, we obtain that $(w,v)\in \sigma$. Hence $\sigma$ is symmetric.
  \item If $(u,v)$ and $(v,w)$ are in $\sigma$, then there exists $k,l\in \NN$ and $v_0,\ldots, v_k, w_0,\ldots, w_l\in \NN^p$ such that $v_0=u$, $v_k=w_0=v$, $w_l=w$ and $(v_i,v_{i+1}),(w_j,w_{j+1})\in \rho^1$ for all suitable
    $i,j$. By concatenating these we obtain $(u,w)\in \sigma$. Thus $\sigma$ is transitive.
  \item Finally, let $(v,w)\in\sigma$ and $u\in \NN^p$. There exists $k\in \NN$ and $v_0,\ldots,v_k\in\NN^p$ such that $v_0=v$, $v_k=w$ and $(v_i,v_{i+1})\in \rho^1$ for all $i\in\{0,\ldots,k-1\}$. By defining $w_i=v_i+u$ for all $i\in\{0,\ldots,k\}$ we have $(w_i,w_{i+1})\in \rho^1$ and consequently $(v+u,w+u)\in \sigma$.
  \end{enumerate}
  It is clear that every congruence containing $\rho$ must contain $\sigma$ and this means that $\sigma$ is the least congruence on $\NN^p$ that contains $\rho$, whence, $\sigma=\mathrm{cong}(\rho)$.
\end{proof}

A \emph{presentation} for $S$ is a generating system of $\ker\varphi$ as a congruence, and a \emph{minimal presentation} is a presentation such that none of its proper subsets is a presentation.

\begin{example}
For instance, a minimal presentation for $S=\langle 2,3\rangle$ is $\{((3,0),(0,2))\}$. This means that $S$ is the commutative monoid generated by two elements, say $a$ and $b$, under the relation $3a=2b$.
\end{example}

For $s\in S$, the \emph{set of factorizations} of $s$ in $S$ is the set 
\[\mathsf Z(s)=\varphi^{-1}(s)=\{ a\in \NN^p \mid \varphi(a)=s\}.\]
Notice that the set of factorizations of $s$ has finitely many elements. 
This can be shown in different ways. For instance the $i$th coordinate of a factorization is smaller than or equal to $s/n_i$. Also, two factorizations are incomparable with respect to the usual partial ordering on $\NN^p$, and thus Dickson's lemma ensures that there are finitely many of them.

We define, associated to $s$, the graph $\nabla_s$ whose vertices are the elements of $\mathsf Z(s)$ and $ab$ is an edge if $a\cdot b\neq 0$ (dot product).

We say that two factorizations $a$ and $b$ of $s$ are $\mathcal R$-related if they belong to the same connected component of $\nabla_s$, that is, there exists a chain of factorizations $a_1,\ldots,a_t\in\mathsf Z(s)$ such that 
\begin{itemize}
\item $a_1=a$, $a_t=b$,
\item for all $i\in \{1,\ldots,t-1\}$, $a_i\cdot a_{i-1}\neq 0$.
\end{itemize}

\begin{example}\label{ej-con-26}
Let $S=\langle 5,7,11,13\rangle$. We draw $\nabla_{26}$.

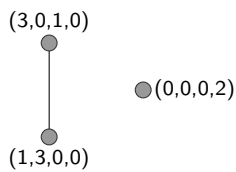
\begin{center}

\begin{tikzpicture}[y=.25cm, x=.25cm,font=\sffamily]

\draw (5,3) -- (5,8);

\filldraw[fill=black!40,draw=black!80] (5,3) circle (3pt)    node[anchor=north] {{\tiny (1,3,0,0)}};

\filldraw[fill=black!40,draw=black!80] (5,8) circle (3pt)    node[anchor=south] {{\tiny (3,0,1,0)}};

\filldraw[fill=black!40,draw=black!80] (10,5.5) circle (3pt)    node[anchor=west] {{\tiny (0,0,0,2)}};

\end{tikzpicture}
\end{center}
This graph has two connected components. Also, we have that $((3,0,1,0),(0,0,0,2))\in\ker\varphi$, and as $((3,0,1,0),(1,3,0,0))\in \ker\varphi$, we also have that removing the common part we obtain a new element in the kernel: $((2,0,1,0),(0,3,0,0))$. This new element corresponds to $21=2\times 5+11=3\times 7$. If we draw $\nabla_{21}$, we obtain 

\begin{center}
\begin{tikzpicture}[y=.25cm, x=.25cm,font=\sffamily]

\filldraw[fill=black!40,draw=black!80] (5,3) circle (3pt)    node[anchor=north] {{\tiny (0,3,0,0)}};

\filldraw[fill=black!40,draw=black!80] (10,3) circle (3pt)    node[anchor=north] {{\tiny (2,0,1,0)}};
\end{tikzpicture}
\end{center}
which is another nonconnected graph.
\end{example}

Let $\tau\subset \NN^p\times \NN^p$. We say that $\tau$ is \emph{compatible} with $s\in S$ if either $\nabla_s$ is connected or if $R_1,\ldots, R_t$ are the connected components of $\nabla_s$, then for every $i\in \{1,\ldots,t\}$ we can choose $a_i\in R_i$ such that for every $i,j\in\{1,\ldots,t\}$, $i\neq j$, there exists $i_1,\ldots,i_k\in \{1,\ldots,t\}$ fulfilling 
\begin{itemize}
\item $i_1=i$, $i_k=j$,
\item for every $m\in\{1,\ldots, k-1\}$ either $(a_{i_m},a_{i_{m+1}})\in \tau$ or $(a_{i_{m+1}},a_{i_m})\in \tau$.
\end{itemize}
Even though this definition might seem strange, we are going to show next that we only have to look at those $\nabla_n$ that are nonconnected in order to construct a (minimal) presentation.

Denote by $\mathbf e_i$ the $i$th row of the $p\times p$ identity matrix.

\begin{teorema}\label{carac-presentation}
Let $S$ be a numerical semigroup minimally generated by $\{n_1,\ldots, n_p\}$, and let $\tau\subseteq \NN^p\times\NN^p$. Then $\tau$ is a presentation of $S$ if and only if $\tau$ is compatible with $s$ for all $s\in S$.
\end{teorema}
\begin{proof}
\emph{Necessity.} If $\nabla_s$ is connected, then there is nothing to prove. 

Let $R_1,\ldots,R_t$ be the $\mathcal R$-classes contained in $\mathsf Z(s)$. Let $i$ and $j$ be in $\{1,\ldots,t\}$ with $i\not=j$. Let $a\in R_i$ and $b\in R_j$. As $a,b\in \mathsf Z(n)$, $(a,b)\in \ker\varphi$. Since $\mathrm{cong}(\tau)=\ker\varphi$, by Proposition \ref{congruence-generated}, there exist $b_0,b_1,\ldots, b_r\in\NN^p$, such that $a=b_0$, $b=b_r$ and $(b_i,b_{i+1})\in \beta^1$ for $i\in\{0,\ldots,r-1\}$. Hence there exist for all $i\in\{0,\ldots,r-1\}$, $z_i\in \NN^p$ and $(x_i,y_i)\in \tau$ such that either $(b_i,b_{i+1})=(x_i+z_i,y_i+z_i)$ or $(b_i,b_{i+1})=(y_i+z_i,x_i+z_i)$. If $z_i\not=0$, then $b_i\mathcal R b_{i+1}$. And if $z_i=0$, then $\{b_i,b_{i+1}\}\subseteq \mathsf Z(s)$. Hence the pairs $(b_i,b_{i+1})\not\in \mathcal R$ yield the $a_i$'s we are looking for.

\emph{Sufficiency.} It suffices to prove that for every $s\in S$ and $a,b\in \mathsf Z(s)$, $(a,b)\in \mathrm{cong}(\tau)$. We use induction on $s$. The result follows trivially for $s=0$, since $\mathsf Z(0)=\{0\}$. 

If $a\mathcal R b$, then there exists $a_1,\ldots, a_k\in \mathsf Z(s)$ such that $a_1=a$, $a_k=b$ and $a_i\cdot a_{i+1}\neq 0$ for all $i\in\{1,\ldots, k-1\}$. Hence for every $i$, there exists $j\in \{1,\ldots,p\}$ such that $a_i-\mathbf e_j,a_{i+1}-\mathbf e_j\in \mathsf Z(s-n_j)$. By induction hypothesis $(a_i-\mathbf e_j,a_{i+1}-\mathbf e_j)\in \mathrm{cong}(\tau)$, whence $(a_i,a_{i+1})\in \mathrm{cong}(\tau)$ for all $i$. By transitivity $(a,b)\in \tau$.

Assume now that $a$ and $b$ are in different connected components of $\nabla_s$. If $R_1,\ldots, R_t$ are the connected components of $\nabla_s$, we may assume without loss of generality that $a\in R_1$ and $b\in R_2$. As $\tau$ is compatible with $s$, there exists a chain $a_1,\ldots, a_k$ such that either $(a_i,a_{i+1})\in \tau$ or $(a_{i+1},a_i)\in \tau$, $a_1\in R_1$ and $a_2\in R_2$. Hence $(a_i,a_{i+1})\in \mathrm{cong}(\tau)$, and by the above paragraph, $(a,a_1),(a_k,b)\in \mathrm{cong}(\tau)$. By transitivity we deduce that $(a,b)\in \mathrm{cong}(\tau)$.
\end{proof}

Observe that as a consequence of this theorem, in order to obtain a presentation for $S$ we only need for every $s\in S$ with nonconnected graph $\nabla_s$ and every connected component $R$ choose a factorization $x$ and pairs $(x,y)$ such that every two connected components of $\nabla_s$ are connected by a sequence of these factorizations with consecutive elements either a chosen pair or its symmetry. The least possible number of edges we need is when we choose the pairs so that we obtain a tree connecting all connected components. Thus the least possible number of pairs for every $s\in S$ with associated nonconnected graph is the number of connected components of $\nabla_s$ minus one. 

\begin{corolario}
Let $S$ be a numerical semigroup. The cardinality of any minimal presentation of $S$ equals $\sum_{s\in S} (\mathrm{nc}(\nabla_s)-1)$, where $\mathrm{nc}(\nabla_s)$ is the number of connected components of $\nabla_s$.
\end{corolario}

We now show that this cardinality is finite by showing that only finitely many elements of $S$ have nonconnected associated graphs.

\begin{proposicion}\label{betti}
Let $S$ be a numerical semigroup minimally generated by $\{n_1,\ldots,n_p\}$, and let $s\in S$. If $\nabla_s$ is not connected, then $s=n_i+w$ with $i\in\{2,\ldots, p\}$ and $w\in \mathrm{Ap}(S,n_1)$.
\end{proposicion}
\begin{proof}
Observe that $\nabla_{n_i}=\{\mathbf e_i\}$. Hence $s\not\in \{n_1,\ldots, n_p\}$, and thus there exists $i\in\{1,\ldots, p\}$ such that $s-n_i\in S^*$.
If $s\in \mathrm{Ap}(S,n_1)$, then $s-n_i\in \mathrm{Ap}(S,n_1)$, and we are done.

Now assume that $s-n_1\in S$. There exists an element $a\in \mathsf Z(s)$ with $a-\mathbf e_1 \in \NN^p$. Take $b\in \mathsf Z(s)$ in a different connected component of $\nabla_s$ than the one containing $a$. Clearly $a\cdot b=0$, and thus $b-\mathbf e_1\not\in \NN^p$. Since $b\neq 0$, there exists $i\in\{2,\ldots,p\}$ such that $b-\mathbf e_i\in \NN^p$, and consequently $s-n_i\in S$. We prove that $s-(n_i+n_1)\not\in S$, and thus $s=(s-n_i)+n_i$ with $s-n_i\in \mathrm{Ap}(S,n_1)$. Suppose to the contrary that $s-(n_1+n_i)\in S$. Hence there exists a factorization of $c$ of $s$ such that $c-(\mathbf e_1+\mathbf e_i)\in\NN^p$. Then $a\cdot c\neq 0$ and $c\cdot b\neq 0$. This force $a$ and $b$ to be in the same connected component of $\nabla_s$, a contradiction.
\end{proof}

We say that $s\in S$ is a \emph{Betti element} if $\nabla_s$ is not connected.

\begin{example}
We continue with the semigroup in Example \ref{ej-con-26}.
\begin{verbatim}

gap> s:=NumericalSemigroup(5,7,11,13);;

\end{verbatim}
We can use the following to compute a minimal presentation for this semigroup.

\begin{verbatim}
gap> MinimalPresentationOfNumericalSemigroup(s);
[ [ [ 0, 1, 1, 0 ], [ 1, 0, 0, 1 ] ], [ [ 0, 3, 0, 0 ], [ 2, 0, 1, 0 ] ],
  [ [ 1, 3, 0, 0 ], [ 0, 0, 0, 2 ] ], [ [ 2, 2, 0, 0 ], [ 0, 0, 1, 1 ] ],
  [ [ 3, 1, 0, 0 ], [ 0, 0, 2, 0 ] ], [ [ 4, 0, 0, 0 ], [ 0, 1, 0, 1 ] ] ]
  
\end{verbatim}
  
Let us have a look at $\nabla_{50}$.   
\begin{verbatim}
gap> FactorizationsElementWRTNumericalSemigroup(50,s);
[ [ 10, 0, 0, 0 ], [ 3, 5, 0, 0 ], [ 5, 2, 1, 0 ], [ 0, 4, 2, 0 ],
  [ 2, 1, 3, 0 ], [ 6, 1, 0, 1 ], [ 1, 3, 1, 1 ], [ 3, 0, 2, 1 ],
  [ 2, 2, 0, 2 ], [ 0, 0, 1, 3 ] ]
gap> RClassesOfSetOfFactorizations(last);
[ [ [ 0, 0, 1, 3 ], [ 0, 4, 2, 0 ], [ 1, 3, 1, 1 ], [ 2, 1, 3, 0 ],
      [ 2, 2, 0, 2 ], [ 3, 0, 2, 1 ], [ 3, 5, 0, 0 ], [ 5, 2, 1, 0 ],
      [ 6, 1, 0, 1 ], [ 10, 0, 0, 0 ] ] ]
gap> Length(last);
1
\end{verbatim}
And this means that $\nabla_{50}$ has a single connected component, and thus is not a Betti element. 

We can compute the set of Betti elements.
\begin{verbatim}
gap> BettiElementsOfNumericalSemigroup(s);
[ 18, 20, 21, 22, 24, 26 ]
\end{verbatim}
So for instance $\nabla_{26}$ has two connected components as we already saw in Example \ref{ej-con-26}.
\begin{verbatim}
gap> FactorizationsElementWRTNumericalSemigroup(26,s);
[ [ 1, 3, 0, 0 ], [ 3, 0, 1, 0 ], [ 0, 0, 0, 2 ] ]
gap> RClassesOfSetOfFactorizations(last);             
[ [ [ 1, 3, 0, 0 ], [ 3, 0, 1, 0 ] ], [ [ 0, 0, 0, 2 ] ] ]
\end{verbatim}

\end{example}

We now show an alternative method to compute a presentation  based on what is known in the literature as Herzog's correspondence (\cite{Herzog}). 

Let $S$ be a numerical semigroup minimally generated by $\{n_1,\ldots, n_p\}$. For $\KK$ a field, the \emph{semigroup ring} associated to $S$ is the ring $\KK[S]=\bigoplus_{s\in S} \KK t^s$, where $t$ is a symbol or an indeterminate. Addition in $\KK[S]$ is performed componentwise, while multiplication is done by using distributivity and the rule $t^st^{s'}=t^{s+s'}$, for $s,s'\in S$. We can see the elements in $\KK[S]$ as polynomials in $t$ whose nonnegative coefficients correspond to exponents in $S$. Also $\KK[S]=\KK[t^{n_1},\ldots,t^{n_p}]\subseteq \KK[t]$. Thus, $\KK[S]$ can be seen as the coordinate ring of a curve parametrized by monomials.

Let $x_1,\ldots, x_p$ be indeterminates, and $\KK[x_1,\ldots,x_p]$ be the polynomial ring over these indeterminates with coefficients in the field $\KK$. For $a=(a_1,\ldots a_p)\in \NN^p$ write 
\[X^a=x_1^{a_1}\cdots x_p^{a_p}.\]
Let $\psi$ the ring homomorphism determined by 
\[\psi:\KK[x_1,\ldots, x_p]\to \KK[S],\quad x_i\mapsto t^{n_i}.\]
This can be seen as a graded morphism if we grade $\KK[x_1,\ldots,x_p]$ in the following way: a polynomial $p$ is $S$-homogeneous of degree $s\in S$ if $p=\sum_{a\in A} c_a X^a$ for some $A\subset \NN^p$ with finitely many elements and $\varphi(a)=s$ for all $a\in A$. Observe that $\KK[S]$ is also $S$-graded in a natural way, and so $\psi$ is a graded epimorphism. 

For $A\subseteq \KK[x_1,\ldots,x_p]$, denote by $(A)$ the ideal generated by $A$.

\begin{proposicion}
$\ker \psi=(X^a-X^b\mid (a,b)\in \ker \varphi)$.
\end{proposicion}
\begin{proof}
Clearly $\psi(X^a)=t^{\varphi(a)}$. Hence for $(a,b)\in \ker\varphi$, $\psi(X^a-X^b)=0$. This implies that $(X^a-X^b\mid (a,b)\in \ker \varphi)\subseteq \ker\psi$. 
Since $\psi$ is a graded morphism, for the other inclusion, it suffices to proof that if $f\in \ker\psi$ is $S$-homogeneous of degree $s\in S$, then $f\in (X^a-X^b\mid (a,b)\in \ker \varphi)$. Write $f=\sum_{a\in A}c_a X^a$, with $c_a\in \KK$ and $a\in\mathsf Z(s)$ for all $a\in A$, and $A$ a finite set. Then $\varphi(f)=t^s\sum_{a\in A}c_a=0$, and consequently $\sum_{a\in A}c_a=0$. Choose $a\in A$. Then $f=\sum_{a'\in A\setminus\{a\}} c_a(X^{a'}-X^a)$, and thus $f\in (X^a-X^b\mid (a,b)\in \ker \varphi)$.
\end{proof}

From Proposition \ref{congruence-generated}, it can be easily derived that for any $\tau\in\NN^p\times \NN^p$ 
\[(X^a-X^b\mid (a,b)\in \tau)= (X^a-X^b\mid (a,b)\in \mathrm{cong}(\tau)).\]
Hence, we get the following consequence.

\begin{corolario}
Let $S$ be a numerical semigroup and $\tau$ a presentation of $S$. Then 
\[\ker \psi=(X^a-X^b\mid (a,b)\in \tau).\]
\end{corolario}

Observe that the generators of $\ker\psi$ can be seen as the implicit equations of the curve whose coordinate ring is $\KK[S]$. In this way we can solve the implicitation problem without the use of elimination theory nor Gr\"obner bases.

\begin{example} 
Let $S=\langle 3,5,7\rangle$. Then $\mathrm{Ap}(S,3)=\{0,5,7\}$. According to Proposition \ref{betti}, $\mathrm{Betti}(S)\subseteq \{10,12,14\}$. The sets of factorizations of $10$, $12$ and $14$ are $\{ ( 0, 2, 0 ), ( 1, 0, 1 )\}$, $\{(4, 0, 0 ), ( 0, 1, 1 )\}$ and $\{( 3, 1, 0 ), ( 0, 0, 2 )\}$, respectively. Hence $\mathrm{Betti}(S)= \{10,12,14\}$, and by Theorem \ref{carac-presentation},  
\[\{(( 0, 2, 0 ), ( 1, 0, 1 )), (( 3, 1, 0 ), ( 0, 0, 2 )), (( 4, 0, 0 ), ( 0, 1, 1 ))\}\]
is a minimal presentation of $S$. The implicit equations of the curve parametrized by $(t^3,t^5,t^7)$ are 
\[\begin{cases}
xz-y^2 & =0, \\
x^3y-z^2 & =0,\\
x^4-yz&=0.
\end{cases}\]
Let us reproduce this example with the use of polynomials. Take $\psi: K[x,y,z]\to K[t]$ be determined by $x\mapsto t^3$, $y\mapsto t^5$ and $z\mapsto t^7$. We consider now the ideal $(x-t^3,y-t^5,z-t^7)$. We now compute a Gr\"obner basis with respect to any eliminating order on $t$. We can for instance do this with \texttt{Singular}, \cite{singular}.
\begin{verbatim}
> ring r=0,(t,x,y,z),lp;
> ideal i=(x-t^3,y-t^5,z-t^4);
> std(i);
_[1]=y4-z5
_[2]=xz3-y3
_[3]=xy-z2
_[4]=x2z-y2
_[5]=x3-yz
_[6]=tz-y
_[7]=ty-x2
_[8]=tx-z
_[9]=t3-x
\end{verbatim}
Now we choose those not having $t$, or we can just type:
\begin{verbatim}
> eliminate(i,t);
_[1]=y7-z5
_[2]=xz-y2
_[3]=xy5-z4
_[4]=x2y3-z3
_[5]=x3y-z2
_[6]=x4-yz
\end{verbatim}
Which by Herzog's correspondence yields a presentation for $S$. However this is not a minimal presentation. In order to get a minimal presentation we can use \texttt{minbase} in \texttt{Singular}, but this applies only to homogeneous ideals. To solve this issue, we give weights 3,5,7 to $x,y,z$, respectively.
\begin{verbatim}
> ring r=0,(t,x,y,z),(dp(1),wp(3,5,7));
> ideal i=(x-t^3,y-t^5,z-t^7);
> ideal j=eliminate(i,t);
> minbase(j);
_[1]=y2-xz
_[2]=x4-yz
_[3]=x3y-z2
\end{verbatim}
\end{example}

\section{Factorizations}

Let $S$ be a numerical semigroup minimally generated by $\{n_1,\ldots,n_p\}$. For $s\in S$, recall that the set of factorizations of $s$ is $\mathsf Z(s)=\varphi^{-1}(s)$. 

For a factorization $x=(x_1,\ldots,x_p)$ of $s$ its \emph{length} is defined as 
\[
\lvert x \rvert = x_1+\dots+x_p,
\]  
and the \emph{set of lengths} of $s$ is 
\[
\mathsf L(s)=\{ \lvert x\rvert \mid x\in\mathsf Z(s)\}.
\]

Since $\mathsf Z(s)$ has finitely many elements, so has $\mathsf L(s)$. A monoid is \emph{half factorial} if the cardinality of $\mathsf L(s)$ is one for all $s\in S$. 

\begin{example}
Let $S=\langle 2,3\rangle$. Here $6$ factors as $6=2\times 3=3\times 2$, that is, $\mathsf Z(6)=\{(3,0) , (0,2) \}$. The length of $(3,0)$ is $3$, while that of $(0,2)$ is 2. So $S$ is not a unique factorization monoid, and it is not either a half factorial monoid. The only half factorial numerical semigroup is $\NN$. 
\end{example}

\subsection{Length based invariants}
One of the first nonunique factorization invariants that appeared in the literature was the elasticity. It was meant to measure how far is a monoid from being half factorial. The elasticity of a numerical semigroup is a rational number greater than one. Actually, half factorial monoids are those having elasticity one.

Let $s\in S$. The \emph{elasticity} of $s$, denoted by $\rho(s)$ is defined as 
\[
\rho(s)=\frac{\max \mathsf L(s)}{\min \mathsf L(s)}.
\]
The elasticity of $S$ is defined as 
\[
\rho(S)=\sup_{s\in S} \rho(s).
\]
The computation of the elasticity in finitely generated cancellative monoids requires the calculation of primitive elements of $\ker\varphi$. However in numerical semigroups, this calculation is quite simple, as the following example shows.

\begin{teorema}
Let $S$ be a numerical semigroup minimally generated by $\{n_1,\ldots,n_p\}$ with $n_1<\cdots <n_p$. Then
\[
\rho(S)=\frac{n_p}{n_1}.
\]
\end{teorema}
\begin{proof}
Let $s\in S$ and assume that $a=(a_1,\ldots,a_p)$ and $b=(b_1,\ldots,b_p)$ are such that $\lvert a\rvert =\max \mathsf L(S)$ and $\lvert b\rvert =\min\mathsf L(S)$. We know that $\varphi(a)=\varphi(b)$, that is, $a_1n_1+\cdots +a_pn_p=b_1n_1+\dots +b_pn_p=s$. Now by using that $n_1<\dots< n_p$, we deduce that
\[
n_1\lvert a\rvert \le  s \le n_p \lvert b\rvert,
\]
and thus 
\[
\rho(s) = \frac{\lvert a\rvert}{\lvert b\rvert} \le \frac{n_p}{n_1}.
\]
This implies that $\rho(S)\le \frac{n_p}{n_1}$.
Also $\rho(n_1n_p)\ge \frac{n_p}{n_1}$, since $n_p \mathbf e_1, n_1\mathbf e_p\in \mathsf Z(n_1n_p)$. Hence 
\[
\frac{n_p}{n_1}\le \rho(n_1n_p)\le \rho(S)\le \frac{n_p}{n_1},
\]
and we get an equality.
\end{proof}

Another way to measure how far we are from half factoriality, is to measure how distant are the different lengths of factorizations. This is the motivation for the following definition.

Assume that $\mathsf L(s)=\{l_1<\cdots < l_k\}$. Define the \emph{Delta set} of $s$ as 
\[
\Delta(s)=\{ l_2-l_1,\ldots, l_k-l_{k-1}\},
\]
and if $k=1$, $\Delta(s)=\emptyset$. The Delta set of $S$ is defined as 
\[
\Delta(S)=\bigcup_{s\in S} \Delta(s).
\]
So, the bigger $\Delta(S)$ is, the farther is $S$ from begin half factorial.

A pair of elements $(a,b)\in \NN^p\times \NN^p$ is in $\ker\varphi$ if $a$ and $b$ are factorizations of the same element in $S$. As a presentation is a system of generators of $\ker\varphi$ it seems natural that the information on the factorizations could be recovered from it. We start showing that this is the case with the Delta sets, and will see later that the same holds for other invariants.

Let $M_S=\{ a-b \mid (a,b)\in \ker\varphi \}\subseteq \ZZ^p$. Since $\ker\varphi$ is a congruence, it easily follows that $M_S$ is a subgroup of $\ZZ^p$.
\begin{lema}\label{gens-M-S}
Let $\sigma$ be a presentation of $S$. Then $M_S$ is generated as a group by $\{a-b\mid (a,b)\in \sigma\}$.
\end{lema}
\begin{proof}
Let $z\in M_S$. Then there exists $(a,b)\in \ker\varphi$. From Proposition \ref{congruence-generated}, there exists $x_1,\ldots, x_t$ such that $x_1=a$, $x_t=b$, and for all $i\in\{1,\ldots, t-1\}$ there exists $(a_i,b_i)$ and $c_i\in\NN^p$ such that $(x_i,x_{i+1})=(a_i+c_i,b_i+c_i)$ with either $(a_i,b_i)\in\sigma$ or $(b_i,a_i)\in \sigma$. Then 
\[ a-b=(x_1-x_2)+(x_2-x_3)+\cdots +(x_{t-1}-x_t)=\sum_{i=1}^{t-1} (a_i-b_i),\]
and the proof follows easily.
\end{proof}

For a given $z=(z_1,\ldots,z_p)\in \ZZ^p$, we also use the notation $\lvert z\rvert =z_1+\cdots+z_p$. 

\begin{lema}\label{combinacion-delta}
Let $\sigma=\{(a_1,b_1),\ldots,(a_t,b_t)\}$ be a presentation of $S$, and set $\delta_i=|a_i-b_i|$, $i\in\{1,\ldots,s\}$. Then every element in $\Delta(S)$ is of the form
\[ \lambda_1\delta_1+\cdots+\lambda_t\delta_t,\]
for some integers $\lambda_1,\ldots,\lambda_t$.
\end{lema}
\begin{proof}
The proof follows easily from the proof of Lemma \ref{gens-M-S}.
\end{proof}

\begin{teorema}
Let $S$ be a numerical semigroup and let $\sigma$ be a presentation of $S$. Then
\[
\min \Delta(S)=\gcd\{ \lvert a-b\rvert \mid (a,b)\in \sigma\}.
\]
\end{teorema}
\begin{proof}
In order to simplify notation, write $d=\gcd\{ \lvert a-b\rvert \mid (a,b)\in \sigma\}$.
If $\delta\in \Delta(S)$, then by Lemma \ref{combinacion-delta}, we know that $\delta$ is a linear combination with integer coefficients of elements of the form $\lvert a-b\rvert$ with $(a,b)\in \sigma$. Hence $d \mid \delta$, and consequently $d \le \min \Delta(S)$. Now let $(a_1,b_1),\ldots,(a_k,b_k)\in \sigma$ and $\lambda_1,\ldots, \lambda_k\in \ZZ$ be such that $\lambda_1\lvert a_1-b_1\rvert+\cdots+ \lambda_k \lvert a_k-b_k\rvert =d$. If $\lambda_i<0$, change $(a_i,b_i)$ with $(b_i-a_i)$, so that we can assume that all $\lambda_i$ are nonnegative. The element $s=\varphi(\lambda_1a_1+\cdots +\lambda_ka_k)=\varphi(\lambda_1b_1+\cdots+ \lambda_kb_k)$ has two factorizations $z=\lambda_1a_1+\cdots +\lambda_ka_k$ and $z'=\lambda_1b_1+\cdots +\lambda_kb_k$ such that the differences in their lengths is $d$. Hence 
\[\min\Delta(S)\le \min\Delta(s)\le d\le \min\Delta(S),\]
and we get an equality.
\end{proof}

\begin{teorema}\label{max-delta}
Let $S$ be a numerical semigroup. Then
\[
\max \Delta(S)=\max \{ \max \Delta(b) \mid b\in \mathrm{Betti}(S)\}.
\]
\end{teorema}
\begin{proof}
The inequality $\max_{n\in\mathrm{Betti}(S)}\max\Delta(n)\le\max\Delta(S)$ is clear. 

Assume to the contrary  $\max\Delta(S)>\max\Delta(b)$ for all Betti elements $b$ of $S$. Take $x,y$ factorizations of an element $s\in S$ so that $\lvert y\rvert -\lvert x\rvert=\max\Delta(S)$, and consequently no other factorization $z$  of $s$ fulfills $\lvert x\rvert<\lvert z\rvert <\lvert y\rvert$. As $\varphi(x)=\varphi(y)$, Proposition \ref{congruence-generated}, ensures the existence of $x_1,\ldots,x_t$ in $\mathsf Z(s)$ such that $x=x_1$, $x_t=y$ and $(x_i,x_{i+1})=(a_i+c_i,b_i+c_i)$, with either $(a_i,b_i)\in \sigma$ or $(b_i,a_i)\in\sigma$ for all $i\in\{1,\ldots,t-1\}$. From the above discussion, there exists $i\in \{1,\ldots,t-1\}$, with $\lvert x_i\rvert \le \lvert x\rvert <\lvert y\rvert \le \lvert x_{i+1}\rvert$. Both $a_i$ and $b_i$ are factorizations of an element $n$ with $\mathsf Z(n)$ having more than one $\mathcal R$-class. So there is a chain of factorizations, say $z_1,\ldots,z_u$, of $n$ such that $a_i=z_1,\ldots, z_u=b_i$, and $\vert z_{j+1}\rvert -\lvert z_j\rvert \le \max\Delta(n)$, which we are assuming smaller than $\Delta(S)$. But then $\varphi(z_j+c_i)=\varphi( x)=\varphi( y)$ for all $j$, and from the choice of $x$ and $y$, there is no $j$ such that $\vert x\rvert<\lvert z_j+c_i\rvert <\lvert y\rvert$. Again, we can find $j\in\{1,\ldots,u-1\}$ such that $\lvert z_j+c_i\rvert \le \lvert x\rvert <\lvert y\rvert \le \lvert z_{j+1}+c_i\rvert$. And this leads to a contradiction, since $\max\Delta(S)=\lvert y\rvert -\lvert x\rvert\le \lvert z_{j+1}+c_i\rvert -\lvert z_j+c_i\rvert =\lvert z_{i+1}-z_i\rvert \le \max\Delta(n)<\max\Delta(S)$.
\end{proof}

\begin{example}
Let us go back to $S=\langle 2,3\rangle$. We know that the only Betti element of $S$ is $6$. The set of factorizations of $6$ is $\mathsf Z(6)=\{(3,0), (0,2)\}$, and $\mathsf L(S)=\{2,3\}$. Whence $\Delta(6)=\{1\}$. The above theorem implies that $\Delta(S)=\{1\}$. This is actually the closest we can be in a numerical semigroup to be half factorial.
\end{example}

\begin{example}\label{ej-10111713}
Now we do some computations with a numerical semigroup with four generators.
\begin{verbatim}

gap> s:=NumericalSemigroup(10,11,17,23);;
gap> FactorizationsElementWRTNumericalSemigroup(60,s);
[ [ 6, 0, 0, 0 ], [ 1, 3, 1, 0 ], [ 2, 0, 1, 1 ] ]
gap> LengthsOfFactorizationsElementWRTNumericalSemigroup(60,s);
[ 4, 5, 6 ]
gap> ElasticityOfFactorizationsElementWRTNumericalSemigroup(60,s);
3/2
gap> DeltaSetOfFactorizationsElementWRTNumericalSemigroup(60,s);
[ 1 ]
gap> BettiElementsOfNumericalSemigroup(s);
[ 33, 34, 40, 69 ]
gap> Set(last, x->DeltaSetOfFactorizationsElementWRTNumericalSemigroup(x,s));
[ [  ], [ 1 ], [ 2 ], [ 3 ] ]
gap> ElasticityOfNumericalSemigroup(s);
23/10
\end{verbatim}
\end{example}

\subsection{Distance based invariants}
We now introduce some invariants that depend on distances between factorizations. These invariants will measure how spread are the factorizations of elements in the monoid.

The set $\NN^p$ is a lattice with respect to the partial ordering $\le$. Infimum and supremum of a set with two elements is constructed by taking minimum and maximum coordinate by coordinate, respectively. For $x=(x_1,\ldots, x_p),y=(y_1,\ldots, y_p)\in\NN^p$, $\inf\{x,y\}$ will be denoted by $x\wedge y$. Thus
\[
x\wedge y=(\min\{x_1,y_1\},\ldots, \min\{x_p,y_p\}).
\]
The \emph{distance} between $x$ and $y$ is defined as 
\[
\mathrm d(x,y)=\max\{ \lvert x-(x\wedge y)\rvert, \lvert y-(x\wedge y)\rvert\}
\]
(equivalently $\mathrm d(x,y)=\max\{ \lvert x\rvert, \lvert y\rvert\}-\lvert x\wedge y\rvert$).

The distance between two factorizations of the same element is lower bounded in the following way.

\begin{lema}\label{max-len-dist}
Let $x,y\in \NN^p$ with $x\neq y$ and $\varphi(x)=\varphi(y)$. Then
\[
2+\big\lvert \lvert x\rvert - \lvert y\rvert \big\rvert \le \mathrm d(x,y).
\]
\end{lema}
\begin{proof}
We can assume that $x\wedge y=0$, since distance is preserved under translations, $\lvert \lvert x\rvert - \lvert y\rvert \rvert = \lvert \lvert x-(x\wedge y)\rvert - \lvert y-(x\wedge y)\rvert \rvert$ and $\varphi(x-(x\wedge y))=\varphi(y-(x\wedge y))$. As $\varphi(x)=\varphi(y)$ and $x\neq y$, in particular we have that $\lvert x\rvert \ge 2$ and the same for $\lvert y\rvert$. Also, as $x\wedge y=0$, $\mathrm d(x,y)=\max\{\lvert x\rvert, \lvert y\rvert\}$. If $\lvert x\rvert \ge \lvert y\rvert$, then $2+\lvert \lvert x\rvert - \lvert y\rvert \rvert =  \lvert x\rvert (2 - \lvert y\rvert) \le \lvert x\rvert =\mathrm d(x,y)$. A similar argument applies for $\lvert x\rvert \le \lvert y\rvert$.
\end{proof}

\begin{example}\label{ejemplo-cat-estacas}
The factorizations of $66\in \langle 6,9,11\rangle$ are \[\mathsf Z(66)=\{ (0, 0, 6 ), ( 1, 3, 3 ), ( 2, 6, 0 ), (4, 1, 3 ), ( 5, 4, 0 ),( 8, 2, 0),( 11, 0, 0 ) \}.\] The distance between $(11,0,0)$ and $(0,0,6)$ is $11$. However we can put other factorizations of 66 between them so that the maximum distance of two consecutive links is at most 4:

\begin{center}
\begin{tikzpicture}[y=.3cm, x=.3cm,font=\sffamily] 
\draw[very 	thick, black] (0,0) -- (0,10);
\draw (0,10) to[bend right] (10,10)  to[bend right] (20,10) to[bend right] (30,10) to[bend right] (40,10) to[bend right] (50,10);
\draw[very thick, black] (10,10) -- (10,0);
\draw[very thick, black] (20,10) -- (20,0);
\draw[very thick, black] (30,10) -- (30,0);
\draw[very thick, black] (40,10) -- (40,0);
\draw[very thick, black] (50,10) -- (50,0);

\filldraw[fill=black!40,draw=black!80] (0,0) circle (3pt)    node[anchor=north] {$(3,0,0)$};

\filldraw[fill=black!40,draw=black!80] (0,10) circle (3pt)    node[anchor=south] 
{$( 11,0,0 )$};

\filldraw[fill=black!40,draw=black!80] (10,10) circle (3pt)    node[anchor=south] 
{$( 8,2,0 )$};

\filldraw[fill=black!40,draw=black!80] (10,0) circle (3pt)    node[anchor=north] 
{$(0,2,0) | (3,0,0) $};

\filldraw[fill=black!40,draw=black!80] (20,10) circle (3pt)    node[anchor=south] 
{$( 5,4,0 )$};

\filldraw[fill=black!40,draw=black!80] (20,0) circle (3pt)    node[anchor=north] 
{$(0,2,0) | (3,0,0) $};

\filldraw[fill=black!40,draw=black!80] (30,10) circle (3pt)    node[anchor=south] 
{$( 2,6,0 )$};

\filldraw[fill=black!40,draw=black!80] (30,0) circle (3pt)    node[anchor=north] 
{$(0,2,0) | (1,3,0) $};

\filldraw[fill=black!40,draw=black!80] (40,10) circle (3pt)    node[anchor=south] 
{$( 1,3,3 )$};

\filldraw[fill=black!40,draw=black!80] (40,0) circle (3pt)    node[anchor=north] 
{$(0,0,3) | (1,3,0) $};

\filldraw[fill=black!40,draw=black!80] (50,0) circle (3pt)    node[anchor=north] 
{$( 0,0,3 )$};

\filldraw[fill=black!40,draw=black!80] (50,10) circle (3pt)    node[anchor=south] {$( 0,0,6 )$};

\node [below] at (5,10) {$3$};

\node [below] at (15,10) {$3$};

\node [below] at (25,10) {$3$};

\node [below] at (35,10) {$4$};

\node [below] at (45,10) {$4$};

\end{tikzpicture}
\end{center}
In the above picture the factorizations are depicted in the top of a post, and they are linked by a ``catenary'' labeled with the distance between two consecutive sticks. On the bottom we have drawn the factorizations removing the 
common part with the one on the left and that of the right, respectively. We will say that the catenary degree of $66$ in $\langle 6,9,11\rangle$ is at most 4.
\end{example}

Given $s\in S$, $x,y\in \mathsf Z(s)$ and a nonnegative integer $N$, an \emph{$N$-chain} joining $x$ and $y$ is a sequence $x_1,\ldots, x_k\in\mathsf Z(s)$ such that 
\begin{itemize}
\item $x_1=x$, $x_k=y$,
\item for all $i\in\{1,\ldots, k-1\}$, $\mathrm d(x_i,x_{i+1})\le N$.
\end{itemize}

The \emph{catenary degree} of $s$, denoted $\mathsf c(s)$, is the least $N$ such that for any two factorizations $x,y\in\mathsf Z(s)$, there is an $N$-chain joining them. The catenary degree of $S$, $\mathsf c(S)$, is defined as 
\[
\mathsf c(S)=\sup_{s\in S} \mathsf c(S).
\]

\begin{example}
Let us compute the catenary degree of $77\in S = \langle 10,11, 23, 35 \rangle$.  We start with a complete graph with vertices the factorizations of $77$ and edges labeled with the distances between them. Then we remove one edge with maximum distance, and we repeat the process until we find a bridge. The label of that bridge is then the catenary degree of $77$. 

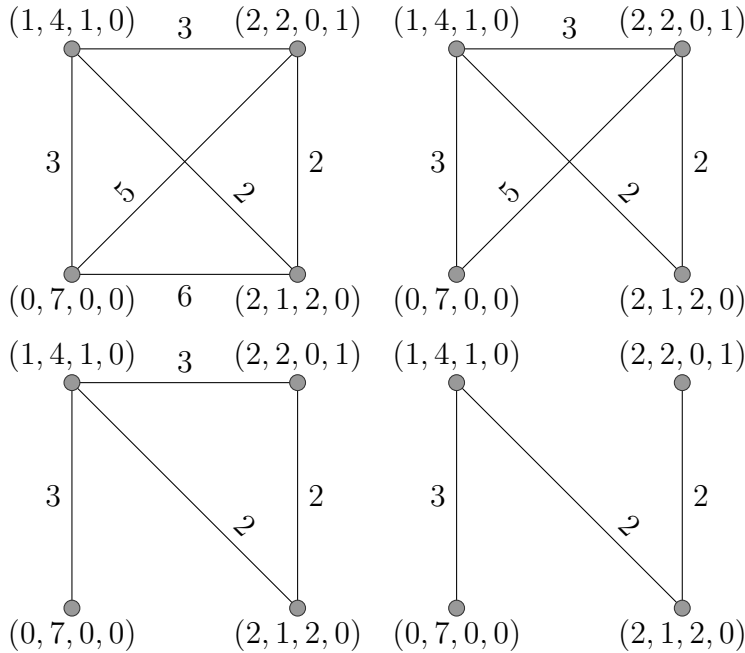
\begin{center}
\begin{tikzpicture}[y=.3cm, x=.3cm,font=\sffamily]
\draw (0,0) -- (10,10);
\draw (0,0) -- (0,10);
\draw (0,0) -- (10,0);
\draw (10,0) -- (10,10);
\draw (10,0) -- (0,10);
\draw (0,10) -- (10,10);

\filldraw[fill=black!40,draw=black!80] (0,0) circle (3pt)    node[anchor=north] {$(0,7,0,0)$};

\filldraw[fill=black!40,draw=black!80] (0,10) circle (3pt)    node[anchor=south] {$( 1, 4, 1, 0 )$};

\filldraw[fill=black!40,draw=black!80] (10,0) circle (3pt)    node[anchor=north] {$( 2, 1, 2, 0 )$};

\filldraw[fill=black!40,draw=black!80] (10,10) circle (3pt)    node[anchor=south] {$(2, 2, 0, 1)$};

\node [above] at (5,10) {$3$};

\node [below] at (5,0) {$6$};

\node [right] at (10,5) {$2$};

\node [left] at (0,5) {$3$};

\node [above, rotate=45] at (3,3) {$5$};

\node [above, rotate=-45] at (7,3) {$2$};

\end{tikzpicture}
\begin{tikzpicture}[y=.3cm, x=.3cm,font=\sffamily]
\draw (0,0) -- (10,10);
\draw (0,0) -- (0,10);
\draw (10,0) -- (10,10);
\draw (10,0) -- (0,10);
\draw (0,10) -- (10,10);

\filldraw[fill=black!40,draw=black!80] (0,0) circle (3pt)    node[anchor=north] {$(0,7,0,0)$};

\filldraw[fill=black!40,draw=black!80] (0,10) circle (3pt)    node[anchor=south] {$( 1, 4, 1, 0 )$};

\filldraw[fill=black!40,draw=black!80] (10,0) circle (3pt)    node[anchor=north] {$( 2, 1, 2, 0 )$};

\filldraw[fill=black!40,draw=black!80] (10,10) circle (3pt)    node[anchor=south] {$(2, 2, 0, 1)$};

\node [above] at (5,10) {$3$};


\node [right] at (10,5) {$2$};

\node [left] at (0,5) {$3$};

\node [above, rotate=45] at (3,3) {$5$};

\node [above, rotate=-45] at (7,3) {$2$};

\end{tikzpicture}

\begin{tikzpicture}[y=.3cm, x=.3cm,font=\sffamily]
\draw (0,0) -- (0,10);
\draw (10,0) -- (10,10);
\draw (10,0) -- (0,10);
\draw (0,10) -- (10,10);

\filldraw[fill=black!40,draw=black!80] (0,0) circle (3pt)    node[anchor=north] {$(0,7,0,0)$};

\filldraw[fill=black!40,draw=black!80] (0,10) circle (3pt)    node[anchor=south] {$( 1, 4, 1, 0 )$};

\filldraw[fill=black!40,draw=black!80] (10,0) circle (3pt)    node[anchor=north] {$( 2, 1, 2, 0 )$};

\filldraw[fill=black!40,draw=black!80] (10,10) circle (3pt)    node[anchor=south] {$(2, 2, 0, 1)$};

\node [above] at (5,10) {$3$};


\node [right] at (10,5) {$2$};

\node [left] at (0,5) {$3$};


\node [above, rotate=-45] at (7,3) {$2$};

\end{tikzpicture}
\begin{tikzpicture}[y=.3cm, x=.3cm,font=\sffamily]
\draw (0,0) -- (0,10);
\draw (10,0) -- (10,10);
\draw (10,0) -- (0,10);

\filldraw[fill=black!40,draw=black!80] (0,0) circle (3pt)    node[anchor=north] {$(0,7,0,0)$};

\filldraw[fill=black!40,draw=black!80] (0,10) circle (3pt)    node[anchor=south] {$( 1, 4, 1, 0 )$};

\filldraw[fill=black!40,draw=black!80] (10,0) circle (3pt)    node[anchor=north] {$( 2, 1, 2, 0 )$};

\filldraw[fill=black!40,draw=black!80] (10,10) circle (3pt)    node[anchor=south] {$(2, 2, 0, 1)$};



\node [right] at (10,5) {$2$};

\node [left] at (0,5) {$3$};


\node [above, rotate=-45] at (7,3) {$2$};
\end{tikzpicture}
\end{center}
Thus the catenary degree of $77$ is $3$.
\end{example}

If one looks at Proposition \ref{congruence-generated} and Example \ref{ejemplo-cat-estacas}, one sees some interconnection between the transitivity and the way we can move from one factorization to another to minimize distances. This idea is exploited in the following result.
 
\begin{teorema}\label{cat-in-betti}
Let $S$ be a numerical semigroup. Then
\[
\mathsf c(S)=\max\{ \mathsf c(b)\mid b\in \mathrm{Betti}(S)\}.
\]
\end{teorema}
\begin{proof}
Set $c=\max_{b\in\mathrm{Betti}(S)}\mathsf c(b)$. Clearly $c\le \mathsf c(S)$. Let us prove the other inequality. Take $s\in S$ and $x,y\in\mathsf Z(s)$. Let $\sigma$ be a minimal presentation of $S$. Then by Proposition \ref{congruence-generated}, there exists a sequence $x_1,\ldots, x_k$ such that $x_1=x$, $x_k=y$, and for every $i$ there exists $c_i\in \NN^p$ (with $p$ the embedding dimension of $S$) such that $(x_i,x_{i+1})=(a_i+c_i,b_i+c_i)$ for some $(a_i,b_i)$ such that either $(a_i,b_i)\in \sigma$ or $(b_i,a_i)\in \sigma$. According to Theorem \ref{carac-presentation}, $a_i,b_i$ are factorizations of a Betti element of $S$. By using the definition of catenary degree, there is a $c$-chain joining $a_i$ and $b_i$ (also $b_i$ and $a_i$). If we add $c_i$ to all the elements of this sequence, we have a $c$-chain joining $x_i$ and $x_{i+1}$ (distance is preserved under translations). By concatenating all these $c$-chains for $i\in\{1,\ldots,k-1\}$ we obtain a $c$-chain joining $x$ and $y$. And this proves that $\mathsf c(S)\le c$, and the equality follows.
\end{proof}

\begin{example}\label{ej-10111723-2}
With the package \texttt{numericalsgps} the catenary degree of an element and of the whole semigroup can be obtained as follows.
\begin{verbatim}
gap> s:=NumericalSemigroup(10,11,17,23);;
gap> FactorizationsElementWRTNumericalSemigroup(60,s);
[ [ 6, 0, 0, 0 ], [ 1, 3, 1, 0 ], [ 2, 0, 1, 1 ] ]
gap> CatenaryDegreeOfElementInNumericalSemigroup(60,s);
4
gap> CatenaryDegreeOfNumericalSemigroup(s);
6
\end{verbatim}
\end{example}

\subsection{How far is an irreducible from being prime}

As the title suggests, the last invariant we are going to present measures how far is an irreducible from being prime. Recall that a prime element is an element such that if it divides a product, then it divides one of the factors. Numerical semigroups are monoids under addition, and thus the concept of divisibility must be defined accordingly. 

Given $s, s'\in S$, recall that we write $s\le_S s'$ if $s'-s\in S$. We will say that $s$ \emph{divides} $s'$. Observe that $s$ divides $s'$ if and only if $s'$ belongs to the ideal $s+S=\{s+x~|~ x\in S\}$ of $S$. If $s\le_S s'$, then $t^s$ divides $t^{s'}$ in the semigroup ring $K[S]$, in the ``multiplicative'' sense.

The \emph{$\omega$-primality} of $s$ in $S$, denoted $\omega(S,s)$, is the least positive integer $N$ such that whenever $s$ divides $a_1+\cdots+a_n$ for some $a_1,\ldots, a_n\in S$, then $s$ divides $a_{i_1}+\cdots+a_{i_N}$ for some  $\{i_1,\ldots,i_N\}\subseteq \{1,\ldots, n\}$.  

Observe that an irreducible element in $S$ (minimal generator) is prime if its $\omega$-primality is 1. It is easy to observe that a numerical semigroup has no primes.

In the above definition, we can restrict the search to sums of the form $a_1+\cdots+a_n$, with $a_1,\ldots,a_n$ minimal generators of $S$ as the following lemma shows.

\begin{lema}\label{w-atomic}
Let $S$ be numerical semigroup and $s\in S$. Then $\omega (S,s)$ is the smallest $N \in \NN \cup \{\infty\}$ with the following property:
\begin{itemize}
\item[] For all $n \in \NN$ and $a_1, \ldots, a_n$ minimal generators of $S$, if  $b$ divides  $a_1 + \cdots + a_n $, then there exists a subset $\Omega \subset [1,n] $ with cardinality less than or equal to $N$ such that 
\[ b \leq_S \, \sum_{i \in \Omega} a_i. \]
\end{itemize}
\end{lema}
\begin{proof}
Let $\omega' (S, s)$ denote the smallest integer $N \in \NN_0 \cup \{\infty\}$ satisfying the property mentioned in the lemma. We show that $\omega (S, s) = \omega' (S, s)$. By definition, we have $\omega' (S, s) \le \omega (S, s)$. 

In order to show that $\omega (S, s) \le \omega' (S, s)$, let $n \in \NN$ and $a_1, \ldots, a_n \in S$ with $s \leq_S a_1 + \cdots + a_n$. For every $i \in [1, n]$ we pick a factorization $a_i = u_{i,1} + \cdots + u_{i, k_i}$ with $k_i \in \NN$ and $u_{i_1}, \ldots , u_{i, k_i}$ minimal generators of $S$. Then there is a subset $I \in [1,n]$ and, for every $i \in I$, a subset $\emptyset \ne \Lambda_i \subset [1, k_i]$ such that \[ \#I \le \sum_{i \in I} \#\Lambda_i \le \omega' (S, s) \quad \text{and} \quad s \leq_S \sum_{i \in I} \sum_{\nu \in \Lambda_i} u_{i, \nu}, \] and hence $s \leq_S \sum_{i \in I} a_i$.
\end{proof}

In order to compute the $\omega$-primality  of an element $s$ in a numerical semigroup $S$, one has to look at the minimal factorizations (with respect to the usual partial ordering) of the elements in the ideal $s+S$; this is proved in the next result.

\begin{proposicion}\label{calc-omega}
Let $S$ be a numerical semigroup minimally generated by $\{n_1,\ldots,n_p\}$. Let $s\in S$. Then
\[ \omega(S,s)= \max\left\{\lvert m\rvert  \mid  m\in \mathrm{Minimals}_{\leq}(\mathsf Z(s+S))\right\}.\]
\end{proposicion}
\begin{proof}
Notice that by Dickson's lemma, the set $\mathrm{Minimals}_{\leq}(\mathsf Z(s+S))$ has finitely many elements, and thus $N=\max\left\{\lvert m\rvert \mid m\in \mathrm{Minimals}_{\leq}(\mathsf Z(s+S))\right\}$ is a nonnegative integer.

Choose $x=(x_1,\ldots,x_p)\in \mathrm{Minimals}_{\leq}(\mathsf Z(s+S))$ such that $\lvert x\rvert=N$. Since $x\in \mathsf Z(s+S)$, $s$ divides $s'=x_1n_1+\cdots +x_pn_p$. Assume that $s$ divides $s''=y_1n_1+\cdots+y_pn_p$ with $(y_1,\ldots,y_p)<(x_1,\ldots,x_p)$ (that is, $s$ divides a proper subset of summands of $s'$) . Then $s''\in s+S$, and $(y_1,\ldots,y_p)\in \mathsf Z(s+S)$, contradicting the minimality of $x$. This proves that $\omega(S,s)\geq N$.

Now assume that $s$ divides $x_1n_1+\cdots+x_pn_p$ for some $x=(x_1,\ldots,x_p)\in \NN^p$. Then $x\in \mathsf Z(s+S)$, and thus there exists $m=(m_1,\ldots,m_p)\in \mathrm{Minimals}_{\leq}(\mathsf Z(s+S))$ with $m\leq x$. By definition, $m_1n_1+\cdots+m_pn_p\in s+S$, and $\lvert m\rvert \leq N$. This proves in view of Lemma \ref{w-atomic} that $N\leq \omega(S,s)$.
\end{proof}

For $S$ a numerical semigroup minimally generated by $\{n_1,\ldots, n_p\}$, the $\omega$-primality of $S$ is defines as 
\[ \omega(S)=\max\{ \omega(S,n_i)\mid i\in\{1,\ldots, p\}\}.\]

We are going to relate Delta sets with catenary degree and $\omega$-primality. To this end we need the following technical lemma. 

For $b=(b_1,\ldots, b_p)\in \NN^p$, define $\mathrm{Supp}(b)=\{i\in\{1,\ldots,p\}\mid b_i\neq 0\}$.

\begin{lema}\label{presentacion-a-minimales}
Let $S$ be a numerical semigroups minimally generated by $\{ n_1,\ldots, n_p\}$, and let $n\in \mathrm{Betti}(S)$.  Let $a,b\in \mathsf Z(n)$ in different $\mathcal R$-classes.  For every $i\in \mathrm{Supp}(b)$ it follows that $a\in \mathrm{Minimals}_\le \mathsf Z(n_i+S)$.
\end{lema}
\begin{proof}
Assume to the contrary that there exists $c\in \mathsf Z(n_i+S)$ and $x\in \mathbb N^k\setminus \{0\}$ such that $c+x=a$. From $c<a$, $a\cdot b=0$ and $i\in \mathrm{Supp}(b)$, we deduce that $i\not\in \mathrm{Supp}(c)$. As $c\in \mathsf Z(n_i+S)$, there exists $d\in \mathsf Z(n_i+S)$ with $i\in \mathrm{Supp}(d)$ and $\varphi(c)=\varphi(d)$. Hence $\varphi(d+x)=\varphi(c+x)=\varphi(a)$. Moreover $(d+x)\cdot(c+x)=(d+x)\cdot a\not=0$, and $(d+x)\cdot b\not=0$, which leads to $a\mathcal R b$, a contradiction.
\end{proof}

\begin{teorema}\label{delta-c-w}
Let $S$ be a numerical semigroup. Then 
\[
\max\Delta(S)+2\le \mathsf c(S)\le \omega(S).
\]
\end{teorema}
\begin{proof}
Assume that $d\in \Delta(S)$. Then there exists $s\in S$ and $x,y\in \mathsf Z(s)$ such that $\lvert x\rvert < \lvert y\rvert$, $d=\lvert y\rvert -\lvert x\rvert$ and there is no $z\in \mathsf Z(s)$ with $\lvert x\rvert <\lvert z\rvert <\lvert y\rvert$. From the definition of $\mathsf c(S)$, there is a $\mathsf c(S)$-chain $z_1,\ldots, z_k$ joining $x$ and $y$. As in the proof of Theorem \ref{max-delta}, we deduce that there exists $i$ such that $\lvert z_i\rvert <\lvert x\rvert <\lvert y\rvert <\lvert z_{i+1}\rvert$. Then $2+d =2+\lvert y\rvert -\lvert x\rvert \le 2 +\lvert z_{i+2}\rvert -\lvert z_i\rvert$, and  by Lemma \ref{max-len-dist}, $2 +\lvert z_{i+2}\rvert -\lvert z_i\rvert\le \mathrm d(z_i,z_{i+1})$. The definition of $\mathsf c(S)$-chain implies that $\mathrm d(z_i,z_{i+1})\le \mathsf c(S)$. Hence $2+d\le \mathsf c(S)$, and consequently $\max \Delta(S)+2\le \mathsf c(S)$. 

Let $\sigma$ be a minimal presentation of $\ker\varphi$. For every $(a,b)\in \sigma$, there exists $n_i$ and $n_j$ minimal generators such that $a\in \mathrm{Minimals}_\le \mathsf Z(n_i+S)$ and $b\in \mathrm{Minimals}_\le \mathsf Z(n_j+S)$ (Lemma \ref{presentacion-a-minimales}). From the definition of $\omega(S)$, both $\lvert a\rvert$ and $\lvert b\rvert$ are smaller than or equal to $\omega(S)$. Set $c=\max\{\max\{\lvert a\rvert, \lvert b\rvert\}\mid (a,b)\in \sigma\}$. Then $c\le \omega(S)$. Now we prove that $\mathsf c(S)\le c$. Let $s\in S$ and $x,y\in \mathsf Z(s)$. Then $\varphi(x)=\varphi(y)$ and as $\sigma$ is a presentation, by Proposition \ref{congruence-generated}, there exists a sequence $x_1,\ldots, x_k\in \NN^p$ ($p=\mathrm e(S)$) such that $x_1=x$, $x_k=y$ and for every $i$ there exists $a_i,b_i,c_i\in \NN^p$ such that $(x_i,x_{i+1})=(a_i+c_i,b_i+c_i)$, with either $(a_i,b_i)\in \sigma$ or $(b_i,a_i)\in \sigma$. Notice that $\mathrm d(x_i,x_{i+1})=\mathrm d(a_
 i,b_i)=\max\{\lvert a_i\rvert, \lvert b_i\rvert\}$ ($a_i$ and $b_i$ are in different $\mathcal R$-classes and thus $a_i\cdot b_i=0$, or equivalently, $a_i\wedge b_i=0$). Hence $\mathrm d(x_i,x_{i+1})\le c$, and consequently $x_1,\ldots,x_k$ is a $c$-chain joining $x$ and $y$. This implies that $\mathsf c(S)\le c$, and we are done.
\end{proof}

\begin{example}
Let us go back to $S=\langle 10,11,17,23\rangle$. From Example \ref{ej-10111713} and Theorem \ref{max-delta}, we know that $\max \Delta(S)=3$. 
\begin{verbatim}
gap> OmegaPrimalityOfNumericalSemigroup(s);
6
\end{verbatim}
From Theorem \ref{delta-c-w}, we deduce that $\mathsf c(S)\in \{5,6\}$. Recall that by Example \ref{ej-10111723-2}, we know that $\mathsf c(S)=6$.
\end{example}

The are many other nonunique factorization invariants that can be defined on any numerical semigroup. It was our intention just to show some of them and the last theorem that relates these invariants coming from lengths, distances and primality (respectively), and at the same time show how minimal presentations can be used to study them. The reader interested in this topic is referred to \cite{GHKb}.

\end{document}